\documentclass{style}

\makeatletter
\@namedef{subjclassname@2020}{%
  \textup{2020} Mathematics Subject Classification}
\makeatother

\renewcommand{\labelenumi}{\textup{(\alph{enumi})}}

\allowdisplaybreaks

\DeclareMathOperator{\rank}{rank}

\DeclareMathOperator{\End}{End}

\DeclareMathOperator{\Irr}{Irr}

\newcommand{\C}{\mathcal{C}}

\newcommand{\D}{\mathcal{D}}

\newcommand{\G}{\mathcal{G}}

\newcommand{\Z}{\mathcal{Z}}

\newcommand{\M}{\mathcal{M}}

\newcommand{\cO}{\mathcal{O}}

\newcommand{\bC}{\mathbb{C}}
\newcommand{\bZ}{\mathbb{Z}}

\newcommand{\fgl}{\mathfrak{gl}}

\newcommand{\kk}{\Bbbk}

\newcommand{\ev}{\mathrm{ev}}
\newcommand{\coev}{\mathrm{coev}}

\newcommand{\id}{\mathrm{id}}
\newcommand{\Rep}{\mathrm{Rep}}

\newcommand{\Hom}{\mathrm{Hom}}

\newcommand{\unit}{\mathbbm{1}}

\newcommand{\Vect}{\mathbf{Vec}}
\newcommand{\vect}{\mathbf{vec}}
\newcommand{\iHom}{\underline{\sf Hom}}

\newcommand{\Inv}{{\sf Inv}}

\newcommand{\loc}{\mathrm{loc}}

\newcommand{\rev}{\mathrm{rev}}
\newcommand{\op}{\mathrm{op}}

\definecolor{darkgreen}{RGB}{55,138,0}
\definecolor{burntorange}{RGB}{180,85,0}
\definecolor{navyblue}{RGB}{18,40,180}
\definecolor{cyan(process)}{rgb}{0.0, 0.6, 1.0}

\newcommand{\fp}{\mathrm{fp}}
\newcommand{\fg}{\mathrm{fg}}

\newcommand{\fl}{\mathrm{fl}}

\newcommand{\Ext}{\mathrm{Ext}}
\newcommand{\Ker}{\mathrm{Ker}}
\newcommand{\Img}{\mathrm{Im}}
\newcommand{\Cok}{\mathrm{Coker}}

\newcommand{\Stab}{\mathrm{Stab}}

\newcommand{\iC}{\widehat{\C}}

\makeatletter
\newcommand{\namedlabel}[2]{%
  \phantomsection
  \def\@currentlabel{#1}%
  \label{#2}%
}
\makeatother

\begin{document}

\begin{abstract}
We give criteria for when finitely generated local modules over a commutative algebra $A$ in the ind-completion $\widehat{\mathcal{C}}$ of a braided tensor category $\mathcal{C}$ inherit the structure of a (rigid, braided, ribbon) tensor category. We then apply this to simple current algebras $A = \bigoplus_{g \in \Gamma} E_g$, where $\Gamma$ is a subgroup of invertible objects in $\mathcal{C}$. Using a description of simple $A$-modules, we verify the required hypotheses for this class of algebras and deduce rigidity, braided, ribbon, and non-degeneracy properties for their finitely generated local modules. As applications, we construct examples of ribbon tensor categories from quantum supergroup categories for unrolled $\mathfrak{gl}(1|1)$.
\end{abstract}

\title{Ribbon categories from ind-exact algebras: simple current case}

\maketitle


\section{Introduction}
Simple currents (i.e.\ invertible objects) play a central role in constructions and classification problems in braided tensor categories arising in representation theory and conformal field theory. Given a braided tensor category $\C$ and a subgroup $\Gamma$ of its group of isomorphism classes of invertible objects, the object $A = \bigoplus_{g \in \Gamma} E_g$ (where $E_g\in\C$ denotes an invertible object corresponding to $g\in\Gamma$) has a structure of an algebra in the ind-completion $\iC$ under an appropriate cocycle condition. We call an algebra of this type a \emph{simple current algebra} (see Definition~\ref{def:simple-current-alg} for the precise definition). The goal of this paper is to give workable criteria ensuring that suitable subcategories of modules over a commutative algebra in $\iC$ inherit the structure of ribbon tensor category from $\C$ and to establish that commutative simple current algebras satisfy these criteria under mild assumptions.

\subsection{Main results}
Throughout this paper, we work over an algebraically closed field $\kk$.
The paper has two parts.
First, we set up a general framework for commutative algebras $A\in\iC$, where $\C$ is a braided tensor category. A right $A$-module is said to be \emph{finitely generated} if it is a quotient of the free $A$-module $X \otimes A$ for some object $X \in \C$. Based on this definition, one can define the category $\fg\text{-}\iC_A$ of finitely generated $A$-modules and its subcategory $\fg\text{-}\iC_A^{\loc}$ of finitely generated local modules.

\begin{namedtheorem}[A]\namedlabel{A}{thm:main-A}
Let $A\in\iC$ be a haploid commutative algebra.
\begin{enumerate}
\item If $A$ is Artinian, then $\fg\text{-}\iC_A$ and $\fg\text{-}\iC_A^{\loc}$ are abelian categories of finite length.
\item If moreover $A$ is ind-exact, then $\fg\text{-}\iC_A$ is tensor and $\fg\text{-}\iC_A^{\loc}$ is braided tensor.
\item If in addition $\C$ is ribbon and $A$ is Frobenius with $\theta_A=\id_A$, then $\fg\text{-}\iC_A^{\loc}$ is ribbon.
\end{enumerate}
\end{namedtheorem}

See Definitions \ref{def:artinian-alg} and \ref{def:ind-exact-alg} for the
definitions of Artinian and ind-exact algebras, respectively. The term
Frobenius algebra is used in an extended sense as defined in
Definition~\ref{def:Frobenius-alg-in-Ind}, and in particular the algebra $A$
in (c) may not be a rigid object in $\iC$. Our discussion of Artinian and ind-exact algebras builds upon \cite{coulembier2023incompressible} and our prior work \cite{shimizu2024commutative}. Part~(c) is the main new contribution in this theorem: we introduce `infinite' Frobenius algebras and use this property to obtain the ribbon structure of $\fg\text{-}\iC_A^{\loc}$.

Second, we discuss a practical way of checking ind-exactness of algebras. Etingof and Penneys proved that in an abelian braided monoidal category with right exact tensor product, if all the simple objects are rigid, then all finite length objects are rigid \cite[Lemma~4.2]{etingof2024rigidity}. In Appendix~\ref{app:EP-lemma}, we prove a non-braided generalization of this result (Theorem~\ref{thm:EP-rigidity-main}). Using this, we obtain the following result (Theorem~\ref{thm:ind-exact}):

\begin{namedtheorem}[B]\namedlabel{B}{thm:main-B}
Let $\C$ be a tensor category and $A$ an Artinian central commutative algebra in $\iC$. Then, $A$ is ind-exact if and only if every simple $A$-module is rigid.
\end{namedtheorem}

Third, and this is the main contribution of the paper: we prove that \emph{simple current algebras} (Definition \ref{def:simple-current-alg}) satisfy these hypotheses under explicit and checkable conditions, and we derive concrete consequences for their module categories. For a simple current algebra $A=\bigoplus_{g\in \Gamma}E_g$, we first classify simple $A$-modules. We then verify the hypotheses of Theorem~\ref{thm:main-A} for this class: in particular, $A$ is always Artinian (Proposition~\ref{prop:A-artinian}) and Frobenius (Theorem~\ref{thm:simple-current-ext-is-Fb}). If $A$ is central commutative and char$(\kk)=0$, we use the explicit description of simple $A$-modules to show that they are rigid. By Theorem~\ref{thm:main-B}, this establishes that $A$ is ind-exact. Section~\ref{sec:simple-current-algebras} also gives a general categorical non-degeneracy criterion for the category $\fg\text{-}\iC_A^\loc$.

\begin{namedtheorem}[C]\namedlabel{C}{thm:main-C}
Suppose that \textup{char}$(\kk)=0$. Let $\Gamma$ be a subgroup of the group of isomorphism classes of invertible objects of $\C$ such that $c_{E_g,E_g}=\id_{E_g\otimes E_g}$ for all $g\in \Gamma$. Then the object $A = \bigoplus_{g \in \Gamma} E_g$ has a unique (up to isomorphism) structure of commutative simple current algebra. For the commutative simple current algebra $A$, we have:
\begin{enumerate}
\item $\fg\text{-}\iC_A$ is a tensor category and $\fg\text{-}\iC_A^{\loc}$ is a braided tensor category.
\item If $\C$ is ribbon and $\theta_{E_g}=\id_{E_g}$ for all $g\in \Gamma$, then $\fg\text{-}\iC_A^{\loc}$ is ribbon.
\end{enumerate}
Now suppose that $\C$ is Frobenius (i.e., it has enough projectives). Then,
\begin{enumerate}
\item[\textup{(c)}] $\fg\text{-}\iC_A$ and $\fg\text{-}\iC_A^{\loc}$ are Frobenius.
\item[\textup{(d)}] $\fg\text{-}\iC_A$ is finite if and only if $\Irr(\C)/\Gamma$ is finite, and $\fg\text{-}\iC_A^{\loc}$ is braided finite if and only if $\Irr(\C_\Gamma)/\Gamma$ is finite. Here, $\C_{\Gamma}$ is the full subcategory
\[\C_\Gamma:=\{X\in\C| c_{E_g,X}c_{X,E_g}=\id_{X\otimes E_g}\ \forall\ g\in \Gamma\} . \]
\item[\textup{(e)}] If the simple objects in $\C_\Gamma$ that trivially double braid with all other simples in $\C_\Gamma$ are precisely $E_g$ for $g\in \Gamma$, then $\fg\text{-}\iC^{\loc}_A$ is non-degenerate.
\end{enumerate}
\end{namedtheorem}

One may call a ribbon tensor category with trivial M\"uger center a \emph{modular tensor category}. Then Theorem~\ref{thm:main-C} gives sufficient conditions that a commutative simple current algebra $A$ must satisfy to ensure that the category $\fg\text{-}\iC_A^\loc$ is modular.

\subsection{VOA motivation and applications}
The study of simple current extensions originated in conformal field theory with the work of Schellekens and Yankielowicz \cite{schellekens1989modular} on modular invariants generated by integer-spin simple currents. In vertex operator algebra theory, Dong, Li, and Mason \cite{dong1996simple} gave an early treatment of extensions by simple currents. In the rational tensor-categorical setting, the algebra-object and local-module viewpoint was developed by Kirillov--Ostrik \cite{kirillov2002q} and by Fuchs, Runkel, and Schweigert \cite{fuchs2004tft}. Later, Huang, Kirillov, and Lepowsky \cite{huang2015braided} showed that, under suitable hypotheses, extensions of a VOA $V$ are equivalent to commutative associative algebra objects with trivial twist in the braided tensor category of $V$-modules. In this framework, the extended VOA $W$ corresponds to an algebra object $A$ in $\C$ (or $\iC$), and the category of $W$-modules is described by the category of local $A$-modules; see \cite{huang2015braided,creutzig2024tensor}. If the underlying object of the algebra is a direct sum of invertible objects, it is called a simple current extension (this is why we call an algebra of the form $A = \bigoplus_{g \in \Gamma} E_g$ a simple current algebra).

Our motivation comes from studying simple current extensions of logarithmic VOAs. The representation categories of such VOAs are typically not semisimple and not finite, and moreover, the simple current extensions are typically not objects of the original category, but rather lie in the ind-completion.
There are many papers which discuss rigid, ribbon structures on categories of modules over simple current extensions in specific examples; see, for example
\cite{auger2018infinite,creutzig2022direct,creutzig2022uprolling,creutzig2022tensor}.

Our results can be viewed as a rigidity inheritance statement for simple current extensions of vertex operator
algebras.
Let $V$ be a VOA such that a suitable category $\C$ of $V$-modules carries a braided tensor category structure.
Assume moreover that $\C$ is rigid, and let $V\subset W$ be a simple current extension.
In the tensor-categorical formulation, the extension corresponds to a commutative algebra object
$A\in \iC$, and $W$-modules are described by local $A$-modules in $\iC$
\cite{huang2015braided,creutzig2024tensor}.
Under the hypotheses of Theorem~\ref{thm:main-A}, we show that the category $\fg\text{-}\iC_A^{\loc}$ is rigid.
Thus $W$ also admits a rigid braided tensor category of representations.
This is useful because rigidity is often one of the hardest properties to establish in vertex tensor category
theory.

\subsection{Examples}
We also obtain examples of (braided, ribbon) tensor categories by applying our criteria to commutative simple current algebras in two families of non-semisimple ribbon categories: the categories of integral weight modules for $U^H_q(\mathfrak{g})$ \cite{rupert2022categories} and $U^E_q(\fgl(1|1))$ \cite{geer2025three}. In these settings, one can sum invertibles over suitable subgroups to obtain commutative simple current algebras, and our criteria produce new categories of local modules. In the unrolled quantum group setting, this specializes to the explicit lattice criterion of Creutzig--Rupert \cite{creutzig2022uprolling}.

Note that the weight module categories of $U^H_q(\mathfrak{g})$ and $U^E_q(\fgl(1|1))$ are relative modular categories in the sense of \cite{costantino2014quantum}. It would be interesting to examine whether the relative modular structure descends to the local module category obtained from our results.

\subsection{Organization}
The paper is organized as follows. Section~\ref{sec:background} recalls basic terminology on tensor categories and centers. Section~\ref{sec:framework} develops a general framework for constructing (braided, ribbon) tensor categories from commutative algebras in ind-completions; the main criteria are summarized in Theorem~\ref{thm:main-criteria}. Section~\ref{sec:simple-current-algebras} applies this framework to simple current algebras. Section~\ref{sec:examples} records concrete examples from the unrolled quantum groups of $\mathfrak{gl}(1|1)$. Appendix~\ref{app:EP-lemma} contains a non-braided variant of the Etingof--Penneys \cite{etingof2024rigidity} lemma used for rigidity arguments.

\subsection{Acknowledgements}
The authors thank RIMS for hospitality during the completion of this work. We thank Pavel Etingof and Dave Penneys for a helpful email exchange regarding \cite[Lemma~4.2]{etingof2024rigidity}.

HY is partially supported by a start-up grant from the University of Alberta and an NSERC Discovery Grant. KS is supported by JSPS KAKENHI Grant Number JP24K06676.


\section{Background}\label{sec:background}

\subsection{Notations}
Throughout the paper, $\kk$ will denote an algebraically closed field. We will write $\vect$ to denote the category of finite-dimensional $\kk$-vector spaces and $\Vect$ to denote the category of all $\kk$-vector spaces.

Given a category $\C$, we write $\C^{\op}$ to denote the opposite category of $\C$. 
Unless otherwise noted, the monoidal product and the unit object of a monoidal category $\C$ are denoted by $\otimes : \C \times \C \to \C$ and $\unit \in \C$, respectively.
$\C^{\mathrm{rev}}$ will denote the category $\C$ with the opposite monoidal product. For a braided category $\C$ with braiding $c$, we use the notation $\overline{\C}$ to denote $\C^{\mathrm{rev}}$ with reversed braiding $\overline{c}_{X,Y} = c_{Y,X}^{-1}$.

Regarding the terminology related to duality, we follow the conventions used in \cite{etingof2016tensor}.
An object of a monoidal category $\C$ is {\em left rigid} and {\em right rigid} if it has a left and a right dual object, respectively. A {\em rigid} object is a left and right rigid object. We say that $\C$ is {\em rigid} if every object of $\C$ is rigid. Given a left rigid object $X \in \C$, we denote by
\begin{equation*}
  (X^*,
  \ \ev_X : X^* \otimes X \to \unit,
  \ \coev_X : \unit \to X \otimes X^*)
\end{equation*}
the left dual object of $X$. We call $\ev_X$ and $\coev_X$ the {\em evaluation} and the {\em coevaluation} for $X$, respectively. An (object-part of) a right dual object of $X$ is denoted by ${}^*\!X$.


\subsection{Tensor categories}
A \emph{locally finite abelian category} is an essentially small $\kk$-linear abelian category where every object is of finite length and every Hom space is finite-dimensional. A \emph{tensor category} is a locally finite abelian category endowed with a structure of a rigid monoidal category such that the monoidal product is $\kk$-bilinear and the unit is simple \cite{etingof2016tensor}. Let $\C$ be a tensor category. We record the following result for later use.

\begin{lemma}[{\cite[Proposition~1.1]{deligne2002categories}}]
\label{lem:hom-finite}
Suppose that $\C$ is an abelian $\kk$-linear monoidal category with $\End_\C(\unit)\cong\kk$ whose objects have finite length. If $\C$ is rigid, then it is Hom-finite.
\end{lemma}

A tensor category $\C$ is said to be \emph{Frobenius} if, in addition, every simple object of $\C$ has an injective hull \cite{andruskiewitsch2015two}. Several equivalent conditions for $\C$ to be Frobenius were given in \cite{shibata2023nakayama}.


\subsection{Braided tensor categories}

\subsubsection{M\"uger center and nondegeneracy}
Given a braided tensor category $\C$ and a monoidal subcategory $\D$, the \emph{centralizer} of $\D$ in $\C$ is defined as:
\[ \Z_2(\D\subset \C) = \{ X\in\C \mid c_{Y,X}\circ c_{X,Y}=\id_{X\otimes Y},\, \forall \; Y\in \D \} .\]
The \emph{M\"uger center} of $\C$ is $\C':= \Z_2(\C \subset \C)$. We call $\C$ \textit{non-degenerate} if $\C$ is a tensor category and its M\"uger center is equivalent to $\vect$. 

We will use the following criteria for nondegeneracy.

\begin{lemma}\label{lem:non-deg-Muger-center}
Suppose that $\mathrm{char}(\kk) = 0$. Then a braided Frobenius tensor category $\C$ is non-degenerate if and only if the only simple object of the M\"uger center of $\C$ is $\unit$ up to isomorphism.
\end{lemma}
\begin{proof}
Since $\C$ is Frobenius, the projective cover of $\unit$ exists in $\C$.
To show the `if' part, we note that $\mathrm{Ext}^1_{\C}(\unit, \unit) = 0$ under our assumption (the proof in the finite case \cite[Theorem~4.4.1]{etingof2016tensor} requires that a projective cover of $\unit$ exists in $\C$ but does not use the finiteness of $\C$). If $\unit$ is the only simple object of $\C'$ up to isomorphism, every object of $\C'$ is a finite direct sum of $\unit$ since $\mathrm{Ext}^1_{\C}(\unit, \unit) = 0$, and therefore $\C'$ is equivalent to $\vect$. The converse is trivial. 
\end{proof}


\subsubsection{Ribbon tensor categories}
A \emph{ribbon tensor category} is a braided tensor category $\C$ equipped with a \emph{twist}, i.e.\ a natural automorphism $\theta = \{\theta_X : X \xrightarrow{\sim} X\}_{X \in \C}$ satisfying
\begin{equation}\label{eq:ribbon-axioms}
  \theta_{X \otimes Y} = c_{Y,X} \circ c_{X,Y} \circ (\theta_X \otimes \theta_Y), \qquad \theta_{X^*} = (\theta_X)^*
\end{equation}
for all $X, Y \in \C$.
One can deduce $\theta_\unit = \id_\unit$ from the former equation.


\subsection{Modules in a closed monoidal category}
Given algebras $A$ and $B$ in a monoidal category $\C$, we denote by ${}_A\C$, $\C_B$ and ${}_A\C_B$, the categories of left $A$-modules, right $B$-modules and $A$-$B$-bimodules in $\C$, respectively.

A monoidal category $\C$ is said to be \emph{closed} if for every $X \in \C$ the endofunctors $X \otimes -$ and $- \otimes X$ on $\C$ have right adjoints. Let $\C$ be a closed monoidal category. Given $X \in \C$, we denote by $[X, -]$ a right adjoint of $- \otimes X$. The assignment $(X, Y) \mapsto [X, Y]$ extends to a functor $\C^{\op} \times \C \to \C$, which we call the \emph{right internal Hom functor}. By definition, there is a natural isomorphism $\Hom_{\C}(W \otimes X, Y) \cong \Hom_{\C}(W, [X, Y])$ for $W, X, Y \in \C$. The \emph{left internal Hom functor} is defined in a similar manner. Since a functor admitting a right adjoint preserves colimits, the endofunctors $X \otimes (-)$ and $(-) \otimes X$ preserve colimits.

We assume that $\C$ admits equalizers and coequalizers. For a right $A$-module $M$ and a left $A$-module $N$ in $\C$, their tensor product $M \otimes_A N$ is defined to be the coequalizer of $a^r_M \otimes \id_N$ and $\id_M \otimes a^l_N$, where $a^r_M : M \otimes A \to M$ and $a^l_N : A \otimes N \to N$ are the right and the left action of $A$ on $M$ and $N$, respectively. Since $\C$ is closed, one can verify that the category ${}_A\C_A$ of $A$-bimodules is a monoidal category with monoidal product $\otimes_A$ and unit $A$. Moreover, ${}_A\C_A$ is a closed monoidal category. A construction of the internal Hom functor is found in \cite[\S3.2]{shimizu2024commutative}.

\subsubsection{Modules over a central algebra}
Let $\C$ be a closed monoidal category admitting equalizers and coequalizers. We call an algebra $A$ in $\C$ \emph{central commutative} if it is equipped with a half-braiding $\sigma$ such that $(A,\sigma)$ is a commutative algebra in the Drinfeld center $\Z(\C)$. In this case, a right $A$-module can be viewed as a left $A$-module by the action given by $\sigma$, and the category $\C_A$ is viewed as a full subcategory of ${}_A\C_A$. One can verify that $\C_A$ is closed under $\otimes_A$, and therefore $\C_A$ is a monoidal category. $\C_A$ is also a closed monoidal category by the same internal Hom functor as ${}_A\C_A$ \cite[\S3.3]{shimizu2024commutative}.

\subsubsection{Local modules}
Assume that the closed monoidal category $\C$ is braided and $A$ is a commutative algebra in $\C$. Then, as noted above, the category $\C_A$ is monoidal with the relative tensor product $\otimes_A$. However, it is not braided in general. The full subcategory of $\C_A$ consisting of \emph{local} modules, that is, those $A$-modules $(M,\rho)$ for which the action $\rho:M\otimes A\to M$ satisfies $\rho \circ c_{A,M} \circ c_{M,A} = \rho$, is braided with the braiding inherited from $\C$ \cite{pareigis1995braiding}.
When $\C$ admits equalizers, the category $\C_A^{\loc}$ is in fact a closed monoidal category by the same internal Hom functor as ${}_A\C_A$ \cite[\S3.3]{shimizu2024commutative}.

 
\subsection{Ind-completions and their subcategories}

\subsubsection{Ind-completions}
Let $\C$ be a tensor category. The ind-completion $\iC$ of $\C$ is the completion of $\C$ under filtered colimits \cite{kashiwara2006catsheaves}. It is a Grothendieck abelian category, and the inclusion $\C\hookrightarrow \iC$ is exact and fully faithful. Moreover, $\iC$ admits a monoidal structure extending that of $\C$, and the inclusion is strong monoidal. While $\iC$ is not rigid, it is closed monoidal.
When $\C$ is braided, the braiding of $\C$ naturally extends to $\iC$.
When $\C$ is moreover a ribbon category, the twist of $\C$ naturally extends to a natural isomorphism $\id_{\iC} \to \id_{\iC}$ satisfying the first equation in \eqref{eq:ribbon-axioms}.


\subsubsection{Category of modules over an algebra in the ind-completion}
Let $\C$ be a tensor category and $A$ an algebra in the ind-completion $\iC$.
The category $\iC_A$ is identified with the category of algebras over the monad $T := -\otimes A$.
Since $\iC$ is closed monoidal, the functor $T$ admits a right adjoint.
Thus $\iC_A$ is an abelian category such that the forgetful functor $\iC_A\to\iC$ is exact and faithful.


\subsubsection{Various module categories}\label{subsubsec:various-modules}
Let $A$ be an algebra in $\iC$. 
We recall the following three finiteness conditions for a right $A$-module. 

\begin{definition}
\label{def:various-modules}
  A right $A$-module $M$ in $\iC$ is said to be \emph{finitely generated} if there exists an epimorphism $X\otimes A\twoheadrightarrow M$ in $\iC_A$ for some $X\in\C$.
  $M$ is \emph{finitely presented} if there exists an exact sequence $X\otimes A \to Y\otimes A \to M \to 0$ in $\iC_A$ for some $X,Y\in\C$.
  $M$ is \emph{of finite length} if it admits a finite composition series in $\iC_A$.
  We denote by $\fg\text{-}\iC_A$, $\fp\text{-}\iC_A$ and $\fl\text{-}\iC_A$ the full subcategories of $\iC_A$ consisting of finitely generated, finitely presented and finite length modules, respectively.
\end{definition}

It is obvious that the full subcategory $\fg\text{-}\iC_A$ is closed under quotients. However, $\fg\text{-}\iC_A$ need not be closed under kernels in $\iC_A$, so it may fail to be abelian. The full subcategory $\fl\text{-}\iC_A$ of finite-length objects is a Serre subcategory of $\iC_A$; in particular it is abelian.

We recall that $\iC_A$ is monoidal with respect to $\otimes_A$ when $A$ is central commutative. Although we do not know whether $\fp\text{-}\iC_A$ and $\fl\text{-}\iC_A$ are closed under $\otimes_A$, one can prove:

\begin{lemma}
  If $A$ is central commutative, then $\fg\text{-}\iC_A$ is closed under $\otimes_A$.
\end{lemma}
\begin{proof}
  If $M_i \in \fg\text{-}\iC_A$ ($i = 1, 2$) is a quotient of $X_i \otimes A$ ($X_i \in \C$), then $M_1 \otimes_A M_2$ is a quotient of $(X_1 \otimes A) \otimes_A (X_2 \otimes A) \cong (X_1 \otimes X_2) \otimes A$ and thus it is finitely generated.
\end{proof}

From the definitions, we have $\fp\text{-}\iC_A\subseteq \fg\text{-}\iC_A$. As we show next, $\fl\text{-}\iC_A\subseteq \fg\text{-}\iC_A$ as well.

\begin{lemma}\label{lem:finite-length-fg}
Every finite length $A$-module is finitely generated.
\end{lemma}
\begin{proof}
Let $M\in \iC_A$ be of finite length. Let $\{F_i\}_{i\in I}$ be the family of all finitely generated submodules of $M$.
We first claim that $M=\sum_{i\in I} F_i$.
Since $M$ is an object of $\iC$, we can write $M=\sum_{\lambda\in \Lambda} X_\lambda$ for some family of subobjects $X_\lambda\subset M$ with $X_\lambda\in \C$. For each $\lambda\in \Lambda$, let $W_{\lambda}$ be the image of the subobject $X_{\lambda} \otimes A$ under the action $M \otimes A \to M$. Then $W_\lambda$ is a submodule of $M$, and by construction it is a quotient of the free $A$-module $X_\lambda\otimes A$. Hence $W_\lambda$ is finitely generated. Moreover, $X_\lambda\subset W_\lambda$ by the unit axiom. Therefore
\[
M=\sum_{\lambda\in\Lambda} X_\lambda \subset \sum_{\lambda\in\Lambda} W_\lambda \subset \sum_{i \in I} F_i \subset M,
\]
which proves the claim.

Next, note that the family $\{F_i\}_{i\in I}$ is directed under inclusion. Indeed, if $F_i$ and $F_j$ are finitely generated submodules of $M$, then $F_i\oplus F_j$ is also finitely generated. Hence, its quotient $F_i+F_j$ is a finitely generated submodule of $M$. Thus there exists $k\in I$ such that $F_i\subset F_k$ and $F_j\subset F_k$.

Assume for contradiction that $M$ is not finitely generated. Choose $j_1\in I$ such that $F_{j_1}\neq 0$. Since $F_{j_1}\neq M$ and $M=\sum_{i\in I}F_i$, there exists $i_2\in I$ such that $F_{i_2}\not\subset F_{j_1}$. By directedness, there exists $j_2\in I$ such that $F_{j_1}+F_{i_2}\subset F_{j_2}$.
Then $F_{j_1}\subsetneq F_{j_2}$. Repeating this argument, we construct a strictly increasing chain $F_{j_1}\subsetneq F_{j_2}\subsetneq F_{j_3}\subsetneq \cdots$ of submodules of $M$, which is impossible because $M$ has finite length. Hence $M$ is finitely generated.
\end{proof}

\begin{remark}
  As $\iC$ is a locally presentable category, so is $\iC_A$. In such a category, there is a general notion of finitely presented and finitely generated objects, defined by the condition that $\Hom_{\iC_A}(M,-)$ commutes with filtered colimits and filtered colimits of monomorphisms, respectively. Our definitions of finitely generated and finitely presented modules coincide with these general notions \cite{shibata2023nakayama}. Thus, the lemma above can be proven by appealing to general results on locally presentable categories, but we have given direct proofs for the convenience of readers.
\end{remark}


\subsubsection{Local modules in ind-completion}\label{subsubsec:local-modules-ind-completion}
Let $\C$ be a braided tensor category, and let $A$ be a commutative algebra in $\iC$. 
By \cite[Proposition~2.56]{creutzig2024tensor} (specialized to the non-super case), $\iC_A^\loc$ is closed under taking kernels and cokernels. Hence it is an abelian category and the inclusion functor $i : \iC_A^\loc \hookrightarrow \iC_A$ is exact.

In fact, $\iC_A^{\loc}$ is a coreflective subcategory of $\iC_A$ \cite[Theorem~3.6]{pareigis1995braiding}. That is, $i$ admits a right adjoint $i^R: \iC_A \to \iC_A^\loc$. Explicitly, for $(M,\rho)\in \iC_A$, $i^R(M,\rho)$ is the equalizer of the maps 
\begin{equation}\label{eq:right-adjoint-local}
  [A,\rho]\circ \eta_M, [A,\rho \circ c_{A,M}\circ c_{M,A}]\circ \eta_M : M\to [A,M],
\end{equation} 
where $\eta_M : M \to [A, M \otimes A]$ is the component of the unit of the adjunction. In particular, $\iC_A^\loc$ is closed under  colimits in $\iC_A$ and the inclusion $i$ preserves colimits.

\begin{definition}
\label{def:fg-local}
We define the full subcategories $\fl\text{-}\iC_A^\loc$, $\fg\text{-}\iC_A^\loc$ and $\fp\text{-}\iC_A^\loc$ of $\iC_A^\loc$ as the intersections of $\iC_A^\loc$ with $\fl\text{-}\iC_A$, $\fg\text{-}\iC_A$ and $\fp\text{-}\iC_A$, respectively.
\end{definition}


\subsection{Invertibles in a tensor category}
Let $\C$ be a locally finite $\kk$-linear abelian monoidal category with simple unit. We recall some basic facts about invertible objects in $\C$. An object $X\in\C$ is called \emph{invertible} if there exists an object $Y\in\C$ such that $X\otimes Y \cong \unit \cong Y \otimes X$. The set of isomorphism classes of invertible objects in $\C$ forms a group under the tensor product, denoted by $\Inv(\C)$.


\subsubsection{Group cohomology}
We fix our convention for group cohomology. Let $G$ be a group. For an integer $n \ge 1$, we denote by $\mathrm{C}^n(G)$ the set of all maps from $G \times \cdots \times G$ ($n$ times) to $\kk^{\times}$ such that $f(x_1, \cdots, x_n) = 1$ whenever one of $x_i$ is the identity element. An element of $\mathrm{C}^n(G)$ is called a \emph{normalized $n$-cochain}. For each $n \ge 1$, there is the \emph{coboundary map $\partial_{n+1} : \mathrm{C}^{n}(G) \to \mathrm{C}^{n+1}(G)$}. For example,
\begin{gather}  
  \label{eq:coboundary-map-1}
  \partial_2(f)(x, y) = f(x) f(x y)^{-1} f(y), \\
  \label{eq:coboundary-map-2}
  \partial_3(g)(x, y, z) = g(x, y) g(x, y z)^{-1} g(x y, z) g(y, z)^{-1}
\end{gather}
for $f \in \mathrm{C}^1(G)$, $g \in \mathrm{C}^2(G)$ and $x, y, z \in G$. An element of $\Ker(\partial_{n+1})$ and of $\Img(\partial_{n})$ is called an \emph{$n$-cocycle} and an \emph{$n$-coboundary}, respectively.
We define the \emph{$n$-th cohomology group} $\mathrm{H}^n(G)$ to be the quotient group $\Ker(\partial_{n+1}) / \Img(\partial_n)$.


\subsubsection{A 3-cocycle arising from the associator} 
For each class $g \in \Inv(\C)$, we choose an object $E_g \in \C$ representing $g$. We assume that $E_1$ is the unit object of $\C$. Since $E_1$ is simple and the endofunctor $E_g \otimes (-)$ on $\C$ is an equivalence, $E_g \cong E_g \otimes E_1$ is also a simple object. Thus, by Schur's lemma and local finiteness, the space $\Hom_{\C}(E_g, E_h)$ for $g, h \in \Inv(\C)$ is one-dimensional if $g = h$, and zero otherwise.
We further choose isomorphisms
\begin{equation*}
  \phi_{x,y} : E_x \otimes E_y \to E_{x y}
  \quad (x, y \in \Inv(\C))
\end{equation*}
such that $\phi_{1,x}$ and $\phi_{x,1}$ are unit isomorphisms for all $x \in \Inv(\C)$. We define the map $\omega : \Inv(\C)^3 \to \kk^{\times}$ so that the following diagram is commutative:
\begin{equation}
  \label{eq:omega-def}
  \begin{tikzcd}[column sep = 48pt]
    (E_x \otimes E_y) \otimes E_z
    \arrow[d, "{\alpha_{E_x , E_y, E_z}}"']
    \arrow[r, "{\phi_{x,y} \otimes \id}"]
    & E_{x y} \otimes E_z
    \arrow[r, "{\phi_{x y, z}}"]
    & E_{x y z}
    \arrow[d, "{\omega(x, y, z) \, \id}"] \\
    E_x \otimes (E_y \otimes E_z)
    \arrow[r, "{\id \otimes \phi_{y,z}}"]
    & E_{x} \otimes E_{y z}
    \arrow[r, "{\phi_{x, y z}}"]
    & E_{x y z},
  \end{tikzcd}
\end{equation}
for all $x, y, z \in \Inv(\C)$, where $\alpha$ is the associator. As is well-known, the pentagon axiom implies that $\omega$ is a 3-cocycle on the group $\Inv(\C)$. The cocycle $\omega$ depends on the choice of the isomorphisms $\phi_{x, y}$; however, its cohomology class does not.

\subsubsection{An abelian 3-cocycle arising from the braiding}

We assume that the monoidal category $\C$ has a braiding $c$. Then we also define the map $\beta : \Inv(\C)^2 \to \kk^{\times}$ by 
\begin{equation}
  \label{eq:beta-def}
  \beta(x, y) \id_{E_{x y}} = \phi_{y, x} \circ c_{E_x, E_y} \circ \phi_{x, y}^{-1}
\end{equation}
for $x, y \in \Inv(\C)$. 
The hexagon axioms for the braiding imply that the pair $(\omega,\beta)$ forms an \emph{abelian 3-cocycle} on $\Inv(\C)$ in the sense of Eilenberg--Mac Lane \cite{eilenberg1954groups}; see also \cite{joyal1993braided}. 
The function $\mathsf{q} : \Inv(\C) \to \kk^\times$ defined by $\mathsf{q}(g) = \beta(g,g)$ is a quadratic form whose associated symmetric bicharacter is the double braiding scalar $b_{\mathsf{q}}(g,h) = \beta(g,h)\beta(h,g)$. This perspective is developed in \cite{joyal1993braided}, \cite[\S2.11, Appendix~D]{drinfeld2010braided} and \cite[\S8.4]{etingof2016tensor}.


\subsubsection{(Commutative) algebras from invertible objects}\label{subsubsec:comm-group-algebras}
Let $\C$ be a locally finite $\kk$-linear abelian closed monoidal category, and define $\omega$ (and $\beta$) as before (when $\C$ is braided).
We now introduce the following class of algebra in $\iC$, which is the main object of study in this paper:

\begin{definition} \label{def:simple-current-alg}
  Let $\Gamma$ be a subgroup of $\Inv(\C)$. A \emph{simple current algebra} over $\Gamma$ is an algebra $A$ such that $A = \bigoplus_{g \in \Gamma} E_g$ as an object of $\iC$, the multiplication $A \otimes A \to A$ restricts to an isomorphism $E_x \otimes E_y \to E_{x y}$ for all elements $x, y \in \Gamma$ and the unit is given by the inclusion morphism $\unit = E_1 \to A$. 
\end{definition}

\begin{lemma}\label{lem:group-algebra-comm}
  We fix a subgroup $\Gamma$ of $\Inv(\C)$ and set $A = \bigoplus_{g \in \Gamma} E_g$.
\begin{enumerate}
  \item The object $A$ becomes a simple current algebra over $\Gamma$ if and only if the restriction of the 3-cocycle $\omega$ to $\Gamma \times \Gamma \times \Gamma$ is a coboundary. If this is the case, for any cochain $\eta \in \mathrm{C}^2(\Gamma)$ such that $\omega|_{\Gamma \times \Gamma \times \Gamma} = \partial_3(\eta)$, the morphism $\mu : A \otimes A \to A$ induced by
  \begin{equation}
  \label{eq:mu-def}
  \mu_{x, y} := \eta(x, y) \phi_{x, y} : E_x \otimes E_y \to E_{x y}
  \quad (x, y \in \Gamma)
  \end{equation}
  makes $A$ a simple current algebra over $\Gamma$.
  \item We assume that $\C$ is braided. Then $A$ is a commutative simple current algebra if and only if $\mathsf{q}|_\Gamma \equiv 1$, i.e., $\beta(g,g) = 1$ for all $g \in \Gamma$. If this is the case, then we have
\begin{equation}\label{eq:comm-implies-trivial-monodromy}
    \beta(x,y)\beta(y,x)=1, \qquad \text{for all} \;\; x,y\in\Gamma.
\end{equation}  
\end{enumerate}
\end{lemma}
Parts~(a) and~(b) are well known; see \cite[\S8.8]{etingof2016tensor} for the pointed fusion category case, \cite[\S2.11]{drinfeld2010braided} for the general statement, and \cite[\S3]{fuchs2004tft} for simple current extensions.

\begin{proof}
(a) Given a normalized 2-cochain $\eta : \Gamma \times \Gamma \to \kk^{\times}$, it is straightforward to verify that the morphism $\mu : A \otimes A \to A$ induced by \eqref{eq:mu-def} makes $A$ a simple current algebra if and only if $\omega|_{\Gamma \times \Gamma \times \Gamma} = \partial_3(\eta)$. In particular, we have proved the `if' part. To show the converse, we assume that $A$ is a simple current algebra with multiplication $\mu$. By the definition of a simple current algebra, the multiplication must be induced by \eqref{eq:mu-def} for some map $\eta : \Gamma \times \Gamma \to \kk^{\times}$. The same computation as `if' part completes the proof.

(b) Suppose $A$ is a simple current algebra with multiplication induced by $\eta$. The commutativity condition $\mu \circ c_{A,A} = \mu$ on the $(x,y)$-component reads
\begin{equation}\label{eq:comm-condition}
    \eta(y,x)\,\beta(x,y) = \eta(x,y) \qquad (x,y \in \Gamma).
\end{equation}
Setting $x = y$ and using $\eta(x,x) \neq 0$ gives $\beta(x,x) = 1$, i.e., $\mathsf{q}|_\Gamma \equiv 1$.
Conversely, assume that $\mathsf{q}|_\Gamma \equiv 1$.
Then the abelian $3$-cocycle $(\omega|_{\Gamma^3}, \beta|_{\Gamma^2})$ has trivial associated quadratic form.
By the Eilenberg--Mac Lane isomorphism between abelian cohomology and quadratic forms (see \cite{galindo2026note} for a recent reference), its class in abelian cohomology is therefore trivial.
Hence there exists a normalized $2$-cochain $\eta \in \mathrm{C}^2(\Gamma)$ such that
\[
  \partial_3(\eta) = \omega|_{\Gamma^3}
  \qquad\text{and}\qquad
  \beta(x,y) = \eta(x,y)\eta(y,x)^{-1}
  \quad (x,y \in \Gamma).
\]
The first identity gives an algebra structure on $A$ by part~(a), while the second identity is exactly the commutativity condition \eqref{eq:comm-condition}. Thus $A$ is a commutative algebra.
Equation \eqref{eq:comm-implies-trivial-monodromy} easily follows from \eqref{eq:comm-condition}.
\end{proof}

The choice of $\eta$ in (a) is not unique. If a simple current algebra $A_i$ ($i = 1, 2$) is constructed from a cochain $\eta_i \in \mathrm{C}^2(\Gamma)$ satisfying $\partial_3(\eta_i) = \omega|_{\Gamma \times \Gamma \times \Gamma}$, then $A_1 \cong A_2$ as algebras if and only if $\eta_1$ and $\eta_2$ are cohomologous. In other words, the set of the isomorphism classes of simple current algebras over $\Gamma$ (if it is non-empty) is a torsor over $\mathrm{H}^2(\Gamma)$.

The choice of $\eta$ in (b) is also not unique. Suppose that cochains $\eta_1,\eta_2 \in \mathrm{C}^2(\Gamma)$ define two commutative simple current algebra structures on the object $A = \bigoplus_{g \in \Gamma} E_g$.
Then $\rho := \eta_1\eta_2^{-1}$ is a symmetric $2$-cocycle on the abelian group $\Gamma$.
Since the group $\kk^\times$ is divisible, every symmetric $2$-cocycle $\Gamma \times \Gamma \to \kk^\times$ is a coboundary.
Thus there exists $\lambda \in \mathrm{C}^1(\Gamma)$ such that $\rho = \partial_2(\lambda)$.
Rescaling each summand $E_x$ by the scalar $\lambda(x)$ yields an algebra isomorphism between the two structures. In particular, the cohomology class of $\eta$ in $\mathrm{H}^2(\Gamma)$, and hence the isomorphism class of the commutative algebra $A$, is uniquely determined.


\section{A construction of ribbon braided tensor categories}\label{sec:framework}
In this section, we present a framework for constructing new ribbon braided tensor categories. We begin with a braided tensor category $\C$ and consider commutative algebras within its ind-completion $\iC$. The full category of modules over such an algebra is typically too large to form a tensor category of the type we study, as it will include non-rigid and infinite length objects. Therefore, we must identify specific subcategories that are closed under tensor products and satisfy necessary finiteness conditions.

We establish that for \emph{Artinian algebras} the three finiteness conditions of finitely generated, finitely presented, and finite length modules coincide (Theorem~\ref{thm:artinian-locally-finite}), yielding a $\kk$-linear abelian monoidal category where every object is of finite length. In Section~\ref{subsec:ind-exact-algebras}, we introduce ind-exact algebras (following \cite{coulembier2023incompressible}) to ensure rigidity of the category of finitely generated modules. In Section~\ref{subsec:frobenius-algebras}, we introduce the notion of a Frobenius algebra in a closed monoidal category by extending the commonly used definition in the rigid case (see Definition \ref{def:Frobenius-alg-in-Ind}), and show that the category of finitely generated local modules is ribbon if the algebra $A$ is Artinian, ind-exact and Frobenius in our sense and the twist is the identity on $A$. The main results are summarized in Theorem~\ref{thm:main-criteria} below.

\smallskip
\noindent
\textbf{Convention.} Throughout this section, $\C$ denotes a tensor category (with additional structure imposed when needed), $\iC$ its ind-completion, and $A$ an algebra in $\iC$.

\begin{theorem}\label{thm:main-criteria}
  Suppose that $\C$ is a braided tensor category and $A\in \iC$ is a commutative algebra.
  If $A$ is Artinian, haploid and ind-exact, then $\fg\text{-}\iC_A$ is a tensor category and
  $\fg\text{-}\iC_A^{\loc}$ is a braided tensor category.
  If in addition, $\C$ is ribbon, $A$ is Frobenius and $\theta_A=\id_A$, then $\fg\text{-}\iC_A^{\loc}$ is a ribbon
  tensor category.
\end{theorem}
\begin{proof}
This theorem is a summary of results proved later in this section. The first claim follows from Theorem~\ref{thm:indexact-multitensor-local}, and the second claim follows from Theorem~\ref{thm:local-modules-ribbon}.
\end{proof}


\subsection{Artinian algebras}\label{subsec:artinian-algebras}
Recall the subcategories $\fl\text{-}\iC_A$, $\fg\text{-}\iC_A$, and $\fp\text{-}\iC_A$ of finite length, finitely generated, and finitely presented $A$-modules, respectively, from Definition \ref{def:various-modules}.

Ideally, one would like to work inside the full subcategory of $\iC_A$ consisting of objects that are simultaneously finitely generated, finitely presented, and of finite length.
This leads to the guiding problem: find algebras $A$ for which these three finiteness conditions coincide.


\begin{definition} \label{def:artinian-alg}
  We call an algebra $A\in\iC$ \emph{Artinian} \cite{coulembier2023incompressible} if every finitely generated right $A$-module is of finite length.
\end{definition}

We first record the key property of Artinian algebras.
\begin{lemma}\label{lem:fgindCA-fin-presented}
	    If $A$ is Artinian, every finitely generated right $A$-module is finitely presented.
\end{lemma}
\begin{proof}
Let $M$ be a finitely generated $A$-module. By definition there exists an epimorphism
$f:X\otimes A\twoheadrightarrow M$ in $\iC_A$ for some $X\in\C$.
Let $K:=\ker(f)$ and let $k:K\hookrightarrow X\otimes A$ be the kernel inclusion, so we have an exact sequence
\[
0\to K \xrightarrow{k} X\otimes A \xrightarrow{f} M \to 0.
\]

Since $X\otimes A$ is finitely generated and $A$ is Artinian, $X\otimes A$ has finite length.
Hence $K$ has finite length as a subobject of $X\otimes A$ in the abelian category $\iC_A$.
By Lemma~\ref{lem:finite-length-fg}, $K$ is finitely generated. Thus there exists $X'\in\C$ and an epimorphism
$g:X'\otimes A\twoheadrightarrow K$.

Then $k\circ g:X'\otimes A\to X\otimes A$ has image $K=\ker(f)$, so $f$ is the cokernel of $k\circ g$.
Equivalently,
\[
X'\otimes A \xrightarrow{k\circ g} X\otimes A \xrightarrow{f} M \to 0
\]
is exact, proving that $M$ is finitely presented.
\end{proof}

Our interest in Artinian algebras stems from the following theorem, which shows that being Artinian forces these finiteness notions to coincide. 

\begin{theorem}
  \label{thm:artinian-locally-finite}
    Let $A\in\iC$ be an Artinian algebra. Then we have $\fl\text{-}\iC_A = \fg\text{-}\iC_A = \fp\text{-}\iC_A$ and these categories are abelian $\kk$-linear categories where objects have finite length.
  If $\C$ has enough projective objects, then so do these categories. 
\end{theorem}
\begin{proof}
By Lemma~\ref{lem:fgindCA-fin-presented}, it is immediate that $\fg\text{-}\iC_A=\fp\text{-}\iC_A$. Next, using Lemma~\ref{lem:finite-length-fg} together with the Artinian property, we see that an object of $\iC_A$ is finitely generated if and only if it is of finite length, i.e. $\fl\text{-}\iC_A=\fg\text{-}\iC_A$. Thus, the three categories are equal. In particular, $\fg\text{-}\iC_A$ is abelian, $\kk$-linear and objects have finite length.

Assume that $\C$ has enough projective objects.
Let $M\in\fg\text{-}\iC_A$, and choose an epimorphism $q:X\otimes A\twoheadrightarrow M$ with $X\in\C$.
Choose an epimorphism $\pi:P\twoheadrightarrow X$ in $\C$ with $P$ projective.
Then $\pi\otimes\id_A:P\otimes A\twoheadrightarrow X\otimes A$ is an epimorphism in $\iC_A$, and hence so is
$q\circ(\pi\otimes\id_A):P\otimes A\twoheadrightarrow M$.
Moreover $P\otimes A$ is projective in $\iC_A$: indeed, for any $A$-module $N$ we have an adjunction
\[
\Hom_{\iC_A}(P\otimes A,N)\cong \Hom_{\iC}(P,U(N)),
\]
where $U$ is the forgetful functor, which is exact because kernels and cokernels in $\iC_A$ are computed in
$\iC$.
Since $\Hom_{\iC}(P,-)$ is exact, it follows that $\Hom_{\iC_A}(P\otimes A,-)$ is exact.
Thus $\fg\text{-}\iC_A$ has enough projectives.
\end{proof}


\subsection{Ind-exact algebras}\label{subsec:ind-exact-algebras}
We consider the following notion from \cite[\S6]{coulembier2023incompressible}.

\begin{definition} \label{def:ind-exact-alg}
We call an algebra $A$ in $\iC$ \emph{ind-exact} if the functor $-\otimes_A -: \fg\text{-}\iC_A \times \fg\text{-}{}_A\iC \to \iC$ is bi-exact. Here, $\fg\text{-}{}_A\iC$ is the category of finitely generated left $A$-modules, which is defined in the same manner as $\fg\text{-}\iC_A$.
\end{definition}

\begin{theorem}\label{thm:ind-exact}
Let $A$ be an Artinian central commutative algebra in $\iC$. The following are equivalent:
\begin{enumerate}
    \item $A$ is ind-exact.
    \item The category $\fg\text{-}\iC_A$ is rigid.
    \item Every simple $A$-module is rigid.
\end{enumerate}
\end{theorem}
\begin{proof}
$(a)\Rightarrow(b)$: The Artinian assumption implies that $\fp\text{-}\iC_A=\fg\text{-}\iC_A$. Thus, every object in $\fg\text{-}\iC_A$ is finitely presented, that is, it is the cokernel of a morphism $X\otimes A \to Y\otimes A$ for some $X,Y\in\C$. Here $X\otimes A$ and $Y\otimes A$ are rigid. So,  
the claim follows because rigid objects in an abelian monoidal category with biexact tensor product are closed under taking cokernels \cite{wiggins2018dualizable}.

$(b)\Rightarrow(a)$: Rigidity implies that the functor $-\otimes_A V$ admits left/right adjoints given by tensoring with left/right duals of $V$ in $\fg\text{-}\iC_A$. Thus, tensoring over $A$ is bi-exact.

$(b)\Rightarrow(c)$: This is immediate.

$(c)\Rightarrow(b)$: Denote $\D=\fg\text{-}\iC_A$. Since $A$ is Artinian, Theorem~\ref{thm:artinian-locally-finite} implies that $\D$ is an abelian
$\kk$-linear category in which every object has finite length. In $\D$, every object is a quotient of a free module $X\otimes A$ for some $X\in\C$, which is rigid.
Thus, by a non-braided version of Etingof and Penneys' lemma (see Theorem~\ref{thm:EP-rigidity-main}), rigidity of all simple objects implies rigidity of the whole category.
\end{proof}

\begin{theorem}\label{thm:indexact-multitensor}
    Let $A\in \iC$ be a haploid, Artinian algebra that is central. If either one of the conditions in Theorem~\ref{thm:ind-exact} is satisfied, then $\fg\text{-}\iC_A$ is a tensor category. If moreover $\C$ is Frobenius, then $\fg\text{-}\iC_A$ is a Frobenius tensor category. 
\end{theorem}
\begin{proof}
By Theorem~\ref{thm:artinian-locally-finite}, $\fg\text{-}\iC_A$ is an abelian $\kk$-linear category where all objects have finite length. Moreover, Theorem~\ref{thm:ind-exact} implies that $\fg\text{-}\iC_A$ is rigid. 
As $A$ is haploid, 
\[ \Hom_{\fg\text{-}\iC_A}(A,A) \cong \Hom_{\iC}(\unit,A) \cong\kk. \]
Consequently, Lemma~\ref{lem:hom-finite} implies that $\fg\text{-}\iC_A$ is locally finite. The same argument as in Theorem \ref{thm:artinian-locally-finite} applies to the monoidal subcategory $\fg\text{-}\iC_A$ and the claim follows.
\end{proof}

\begin{remark}
We can relax the assumption that $\Hom(\unit,A)\cong\kk$ to the weaker assumptions that $\dim(\Hom(\unit,A))<\infty$ and that $A$ is connected\footnote{$A$ is called connected if the ring $A^{\mathrm{inv}}:=\Hom(\unit,A)$ is a finite dimensional connected ring, that is, there are no nontrivial idempotents. Equivalently, $A$ is not isomorphic to $A_1\oplus A_2$ for some algebras $A_1$ and $A_2$.}. Indeed, $\Hom_{\iC}(\unit,A)$ is a finite-dimensional connected algebra, and the unit object of $\fg\text{-}\iC_A$ (namely $A$) satisfies $\Hom_{\fg\text{-}\iC_A}(A,A) \cong \Hom_{\iC}(\unit,A)$, which is a finite dimensional connected algebra. Hence, by a similar argument as in \cite[Theorem~4.3.1]{etingof2016tensor}, $\Hom_{\fg\text{-}\iC_A}(A,A)$ is isomorphic to a direct sum of copies of $\kk$. Since it is also connected, we get $\Hom_{\fg\text{-}\iC_A}(A,A)\cong \kk$.
\end{remark}


\subsection{Frobenius algebras}\label{subsec:frobenius-algebras}
In this subsection, we introduce Frobenius algebras in $\iC$ and discuss their properties. We show that if $\C$ is ribbon and $A$ is a Frobenius algebra, then the category of local modules over $A$ admits a ribbon structure.

Recall that $\iC$ is closed monoidal and we denote the right internal Hom as $[-,-]$. It satisfies  $\Hom_{\iC}(L\otimes M, N)\cong \Hom_{\iC}(L,[M,N])$.

\begin{definition} \label{def:Frobenius-alg-in-Ind}
  Let $A$ be an algebra in $\iC$ with multiplication $\mu$, and let $\lambda : A \to \unit$ be a morphism in $\iC$. We define $\phi : A \to [A, \unit]$ to be the morphism in $\iC$ corresponding to $\lambda \mu : A \otimes A \to \unit$ under the adjunction isomorphism $\Hom_{\iC}(A \otimes A, \unit) \cong \Hom_{\iC}(A, [A, \unit])$. We say that $A$ is a \emph{Frobenius algebra with Frobenius form $\lambda$} if the morphism $\phi$ is invertible.
\end{definition}

As we will see in Section~\ref{sec:simple-current-algebras}, a `simple current algebra' is such a Frobenius algebra (Theorem~\ref{thm:simple-current-ext-is-Fb}). We expect co-Frobenius Hopf algebras to provide more examples, and we will pursue this line of inquiry in future work. 

Now we give properties of Frobenius algebras. Lemma \ref{lem:free-right-adj-1} below is an analogue of the fact that the free module functor for an ordinary Frobenius algebra is right adjoint to the forgetful functor from the category of modules.

\begin{lemma}
  \label{lem:free-right-adj-1}
    Let $A$ be a Frobenius algebra in $\iC$. Then there is a natural isomorphism
    \begin{equation}
        \label{eq:free-right-adj-1}
        \Hom_{\iC_A}(M, X \otimes A) \cong \Hom_{\iC}(M, X)
    \end{equation}
    for $M \in \iC_A$ and $X \in \C$.
\end{lemma}
\begin{proof}
There is a natural transformation $f_{X,M} : X \otimes [A, M] \to [A, X \otimes M]$ for $X, M \in \iC$ making the functor $[A, -] : \iC \to \iC$ a `lax' left $\iC$-module functor since $[A, -]$ is a right adjoint of the left $\iC$-module functor $- \otimes A : \iC \to \iC$. We note that $f_{X,M}$ is invertible if $X$ is rigid. It is easy to see that the morphism $\phi : A \to [A, \unit]$ in Definition~\ref{def:Frobenius-alg-in-Ind} is an isomorphism of right $A$-modules. Hence, for $X \in \C$, we have $X \otimes A \cong X \otimes [A, \unit]\cong [A, X]$ as right $A$-modules. Now we have
    \begin{gather*}
        \Hom_{\iC_A}(M, X \otimes A)
        \cong \Hom_{\iC_A}(M, [A, X])
        \cong \Hom_{\iC}(M \otimes_A A, X) \cong \Hom_{\iC}(M, X)
    \end{gather*}
    for $M \in \iC_A$ and $X \in \C$. The proof is done.
\end{proof}


We define $\iHom_A(M,-)$ to be the right adjoint of the functor $-\otimes_A M: \iC_A \to \iC$. This can be constructed in the same way as the internal Hom functor for ${}_A\iC_A$ (see \cite{shimizu2024commutative}).

\begin{lemma}
  \label{lem:free-right-adj-2}
  With the above notation, we have a natural isomorphism
  \begin{equation*}
    \iHom_A(M, X \otimes A) \cong [M, X]
  \end{equation*}
  of right $A$-modules for $M \in {}_A\iC_A$ and $X \in \C$.
\end{lemma}

\begin{proof}[Proof of Lemma~\ref{lem:free-right-adj-2}]
  For a right $A$-module $L$, we have natural isomorphisms
  \begin{align*}
    \Hom_{A}(L, \iHom_A(M, X \otimes A))
    & \cong \Hom_{A}(L \otimes_A M, X \otimes A) \\
    & \cong \Hom_{\iC}(L \otimes_A M, X) \\
    & \cong \Hom_{A}(L, [M, X]),
  \end{align*}
  where we have used Lemma \ref{lem:free-right-adj-1}.
  Thus the claim follows from the Yoneda lemma.
\end{proof}


\subsection{Braided case}
We also record the analogous finiteness consequences for the local module category. In this section, $\C$ is assumed to be a braided tensor category.

\begin{proposition}\label{prop:artinian-locally-finite-local}
Let $A\in\iC$ be a commutative Artinian algebra. Then we have
\[ \fl\text{-}\iC_A^\loc = \fg\text{-}\iC_A^\loc = \fp\text{-}\iC_A^\loc, \]
and these categories are finite length, abelian, $\kk$-linear, braided monoidal categories.
If, moreover, $\fg\text{-}\iC_A$ has enough injective objects, then so do these categories. 
\end{proposition}

\begin{proof}
The first statement follows from Theorem~\ref{thm:artinian-locally-finite}. As $\fg\text{-}\iC_A^\loc = \fg\text{-}\iC_A\cap \iC_A^\loc$ and both $\fg\text{-}\iC_A$ and $\iC_A^\loc$ are abelian $\kk$-linear subcategories of $\iC_A$, $\fg\text{-}\iC_A^\loc$ is an abelian $\kk$-linear subcategory as well. Lastly, since $\fg\text{-}\iC_A$ is finite length and $\fg\text{-}\iC_A^\loc$ is a full subcategory of $\fg\text{-}\iC_A$, $\fg\text{-}\iC_A^\loc$ is finite length as well.

Recall that the inclusion functor $i: \iC_A^\loc \to \iC_A$ admits a right adjoint $i^R$. If $M\in \fg\text{-}\iC_A=\fl\text{-}\iC_A$, then $i^R(M)$ (which is a subobject of $M$, see \eqref{eq:right-adjoint-local}) is also of finite length, hence finitely generated. Thus, the right adjoint $i^R: \iC_A \to \iC_A^\loc$ restricts to a functor $\fg\text{-}\iC_A \to \fg\text{-}\iC_A^\loc$. 

Suppose that $\fg\text{-}\iC_A$ has enough injective objects. Let $M\in \fg\text{-}\iC_A^\loc$ and choose an injection $i(M)\hookrightarrow I$ in $\fg\text{-}\iC_A$ with $I$ injective. Then, $i^R(I)$ is injective in $\iC_A^\loc$ because $\Hom_{\iC_A^\loc}(-, i^R(I)) \cong \Hom_{\iC_A}(i(-), I)$ is exact. As $M\hookrightarrow i^R(I)$, we see that $\fg\text{-}\iC_A^\loc$ has enough injective objects as well.

Lastly, when $A$ is a commutative algebra, the braided monoidal structure on $\iC_A^\loc$ restricts to $\fg\text{-}\iC_A^\loc$ because the tensor product of two finitely generated local modules is again finitely generated.
\end{proof}


\subsubsection{Rigidity}

\begin{theorem}\label{thm:indexact-multitensor-local}
 Let $A\in \iC$ be an Artinian, haploid commutative algebra. If either one of the conditions in Theorem~\ref{thm:ind-exact} is satisfied, then $\fg\text{-}\iC_A$ is a tensor category and
  $\fg\text{-}\iC_A^\loc$ is a braided tensor category. If moreover $\C$ is Frobenius, then both are Frobenius tensor categories. 
\end{theorem}
\begin{proof}
  The proof is similar to that of Theorem~\ref{thm:indexact-multitensor}. The key point is that $\iC_A^\loc$ is a closed monoidal subcategory of $\iC_A$. Hence, $\fg\text{-}\iC_A^\loc$ is a closed braided monoidal subcategory of $\fg\text{-}\iC_A$. Thus, if $\fg\text{-}\iC_A$ is rigid, so is $\fg\text{-}\iC_A^\loc$. Since $A$ is haploid, $\End_{\fg\text{-}\iC_A}(A)\cong\Hom_{\iC}(\unit,A)$ is connected; as above, this identifies the unit endomorphism ring with $\kk$, so these are tensor categories. 

  We recall from \cite{shibata2023nakayama} that a tensor category is Frobenius if and only if it has enough injective objects, if and only if it has enough projective objects. Now we assume that $\C$ is Frobenius. Then, by Theorem~\ref{thm:artinian-locally-finite}, the tensor category $\fg\text{-}\iC_A$ is Frobenius.
  Hence, by Proposition~\ref{prop:artinian-locally-finite-local}, $\fg\text{-}\iC_A^\loc$ is also Frobenius.
\end{proof}


\subsubsection{Ribbon structure on local modules}
Let $\C$ be a ribbon tensor category with braiding $c$ and twist $\theta$. Then $\iC$ is a closed braided monoidal category. Moreover, the twist induces an automorphism of the identity functor of $\iC$ that satisfies $\theta_{M \otimes N} = c_{N,M} c_{M,N} (\theta_M \otimes \theta_N)$ for all $M, N \in \iC$.

We denote by $\iC'$ the M\"uger center of $\iC$. Next we will show that, in the ind-completion of a ribbon category, the internal Hom is compatible with the ribbon structure. In particular, $\theta_{[M, \unit]} = [\theta_M, \unit]$ for $M \in \iC$.
\begin{lemma}
  For all objects $M \in \iC$ and $X \in \iC'$, we have $\theta_{[M,X]} = [\theta_M, \theta_X]$.
\end{lemma}
\begin{proof}
  We write $M$ as $M = \varinjlim_{i \in I} M_i$ for $M_i \in \C$. Then we have $[M,X] = \varprojlim_{i \in I} \, [M_i, X]$.
  If $W \in \C$, then we may, and do, identify $[W, X]$ with $X \otimes W^*$. Hence we have
  \begin{equation*}
    \theta_{[W, X]} = \theta_{X \otimes W^*}
    = c_{W^*,X}c_{X,W^*}(\theta_{X} \otimes \theta_{W^*})
    = \theta_X \otimes (\theta_{W})^*
    = [\theta_{W}, \theta_X],
  \end{equation*}
  where we have used the assumption $X \in \iC'$ at the third equality.
  Since $M_i \in \C$ for all $i$,
  \begin{equation*}
    \theta_{[M,X]}
    = \varprojlim_{i \in I} \theta_{[M_i, X]}
    = \varprojlim_{i \in I} \, [\theta_{M_i}, \theta_X]
    = [\varinjlim_{i \in I} \theta_{M_i}, \theta_X]
    = [\theta_M, \theta_X]. \qedhere
  \end{equation*}
\end{proof}

Let $A$ be a commutative Frobenius algebra in $\iC$. Then the category $\D := \fg\text{-}\iC_A^{\loc}$ of finitely-generated local $A$-modules is a braided monoidal category.

\begin{theorem}\label{thm:local-modules-ribbon}
  If $\theta_A = \id_A$ and $\D$ is rigid, then $\D$ is a ribbon category.
\end{theorem}

We explain that this theorem can be applied to simple current algebras in $\iC$ in characteristic zero in Section \ref{sec:simple-current-algebras}.
The condition $\theta_A = \id_A$ is analogous to triviality of the Nakayama automorphism in the finite-dimensional Frobenius case \cite[Proposition~2.25]{frohlich2006correspondences}. 

\begin{proof}
  By the assumption that $\theta_A = \id_A$, we can define a natural isomorphism $\theta^A_M : M \to M$ for $M \in \D$ by $\theta^A_M = \theta_M$ in the same way as \cite[Theorem 1.17 (2)]{kirillov2002q}. With the use of the internal Hom functor, the duality functor of $\D$ is given by $M \mapsto M^{\dagger} := \iHom_A(M, A)$.
  It remains to show that $(\theta^A_M)^{\dagger} = \theta^A_{M^{\dagger}}$ holds for all $M \in \D$. By Lemma~\ref{lem:free-right-adj-2}, there is a natural isomorphism $\phi_M : M^{\dagger} \to [M, \unit]$ for $M \in \iC_A$. Now we have
  \begin{equation*}
    \phi_M \circ \theta^A_{M^{\dagger}}
    = \theta^A_{[M, \unit]} \circ \phi_M
    = [\theta^A_M, \unit] \circ \phi_M
    = \phi_M \circ (\theta^A_M)^{\dagger}
  \end{equation*}
  for $M \in \D$, where the first equality follows from the naturality of $\theta$, the second from the previous lemma, and the last from the naturality of $\phi$. The proof is done.
\end{proof}

\section{Representation theory of simple current algebras}\label{sec:simple-current-algebras}
Throughout this section, $\C$ denotes a locally finite $\kk$-linear abelian closed monoidal category with simple unit; additional hypotheses (rigidity, braiding, etc.) will be imposed when needed.
We also fix a subgroup $\Gamma$ of the group $\Inv(\C)$ of invertible objects of $\C$ and, for each $g \in \Gamma$, choose an object $E_g \in \C$ representing $g$. In view of Lemma \ref{lem:group-algebra-comm} (a), we assume that there is a 2-cochain $\eta \in \mathrm{C}^2(\Gamma)$ such that $\partial_3(\eta) = \omega|_{\Gamma \times \Gamma \times \Gamma}$, where $\omega$ is the 3-cocycle on $\Inv(\C)$ arising from the associator, and construct a simple current algebra $A = \bigoplus_{g \in \Gamma} E_g$ by the 2-cochain $\eta$.
The aim of this section is to investigate the representation theory of $A$.
In this section, we verify that $A$ satisfies the hypotheses of Theorem~\ref{thm:main-criteria} under some assumptions on $\Gamma$ and $\C$.
As a consequence, we obtain criteria ensuring the category of finitely-generated local modules over $A$ is a ribbon (Frobenius) tensor category (Theorem \ref{thm:fg-local-simple-main}). 


\subsection{Classification of simple modules}
In this subsection, we give a classification of simple right modules of the algebra $A = \bigoplus_{g \in \Gamma} E_g$ in $\iC$. 
To do this, we study the restriction of $A$-modules to subalgebras of $A$, which is similar in spirit to Clifford theory. 

The invertible summands $E_g$ act on simples of $\iC$ by tensor product, so any simple $A$-module $M$ is supported on a single $\Gamma$-orbit of a simple constituent $X\subset M$. One then passes from the full algebra $A$ to the stabilizer subgroup $S=\Stab_\Gamma(X)$ and the corresponding subalgebra $A_X=\bigoplus_{s\in S}E_s$, which captures the part of the action that preserves $X$. The associativity constraints for an $A_X$-action on $X$ produce a canonical 2-cocycle $\xi_X$ on $S$; this cocycle measures the obstruction to choosing the identifications $X\otimes E_s\simeq X$ compatibly, and it forces the multiplicity space to carry a $\xi_X$-projective representation of $S$. Thus simple $A_X$-modules whose underlying object is isomorphic to $X \oplus \dotsb \oplus X$ are parametrized by irreducible $\xi_X$-projective representations $W$ of $S$, and the corresponding simple $A$-modules are obtained by induction $\mathbf{M}(X,W)=(W\otimes X)\otimes_{A_X}A$. Finally, isomorphism classes are obtained by the obvious orbit relation: twisting $X$ by $E_g$ replaces $(S,\xi_X,W)$ by the conjugate data, yielding the same induced $A$-module.


\subsubsection{Projective representations} 
We recall basics on projective representations of a group.
Let $S$ be a group and $\xi \in \mathrm{C}^2(S)$ a normalized 2-cocycle. A \emph{projective representation of $S$ with multiplier $\xi$} is a finite-dimensional vector space $V$ together with a map $\rho: S \to \mathrm{GL}(V)$ satisfying
\[
  \rho(s)\,\rho(t) = \xi(s,t)\,\rho(st) \quad \text{for all } s, t \in S.
\]
Morphisms are linear maps that intertwine the $S$-actions. We denote by $\mathrm{Rep}(S,\xi)$ the category of such representations. It is equivalent to the category of finite-dimensional modules over the twisted group algebra $\kk^{\xi}[S]$, which has $\kk$-basis $\{e_s\}_{s \in S}$ and multiplication $e_s \cdot e_t = \xi(s,t)\,e_{st}$. If $\xi$ and $\xi'$ are cohomologous, then $\mathrm{Rep}(S,\xi) \simeq \mathrm{Rep}(S,\xi')$, so the category depends only on the class $[\xi] \in \mathrm{H}^2(S)$.

\subsubsection{Modules over the stabilizer}\label{subsubsec:modules-over-stabilizer}

For a simple object $X \in \C$, we define
\begin{equation*}
  \Stab_{\Gamma}(X) = \{ g \in \Gamma \mid X \otimes E_g \cong X \}
\end{equation*}
and call it the {\em stabilizer} of $X$.
It is obvious that $\Stab_{\Gamma}(X)$ is a subgroup of $\Gamma$.
We fix a simple object $X \in \C$ and write $S = \Stab_{\Gamma}(X)$.
Later, we will construct a simple right $A$-module by induction from the subalgebra
\begin{equation*}
  A_X := \bigoplus_{s \in S} E_s
\end{equation*}
of $A$. We first remark that $A_X$ is in fact an algebra in $\C$. Namely,

\begin{lemma}
  The group $S$ is finite.
\end{lemma}
\begin{proof}
  We denote by $[-,-]^l$ the left internal Hom functor of $\C$ (thus $[X, -]^{l}$ is right adjoint to $X \otimes -$). If $g \in S$, then we have $\Hom_{\C}(E_g, [X, X]^{l}) \cong \Hom_{\C}(X \otimes E_g, X) \cong \Hom_{\C}(X, X) \ne 0$.
  This means that $E_g$ is a simple subobject of $[X, X]^{l}$. Hence the cardinality of $S$ is bounded by the length of the socle of $[X, X]^{l} \in \C$, which is at most finite.
\end{proof}

We choose a family $\tilde{\phi}_s : X \otimes E_s \to X$ ($s \in S$) of isomorphisms in $\C$ such that $\tilde{\phi}_1$ is the unit isomorphism, and define the map $\eta_X : S \times S \to \kk^{\times}$ so that the following diagram is commutative:
\begin{equation}
  \label{eq:eta-X-def}
  \begin{tikzcd}[column sep = 48pt]
    (X \otimes E_s) \otimes E_t
    \arrow[r, "{\tilde{\phi}_s \otimes \id}"]
    \arrow[d, "{\alpha_{X,E_s,E_t}}"']
    & X \otimes E_t
    \arrow[r, "{\tilde{\phi}_t}"]
    & X \arrow[d, "{\eta_X(s, t) \, \id}"] \\
    X \otimes (E_s \otimes E_t)
    \arrow[r, "{\id \otimes \phi_{s,t}}"]
    & X \otimes E_{s t}
    \arrow[r, "{\tilde{\phi}_{s t}}"]
    & X
  \end{tikzcd}
\end{equation}
for all $s, t \in S$.

\begin{lemma}
  $\partial_3(\eta_X^{-1}) = \omega|_{S \times S \times S}$.
\end{lemma}
\begin{proof}
  There is a commutative diagram given by Figure \ref{fig:eta-X-coboundary-condition} (in the figure, the tensor product of objects of $\C$ is expressed by juxtaposition to save space). The composition along the first column is equal to the identity morphism by the pentagon axiom. Thus we have
  \begin{equation}
    \label{eq:eta-X-coboundary-condition}
    \eta_X(s, t)^{-1}
    \eta_X(s t, u)^{-1}
    \omega(s, t, u)^{-1}
    \eta_X(s, t u) \eta_X(t, u) = 1
    \quad (s, t, u \in S),
  \end{equation}
  which implies the claim (see the definition \eqref{eq:coboundary-map-2} of $\partial_3$).
\end{proof}

\begin{figure}
  \begin{tikzcd}[column sep = 48pt]
    ((X E_s) E_t) E_u
    \arrow[r, "{(\tilde{\phi}_s \otimes \id) \otimes \id}"]
    \arrow[dd, "{\alpha_{X E_s, E_t, E_u}}"']
    & (X E_t) E_u
    \arrow[r, "{\tilde{\phi}_t \otimes \id}"]
    \arrow[d, "{\alpha_{X, E_t, E_u}}"]
    & X E_u
    \arrow[dd, phantom, "\text{\scriptsize Diagram \eqref{eq:eta-X-def}}"]
    \arrow[r, "{\tilde{\phi}_{u}}"]
    & X
    \arrow[dd, "{\eta_X(t, u) \, \id}"] \\
    & X (E_t E_u)
    \arrow[rd, "{\id \otimes \phi_{t,u}}"] \\[-20pt]
    (X E_s) (E_t E_u)
    \arrow[ru, "{\tilde{\phi}_s \otimes (\id \otimes \id)}"]
    \arrow[rd, "{(\id \otimes \id) \otimes \phi_{t,u}}"']
    \arrow[dd, "{\alpha_{X, E_s, E_t E_u}}"']
    & & X E_{t u}
    \arrow[r, "{\tilde{\phi}_{t u}}"]
    \arrow[dd, phantom, "\text{\scriptsize Diagram \eqref{eq:eta-X-def}}"]
    & X
    \arrow[dd, "{\eta_X(s, t u) \, \id}"] \\[-20pt]
    & (X E_s) E_{t u}
    \arrow[ru, "{\tilde{\phi}_s \otimes \id}"']
    \arrow[d, "{\alpha_{X, E_s, E_{t u}}}"] \\
    X (E_s (E_t E_u))
    \arrow[r, "{\id \otimes (\id \otimes \phi_{t,u})}"]
    \arrow[d, "{\id \otimes \alpha_{E_s, E_t, E_u}^{-1}}"']
    \arrow[rrd, phantom, "\text{\scriptsize Diagram \eqref{eq:omega-def}}"]
    & X (E_s E_{t u})
    \arrow[r, "{\id \otimes \phi_{s, t u}}"]
    & X E_{s t u}
    \arrow[r, "{\tilde{\phi}_{s t u}}"]
    \arrow[d, "{\omega(s, t, u)^{-1} \, \id}"]
    & X
    \arrow[d, "{\omega(s, t, u)^{-1} \, \id}"] \\
    X ((E_s E_t) E_u)
    \arrow[r, "{\id \otimes (\phi_{s t} \otimes \id)}"]
    \arrow[d, "{\alpha_{X, E_s E_t, E_u}^{-1}}"']
    & X (E_{s t} E_u)
    \arrow[r, "{\id \otimes \phi_{s t, u}}"]
    \arrow[d, "{\alpha_{X, E_{s t}, E_u}^{-1}}"']
    \arrow[rrd, phantom, "\text{\scriptsize Diagram \eqref{eq:eta-X-def}}"]
    & X E_{s t u}
    \arrow[r, "{\tilde{\phi}_{s t u}}"]
    & X
    \arrow[d, "{\eta_X(s t, u)^{-1} \, \id}"] \\
    (X (E_s E_t)) E_u
    \arrow[d, "{\alpha_{X,E_s,E_t}^{-1} \otimes \id}"']
    \arrow[r, "{(\id \otimes \phi_{s, t}) \otimes \id}"]
    \arrow[rrd, phantom, "\text{\scriptsize Diagram \eqref{eq:eta-X-def}}"]
    & (X E_{s t}) E_u
    \arrow[r, "{\tilde{\phi}_{s t} \otimes \id}"]
    & X E_u
    \arrow[d, "{\eta_X(s,t)^{-1}\,\id}"]
    \arrow[r, "{\tilde{\phi}_{u}}"]
    & X
    \arrow[d, "{\eta_X(s,t)^{-1}\,\id}"] \\
    ((X E_s) E_t) E_u
    \arrow[r, "{(\tilde{\phi}_s \otimes \id) \otimes \id}"]
    & (X E_t) E_u
    \arrow[r, "{\tilde{\phi}_t \otimes \id}"]
    & X E_u
    \arrow[r, "{\tilde{\phi}_{u}}"]
    & X
  \end{tikzcd}
  \caption{Proof of Equation~\eqref{eq:eta-X-coboundary-condition}}
  \label{fig:eta-X-coboundary-condition}
\end{figure}

By this lemma, we have a 2-cocycle $\xi_X$ on $S$ given by
\begin{equation}
  \label{eq:xi-X-def}
  \xi_X(s, t) = \eta(s,  t) \cdot \eta_X(s, t)
  \quad (s, t \in S).
\end{equation}

We call $\xi_X$ {\em the 2-cocycle associated to $\{ \tilde{\phi}_s \}_{s \in S}$}.
The 2-cocycle $\xi_X$ depends on the choice of the family $\{ \tilde{\phi}_s \}_{s \in S}$ of isomorphisms (and, in fact, also depends on choices of $\phi_{s,t}$ and $\eta$); however, its cohomology class depends only on $X$.

Now we consider the tensor category $\vect^S$ of finite-dimensional $S$-graded vector spaces over $\kk$. Let $\kk_s \in \vect^S$ denote the one-dimensional vector space $\kk$ graded by $s \in S$. We denote by $\M_X$ the category $\vect$ made into a right module category over $\vect^S$ by the action given by $V \triangleleft W  = V \otimes W$ for $V \in \vect$ and $W \in \vect^S$ and the module associator $\tilde{\alpha}$ determined by
\begin{equation}
  \label{eq:module-assoc-def-1}
  \tilde{\alpha}_{V, \kk_s, \kk_t} = \xi_X(s, t) \, \id_V : (V \triangleleft \kk_s) \triangleleft \kk_t \to V \triangleleft (\kk_s \otimes \kk_t)
\end{equation}
for $s, t \in S$.

There is a unique (up to isomorphism) linear functor $\mathcal{E}: \vect^S \to \C$ such that $\mathcal{E}(\kk_s) = E_s$ for all elements $s \in S$. The functor $\mathcal{E}$ is a monoidal functor with the monoidal structure induced by the morphism $\mu_{s,t}$ given by \eqref{eq:mu-def}. Hence $\C$ is a right module category over $\vect^S$ through the monoidal functor $\mathcal{E}$. We note that the module associator $\tilde{\alpha}$ of $\C$ is given by
\begin{equation}
  \label{eq:module-assoc-def-2}
  \tilde{\alpha}_{V, \kk_s, \kk_t}
  = (\id_V \otimes \mu_{s,t}) \circ \alpha_{V, E_s, E_t}:
  (V \triangleleft \kk_s) \triangleleft \kk_t
  \to V \triangleleft (\kk_s \otimes \kk_t)
\end{equation}
for $V \in \C$ and $s, t \in S$.

From now on, we view $\vect$ as a full subcategory of $\C$ by identifying the unit object of $\C$ with the vector space $\kk$ so that $V \otimes Y$ for $V \in \vect$ and $Y \in \C$ makes sense.

\begin{lemma}
  \label{lem:module-functor-FX}
  The functor
  \begin{equation*}
    \mathcal{F}_X : \M_X \to \C,
    \quad \mathcal{F}_X(W) = W \otimes X
  \end{equation*}
  is a right $\vect^S$-module functor by the module structure $f$ determined by
  \begin{equation}
    \label{eq:module-functor-FX-def-2}
    f_{W,\kk_s} = (\id_{\mathcal{E}(W)} \otimes \tilde{\phi}_s) \circ \alpha_{W, X, E_s}
    : \mathcal{F}_X(W) \triangleleft \kk_s
    \to \mathcal{F}_X(W \triangleleft \kk_s)
  \end{equation}
  for $s \in S$.
\end{lemma}
\begin{proof}
  It is easy to see $f_{W, \kk_1} = \id_{W} \otimes \id_X$ for $W \in \M_X$. To complete the proof, it suffices to show that the equation
  \begin{equation}
    \label{eq:lem-module-functor-FX-proof-1}
    \mathcal{F}_X(\tilde{\alpha}_{W, \kk_s, \kk_t})
    \circ f_{W \triangleleft \, \kk_s, \kk_t} \circ (f_{W,\kk_s} \triangleleft \kk_t)
    = f_{W, \kk_s \otimes \kk_t} \circ \tilde{\alpha}_{\mathcal{F}_X(W), E_s, E_t}
  \end{equation}
  holds for all $W \in \M_X$ and $s, t \in S$.
  We consider the commutative diagram given by Figure~\ref{fig:proof-lem-module-functor-FX}. 
  By \eqref{eq:module-assoc-def-1} and \eqref{eq:module-functor-FX-def-2}, the left hand side of \eqref{eq:lem-module-functor-FX-proof-1} is equal to the composition of morphisms along the counter-clockwise path from the upper left corner to the bottom right one in the diagram of Figure~\ref{fig:proof-lem-module-functor-FX}. By \eqref{eq:module-assoc-def-2} and \eqref{eq:module-functor-FX-def-2}, the clockwise path yields the right hand side of \eqref{eq:lem-module-functor-FX-proof-1}. Thus we obtain \eqref{eq:lem-module-functor-FX-proof-1}.
\end{proof}

\begin{figure}
  \begin{tikzcd}[column sep = 32pt]
    ((W X) E_s) E_t
    \arrow[rr, "{\alpha_{W \otimes X, E_s, E_t}}"' { yshift = -5pt }]
    \arrow[d, "{\alpha_{W, X, E_s} \otimes \id}"]
    & & (W X) (E_s E_t)
    \arrow[r, "{(\id \otimes \id) \otimes \mu_{s,t}}"' { yshift = -5pt }]
    \arrow[dd, "{\alpha_{W, X, E_s \otimes E_t}}"]
    \arrow[ddl, phantom, "{\text{\scriptsize (pentagon axiom)}}"]
    & (W X) E_{s t}
    \arrow[dd, "{\alpha_{W,X,E_{s t}}}"] \\
    (W (X E_s)) E_t
    \arrow[rd, "{\alpha_{W, X \otimes E_s, E_t}}"]
    \arrow[d, "{(\id \otimes \tilde{\phi}_s) \otimes \id}"'] \\
    (W X) E_t
    \arrow[d, "{\alpha_{W,X,E_t}}"']
    & W ((X E_s) E_t)
    \arrow[dl, "{\id \otimes (\tilde{\phi}_s \otimes \id)}"]
    \arrow[r, "{\id \otimes \alpha_{X,E_s, E_t}}" { yshift = 5pt }]
    & W (X (E_s E_t))
    \arrow[r, "{\id \otimes (\id \otimes \mu_{s, t})}" { yshift = 5pt }]
    \arrow[d, "{\id \otimes (\id \otimes \phi_{s,t})}"']
    \arrow[rd, "{\text{\scriptsize \eqref{eq:mu-def}}}", phantom]
    & W (X E_{s t}) \arrow[d, equal] \\
    W (X E_t)
    \arrow[d, "{\id \otimes \tilde{\phi}_t}"']
    \arrow[rr, phantom, "{\text{\scriptsize Diagram \eqref{eq:eta-X-def}}}" {yshift=-12pt}]
    & & W (X E_{s t})
    \arrow[d, "{\id \otimes \tilde{\phi}_{s t}}"']
    \arrow[r, "{\eta(s, t) \, \id}"]
    & W (X E_{s t})
    \arrow[d, "{\id \otimes \tilde{\phi}_{s t}}"] \\
    W X \arrow[rr, "{\eta_X(s, t) \, \id}"]
    & & W X \arrow[r, "{\eta(s, t) \, \id}"] & W X
  \end{tikzcd}
  \caption{Proof of Lemma~\ref{lem:module-functor-FX}}
  \label{fig:proof-lem-module-functor-FX}
\end{figure}

Given a right module category $\M$ over a monoidal category $\D$ and an algebra $R$ in $\D$ with multiplication $m$ and unit $u$, the category of right $R$-modules in $\M$ is defined. We recall that an object of this category is an object $M \in \M$ together with a morphism $a : M \triangleleft R \to M$ such that the following diagrams are commutative:
\begin{equation*}
  \begin{tikzcd}
    (M \triangleleft R) \triangleleft R
    \arrow[d, "{\cong}"']
    \arrow[r, "{a \triangleleft \id}"]
    & M \triangleleft R \arrow[r, "{a}"]
    & M \arrow[d, "{\id}"] \\
    M \triangleleft (R \otimes R)
    \arrow[r, "{\id \triangleleft m}"]
    & M \triangleleft R \arrow[r, "{a}"]
    & M 
  \end{tikzcd}
  \quad
  \begin{tikzcd}
    M \triangleleft 1
    \arrow[r, "{\id \triangleleft u}"]
    \arrow[d, "{\cong}"']
    & M \triangleleft R
    \arrow[d, "{a}"] \\
    M \arrow[r, "{\id}"]
    & M.
  \end{tikzcd}
\end{equation*}

We consider the group algebra $R = \kk S$, which is an algebra in $\vect^S$ by the grading given by $R_s = \kk s$ ($s \in S$). When we view $\C$ as a right module category over $\vect^S$ via the monoidal functor $\mathcal{E} : \vect^S \to \C$, the category of right $R$-modules in $\C$ is identical to $\C_{A_X}$ since $\mathcal{E}(\kk S) = A_X$ as algebras. Taking the formula \eqref{eq:module-assoc-def-2} of the module associator into account, we see that a right $R$-module in $\M_X$ is the same thing as a finite-dimensional vector space $W$ together with a family $\{ \rho(s) \}_{s \in S}$ of linear endomorphisms on $W$ such that the equations
\begin{equation*}
  \rho(1) = \id_W
  \quad \text{and} \quad
  \xi_X(s, t) \rho(s t) = \rho(t) \rho(s)
\end{equation*}
hold for all $s, t \in S$. Therefore the category of right $R$-modules in $\M_X$ is identified with the category $\Rep(S^{\op}, \xi_X^{\op})$ of projective representations of $S^{\op}$ with multiplier $\xi^{\op}_X$, where $\xi^{\op}_X(t, s) = \xi_X(s, t)$ ($s, t \in S$). We use this observation to prove:

\begin{lemma}
  \label{lem:functor-LX}
  Given $\rho \in \Rep(S^{\op}, \xi_X^{\op})$ with representation space $W$, we define
  \begin{equation*}
    \mathbf{L}(X, \{ \tilde{\phi}_s \}_{s \in S}, \rho) = W \otimes X
  \end{equation*}
  and make it a right $A_X$-module in $\C$ by the action induced by
  \begin{equation*}
    (W \otimes X) \otimes E_s
    \xrightarrow{\ \alpha_{W, X, E_s} \ }
    W \otimes (X \otimes E_s)
    \xrightarrow{\ \rho(s) \otimes \tilde{\phi}_s \ }
    W \otimes X.
  \end{equation*}
  This construction gives rise to a functor
  \begin{equation*}
    \mathbf{L}(X, \{ \tilde{\phi}_s \}_{s \in S}, -) : \Rep(S^{\op}, \xi_X^{\op}) \to \C_{A_X},
  \end{equation*}
  which induces an equivalence from $\Rep(S^{\op}, \xi_X^{\op})$ to the category of right $A_X$-modules in $\C$ whose underlying object is isomorphic to the direct sum of finitely many copies of $X$.
\end{lemma}
\begin{proof}
  The right $\vect^S$-module functor $\mathcal{F}_X : \M_X \to \C$ induces a functor from the category of right $R$-modules in $\M_X$ to the category of those in $\C$, and the induced functor is nothing but the functor $\mathbf{L}(X, \{ \tilde{\phi}_s \}_{s \in S}, -)$ in question. The last claim of this lemma follows because $\mathcal{F}_X$ induces an equivalence from $\M_X$ to the full subcategory of $\C$ consisting of finite direct sums of copies of $X$.
\end{proof}

Let $(\rho, W)$ be as in Lemma~\ref{lem:functor-LX}, and let $(\rho', W')$ be a projective representation of $S^{\op}$ that is projectively equivalent to $\rho$. By definition, there are an isomorphism $T : W \to W'$ and a map $\theta : S^{\op} \to \kk^{\times}$ such that $\rho'(s) = \theta(s) T \rho(s) T^{-1}$ for all $s \in S^{\op}$. Since the multiplier of $\rho'$ is
\begin{equation*}
  \xi'(t, s) = \theta(s) \theta(t) \theta(s t)^{-1} \xi_X(s, t)
  \quad (s, t \in S^{\op})
\end{equation*}
and may not be equal to $\xi_X^{\op}$, the right $A_X$-module $\mathbf{L}(X, \{ \tilde{\phi}_s \}_{s \in S}, \rho')$ may not be defined. However, we can use the map $\theta$ to introduce a new family
\begin{equation*}
  \tilde{\phi}'_s = \theta(s)^{-1} \tilde{\phi}_s : X \otimes E_s \to X
  \quad (s \in S)
\end{equation*}
of isomorphisms. Then the right $A_X$-module $\mathbf{L}(X, \{ \tilde{\phi}'_s \}_{s \in S}, \rho')$ is defined and
\begin{equation*}
  \mathbf{L}(X, \{ \tilde{\phi}'_s \}_{s \in S}, \rho')
  \cong \mathbf{L}(X, \{ \tilde{\phi}_s \}_{s \in S}, \rho)
\end{equation*}
as right $A_X$-modules. Thus we introduce the following notation:

\begin{definition}
  Given a projective representation $\rho'$ of $S^{\op}$ whose multiplier is cohomologous to the 2-cocycle $\xi_X^{\op}$, we define $\mathbf{L}(X, \rho') := \mathbf{L}(X, \{ \tilde{\phi}'_{s} \}_{s \in S}, \rho')$, where $\{ \tilde{\phi}'_s : X \otimes E_{s} \to X \}_{s \in S}$ is a family of isomorphisms in $\C$ such that the associated 2-cocycle is equal to the multiplier of $\rho'$.
\end{definition}


\subsubsection{Construction of simple modules}

Let $X$ be a simple object of $\C$, and let $S$ be the stabilizer of $X$.
We fix a family $\tilde{\phi}_s : X \otimes E_s \to X$ ($s \in S$) of isomorphisms and define $\xi_X$ by \eqref{eq:xi-X-def}. 

\begin{definition}
  We define the following functor
  \begin{equation*}
      \mathbf{M}(X, -) : \Rep(S^{\op}, \xi_X^{\op}) \to \iC_A,
      \quad \rho \mapsto \mathbf{L}(X, \rho) \otimes_{A_X} A.
  \end{equation*}
\end{definition}

\begin{lemma} \label{lem:functor-MX}
    The functor $\mathbf{M}(X, -)$ introduced in the above is exact and faithful.    
\end{lemma}
\begin{proof}
The functor $\mathbf{M}(X, -)$ is decomposed as follows:
\begin{equation*}
    \Rep(S^{\op}, \xi_X^{\op})
    \xrightarrow{\quad \mathbf{L}(X, -) \quad}
    \mathcal{X}
    \xrightarrow{\quad i \quad}
    \iC_{A_X}
    \xrightarrow{\quad F \quad}
    \iC_A,
\end{equation*}
where $\mathcal{X}$ is the category of right $A_X$-modules in $\C$ whose underlying object is isomorphic to a direct sum of finitely many copies of $X$, $i$ is the inclusion functor, and $F = - \otimes_{A_X} A$ is the induction functor.
By Lemma \ref{lem:functor-LX}, the first arrow is exact and faithful. It is obvious that the second arrow is also exact and faithful.
Since $A$ is free over $A_X$, the functor $F$ is also exact and faithful. Thus $\mathbf{M}(X, -)$ is exact and faithful as the composition of such functors.
\end{proof}

We fix coset representatives $\{ g_i \}_{i \in I}$ of $S \backslash \Gamma$. 
Then we have an isomorphism
\begin{equation*}
  A = \bigoplus_{i \in I} \bigoplus_{s \in S} E_{s g_i}
  \xrightarrow[\cong]{\quad \bigoplus_{i \in I} \bigoplus_{s \in S} \mu_{s,g_i}^{-1} \quad}
  \bigoplus_{i \in I} \bigoplus_{s \in S} E_{s} \otimes E_{g_i}
  = \bigoplus_{i \in I} A_X \otimes E_{g_i}
\end{equation*}
of left $A_X$-modules. Thus we have
\begin{equation}
  \label{eq:M-X-rho-decomposition}
  \mathbf{M}(X, \rho)
  \cong \bigoplus_{i \in I} (W \otimes X) \otimes E_{g_i}
\end{equation}
as an object of $\iC$, where $W$ is the representation space of $\rho$.
The action of $A$ on the right hand side induced by \eqref{eq:M-X-rho-decomposition} is given by
\begin{equation}
  \label{eq:M-X-rho-action}
  \begin{aligned}
    ((W X) E_{g_i}) E_{\gamma}
    & \cong (W X) (E_{g_i} E_{\gamma})
    \xrightarrow{\makebox[9em][c]{\scriptsize
        $(\id \otimes \id) \otimes \mu_{s,g_j}^{-1} \mu_{g_i, \gamma}^{}$}}
    (W X) (E_{s} E_{g_j}) \\
    & \cong (W (X E_{s})) E_{g_j}
    \xrightarrow{\makebox[9em][c]{\scriptsize
        $(\rho(s) \otimes \tilde{\phi}_s) \otimes \id$}}
    (W X) E_{g_j}
  \end{aligned}
\end{equation}
for $\gamma \in \Gamma$, where the tensor product of objects is expressed by juxtaposition, $s \in S$ and $j \in I$ are elements determined by $s g_j = g_i \gamma$, and $\cong$'s are canonical isomorphisms obtained by the associator. We now give some results on the structure of $\mathbf{M}(X, \rho)$.

\begin{lemma}\label{lem:M-X-rho-Eg}
The right $A$-module $\mathbf{M}(X, \rho)$ is isomorphic to $A$ if and only if $X \cong E_g$ for some $g\in \Gamma$.
\end{lemma}
\begin{proof}
Consider $X=E_{g}$ for some $g\in \Gamma$. Then, the stabilizer in $\Gamma$ is trivial, hence $A_X=\unit$ and $\mathbf{M}(E_{g},\rho)\ \cong\ E_{g}\otimes A$ as right $A$-modules (necessarily with $\rho$ the trivial one-dimensional representation).
Finally, multiplication in $A$ induces an isomorphism $E_{g} \otimes A \cong A$ of right $A$-modules, since it identifies each summand $E_{g}\otimes E_g'\cong E_{gg'}$ with the corresponding summand of $A$.

For the converse, we assume that $\mathbf{M}(X, \rho) \cong A$ in $\iC_A$ and let $Y$ be a simple subobject of $\mathbf{M}(X, \rho)$  as an object of $\iC$. Then $Y \cong E_{h}$ for some $h \in \Gamma$ by the definition of $A$, while $Y \cong X \otimes E_{g_i}$ for some $i \in I$ by the decomposition \eqref{eq:M-X-rho-decomposition}. Thus $X \cong E_h \otimes E_{g_i}^* \cong E_{h g_i^{-1}}$, as desired.
\end{proof}

\begin{lemma}\label{lem:M-X-rho-simple}
  The right $A$-module $\mathbf{M}(X, \rho)$ in $\iC$ is simple if and only if the projective representation $\rho$ is irreducible.
\end{lemma}
\begin{proof}
  We assume that $\rho$ is irreducible. Let $M'$ be a non-zero submodule of $\mathbf{M}(X, \rho)$, and let $Y$ be a simple subobject of $M'$ as an object of $\iC$. Then $Y \cong X \otimes E_{g_i}$ for some $i \in I$ and by the same argument as in Lemma~\ref{lem:M-X-rho-Eg}, we find that $X$ is also a subobject of $M'$. 
  Since $\rho$ is irreducible, $M'$ contains the subobject corresponding to $(W \otimes X) \otimes E_1$ in the decomposition \eqref{eq:M-X-rho-decomposition}. By \eqref{eq:M-X-rho-action}, it is easy to see that $M'$ contains $(W \otimes X) \otimes E_{g_j}$ for all $j \in I$. Thus $\mathbf{M}(X, \rho)$ is simple.

  We now consider the case where $\rho$ is reducible. Then there is an exact sequence $0 \to \rho' \to \rho \to \rho'' \to 0$ in the category $\Rep(S^{\op}, \xi_X^{\op})$ with $\rho'$ and $\rho''$ non-zero. 
  By Lemma \ref{lem:functor-MX}, we have an exact sequence $0 \to \mathbf{M}(X, \rho') \to \mathbf{M}(X, \rho) \to \mathbf{M}(X, \rho'') \to 0$ in $\iC_A$ with $\mathbf{M}(X, \rho')$ and $\mathbf{M}(X, \rho'')$ non-zero. This means that $\mathbf{M}(X, \rho)$ is not a simple right $A$-module. The proof is done.
\end{proof}

\begin{lemma} \label{lem:functor-MX-preserves-length}
    The functor $\mathbf{M}(X, -)$ preserves the length.
\end{lemma}
\begin{proof}
    Lemma \ref{lem:M-X-rho-simple} implies that $\mathbf{M}(X, -)$ sends a simple object to a simple object. By using Lemma \ref{lem:functor-MX}, one can prove that $\mathbf{M}(X, \rho)$ and $\rho$ have the same length by induction on the length of a projective representation $\rho$.
\end{proof}


\subsubsection{Classification of simple modules}

\begin{theorem}
  \label{thm:classification-1}
  Let $M$ be a simple right $A$-module in $\iC$, and let $X$ be a simple subobject of $M$ as an object of $\iC$.
  Then, $M$ is isomorphic to $\mathbf{M}(X, \rho)$ for some irreducible projective representation $\rho$ of $\Stab_{\Gamma}(X)^{\op}$ with multiplier $\xi_X^{\op}$.
\end{theorem}
\begin{proof}
  We choose a simple subobject $X$ of $M$ in $\iC$.
  Let $L$ be the image of $X \otimes A_X$ under the action $M \otimes A \to M$. Then $L$ is a non-zero $A_X$-submodule of $M$. We choose a simple $A_X$-submodule $L'$ of $L$. Since $X \otimes A_X$ is the finite direct sum of copies of $X$, so is $L$, and therefore so is $L'$. Thus, by Lemma~\ref{lem:functor-LX}, the right $A_X$-module $L'$ is isomorphic to $\mathbf{L}(X, \rho)$ for some irreducible $\rho \in \Rep(S^{\op}, \xi_X^{\op})$. By the universal property of the induction, there is a non-zero morphism of right $A$-modules, say $f$, from $\mathbf{M}(X, \rho)$ to $M$. Since both $\mathbf{M}(X, \rho)$ and $M$ are simple, the morphism $f$ is in fact an isomorphism by Schur's lemma. The proof is done.
\end{proof}

Given an object $M \in \iC$ and a simple object $X \in \C$, we denote by $M|_X$ the sum of all subobjects of $M$ isomorphic to $X$. Let $M$ be a simple right $A$-module in $\iC$, and let $X$ be a simple subobject of $M$ in $\iC$. As shown in the above theorem, $M$ is isomorphic to $\mathbf{M}(X, \rho)$ for some $\rho \in \Rep(S^{\op}, \xi_X^{\op})$. In view of the decomposition \eqref{eq:M-X-rho-decomposition} and the proof of the above theorem, the subobject $M|_X$ is an $A_X$-submodule of $M$ such that $M|_X \cong \mathbf{L}(X, \rho)$, and thus we can recover the projective equivalence class of $\rho$ from $M|_X$ by Lemma~\ref{lem:functor-LX}. We use this observation to establish:

\begin{theorem}
  \label{thm:classification-2}
  $\mathbf{M}(X, \rho)$ and $\mathbf{M}(X', \rho')$ are isomorphic as right $A$-modules if and only if there exists an element $g \in \Gamma$ such that $X'$ is isomorphic to $X \otimes E_g$ and $\rho'$ is projectively equivalent to $\rho^g$.
\end{theorem}

Here, $\rho^g$ is defined by $\rho^g(s') = \rho(g s' g^{-1})$ for $s' \in S'$.
We assume that $X'$ is isomorphic to $X \otimes E_g$ and write $S = \Stab_{\Gamma}(X)$ and $S' = \Stab_{\Gamma}(X')$. It is easy to see that $S' = g^{-1} S g$. Thus $\rho^g$ is a well-defined projective representation of $S'{}^{\op}$.

\begin{proof}
  The `if' part is easy. We prove the `only if' part.
  We assume that $\mathbf{M}(X, \rho)$ and $\mathbf{M}(X', \rho')$ are isomorphic as right $A$-modules. Then, $X'$ is a subobject of $\mathbf{M}(X, \rho)$. Thus, by \eqref{eq:M-X-rho-decomposition}, $X'$ is isomorphic to $X \otimes E_g$ for some $g \in \Gamma$. By the discussion preceding this theorem, we have isomorphisms
  \begin{equation}
    \label{eq:proof-thm-classification-2-eq-1}
    \mathbf{L}(X', \rho')
    \cong \mathbf{M}(X', \rho')|_{X'}
    \cong \mathbf{M}(X, \rho)|_{X'}
  \end{equation}
  of right $A_{X'}$-modules in $\C$.
  Now we fix an isomorphism $\psi : X \otimes E_g \to X'$ and consider the diagram given by Figure~\ref{fig:proof-thm-classification-2}, where $s' \in S'$ and $s = g s' g^{-1} \in S$ so that $s g = g s'$. The diagram is commutative except the cell labeled ($\heartsuit$). We can, and do, choose isomorphisms $\tilde{\phi}'_{s'} : X' \otimes E_{s'} \to X'$ ($s' \in S'$) and put it on the dashed arrow of the diagram so that the cell ($\heartsuit$) is commutative. In view of \eqref{eq:M-X-rho-action}, the left column of the diagram is the action of $A_{X'}$ on $\mathbf{M}(X, \rho)|_{X'}$. Thus, by the diagram, we obtain an isomorphism
  \begin{equation*}
    (\id_W \otimes \psi) \alpha_{W,X,E_g} : \mathbf{M}(X, \rho)|_{X'}
    \to \mathbf{L}(X', \{ \tilde{\phi}'_{s'} \}_{s' \in S'}, \rho^g)
  \end{equation*}
  of right $A_{X'}$-modules. By \eqref{eq:proof-thm-classification-2-eq-1} and Lemma~\ref{lem:functor-LX}, we conclude that $\rho^g$ is projectively equivalent to $\rho'$. The proof is done.
\end{proof}

\begin{figure}
  \begin{tikzcd}[column sep = 64pt]
    ((W X) E_g) E_{s'}
    \arrow[r, "{\alpha_{W,X,E_g} \otimes \id}"]
    \arrow[d, "{\alpha_{W \otimes X, E_g, E_{s'}}}"']
    \arrow[rd, phantom, "{\text{\scriptsize (pentagon axiom)}}" {yshift = -5pt}]
    & (W (X E_g)) E_{s'}
    \arrow[r, "{(\id \otimes \psi) \otimes \id}"]
    \arrow[d, "{\alpha_{W, X \otimes E_g, E_{s'}}}"]
    & (W X') E_{s'}
    \arrow[d, "{\alpha_{W,X',E_{s'}}}"] \\
    (W X) (E_g E_{s'})
    \arrow[rd, "{\alpha_{W, X, E_g \otimes E_{s'}}}" {xshift = -10pt}]
    \arrow[d, "{(\id \otimes \id) \otimes \mu_{g,s'}}"']
    & W ((X E_g) E_{s'})
    \arrow[r, "{\id \otimes (\psi \otimes \id)}"]
    \arrow[d, "{\id \otimes \alpha_{X,E_g,E_{s'}}}"]
    \arrow[rddddd, phantom, "{(\heartsuit)}"]
    & W (X' E_{s'})
    \arrow[ddddd, dashed] \\
    (W X) E_{g s'}
    \arrow[d, "{(\id \otimes \id) \otimes \mu_{s,g}^{-1}}"']
    \arrow[rd, "{\alpha_{W,X,E_{g s'}}}" {xshift = -10pt}]
    & W (X (E_g E_{s'}))
    \arrow[d, "{\id \otimes (\id \otimes \mu_{g, s'})}"] \\
    (W X) (E_s E_g)
    \arrow[rd, "{\alpha_{W, X, E_s \otimes E_g}}" {xshift = -10pt}]
    \arrow[d, "{\alpha_{W \otimes X, E_g, E_s}^{-1}}"']
    & W (X E_{g s'})
    \arrow[d, "{\id \otimes \mu_{s,g}^{-1}}"] \\
    ((W X) E_s) E_g
    \arrow[d, "{\alpha_{W,X,E_s} \otimes \id}"']
    \arrow[rd, phantom, "{\text{\scriptsize (pentagon axiom)}}" {yshift = 5pt}]
    & W (X (E_s E_g))
    \arrow[d, "{\id \otimes \alpha_{X, E_s, E_g}^{-1}}"] \\
    (W (X E_s)) E_g
    \arrow[r, "{\alpha_{W, X \otimes E_s, E_g}}"]
    \arrow[dd, "{\rho(s) \otimes \tilde{\phi}_s}"']
    & W ((X E_s) E_g)
    \arrow[d, "{\id \otimes (\tilde{\phi}_s \otimes \id)}"] \\
    & W (X E_g)
    \arrow[d, "{\rho(s) \otimes \id}"]
    \arrow[r, "{\id \otimes \psi}"]
    & W X' \arrow[d, "{\rho(s) \otimes \id}"] \\
    (W X) E_g
    \arrow[r, "{\alpha_{W,X,E_g}}"]
    & W (X E_g)
    \arrow[r, "{\id \otimes \psi}"]
    & W X'
  \end{tikzcd}
  \caption{Proof of Theorem~\ref{thm:classification-2}}
  \label{fig:proof-thm-classification-2}
\end{figure}


\subsubsection{Finiteness of the number of simple modules}

We give a formula of the number of isomorphism classes of simple objects of $\iC_A$ and, in particular, determine when it is finite. Given an abelian category $\mathcal{A}$, we denote by $\Irr(\mathcal{A})$ the set of isomorphism classes of simple objects of $\mathcal{A}$. The group $\Gamma$ acts on $\Irr(\C)$ by $[X] \cdot g = [X \otimes E_g]$ for $g \in \Gamma$ and $[X] \in \Irr(\mathcal{C})$ (this action is well-defined since $(-)\otimes E_g$ is an autoequivalence and preserves simplicity).

\begin{corollary}\label{cor:number-of-simples}
For brevity, we set $I_X = \Irr\bigl(\Rep(\Stab_\Gamma(X)^{\op},\,\xi_X^{\op})\bigr)$ for a simple object $X \in \C$, where $\xi_X$ is the $2$-cocycle determined by \eqref{eq:xi-X-def}. Then there is a bijection
\[ \bigsqcup_{X \in \cO} I_X \to \Irr(\iC_A),
\quad [\rho] \mapsto [\mathbf{M}(X, \rho)] \quad ([\rho] \in I_X), \]
where $\cO$ is the set of orbit representatives of $\Irr(\C) / \Gamma$.    
\end{corollary}

The bijection implies $\#\,\Irr(\iC_A) = \sum_{[X]\in \cO}\;\# I_X$. Since the set $I_X$ is non-empty and finite for all simple $X \in \C$ (since the stabilizer $\Stab_\Gamma(X)$ is finite), we have that $\Irr(\iC_A)$ is finite if and only if $\Irr(\C)/\Gamma$ is finite.

\begin{proof}
We first prove that the map is surjective.
By Theorem~\ref{thm:classification-1}, every simple right $A$-module is isomorphic to $\mathbf{M}(X',\rho')$ for some simple $X' \in \C$ and some $[\rho'] \in I_{X'}$. By the definition of the action of $\Gamma$ on $\Irr(\C)$, there exists $X \in \mathcal{O}$ and $g \in \Gamma$ such that $X' \cong X \otimes E_g$. Letting $\rho = (\rho')^{g^{-1}}$, we have $\mathbf{M}(X', \rho') \cong \mathbf{M}(X, \rho)$ by Theorem \ref{thm:classification-2}. This means that the map is surjective.

To show the injectivity, we let $X, Y \in \cO$, $[\rho] \in I_X$ and $[\sigma] \in I_Y$ and assume that $[\mathbf{M}(X, \rho)] = [\mathbf{M}(Y, \sigma)]$. By Theorem \ref{thm:classification-2}, there is an element $g \in \Gamma$ such that $Y \cong X \otimes E_g$ and $\sigma \sim \rho^g$. Since $X$ and $Y$ are representatives of some orbits, we have $X = Y$ and thus $g \in \Stab_{\Gamma}(X)$. Hence $\sigma \sim \rho^g \sim \rho$. This shows that the map is injective. The proof is done.
\end{proof}


\subsection{Tensor category \texorpdfstring{$\fg\text{-}\iC_A$}{fg-C\_A}}
From here on, $\C$ is a tensor category.
As above, $A=\bigoplus_{g\in\Gamma} E_g$ is a simple current algebra over a subgroup $\Gamma < \Inv(\C)$. With some additional assumptions, we will verify that $A$ satisfies the assumptions of Theorem~\ref{thm:indexact-multitensor}, and thus the conclusion of this theorem holds.


\subsubsection{Artinian property}
\begin{proposition}\label{prop:A-artinian}
The algebra $A$ is Artinian.
\end{proposition}

\begin{proof}
  Since a finitely generated object of $\fg\text{-}\iC_A$ is a quotient of $X \otimes A$ for some $X \in \C$, and since a quotient of a finite length object is also of finite length, it suffices to show that the free module $X \otimes A$ for $X \in \C$ is of finite length. We prove this by induction on the length of $X$ as an object of $\C$.

  The claim is trivial if $\ell(X) = 0$. We assume that $\ell(X) = 1$ or, equivalently, $X$ is a simple object. Since $X \otimes A_X$ is a right $A_X$-module whose underlying object is isomorphic to a direct sum of finitely many copies of $X$, by Lemma \ref{lem:functor-LX}, there exists $\rho \in \Rep(S^{\op}, \xi_X^{\op})$ such that $X \otimes A_X \cong \mathbf{L}(X, \rho)$ as right $A_X$-modules. There are isomorphisms
  \begin{equation*}
      X \otimes A \cong (X \otimes A_X) \otimes_{A_X} A
      \cong \mathbf{L}(X, \rho) \otimes_{A_X} A = \mathbf{M}(X, \rho)
  \end{equation*}
  of right $A$-modules. Thus, by Lemma \ref{lem:functor-MX-preserves-length}, $X \otimes A$ is of finite length.

  Now we assume that $\ell(X) > 1$ and $Y \otimes A$ is of finite length for all $Y \in \C$ with $\ell(Y) < \ell(X)$. We choose a simple subobject $X' \subset X$ and let $X'' = X/X'$ so that there is a short exact sequence $0 \to X' \to X \to X'' \to 0$ in $\C$. By applying the free module functor to this sequence, we obtain a short exact sequence $0 \to X' \otimes A \to X \otimes A \to X'' \otimes A \to 0$ in $\iC_A$. Since $X' \otimes A$ and $X'' \otimes A$ are of finite length, so is $X \otimes A$. This proves that $A$ is Artinian.
\end{proof}


\subsubsection{Frobenius algebra structure}
In this subsection, we prove that $A$ is Frobenius. We recall Iovanov's result \cite{iovanov2006isomorphic}, which plays a key role. Let $\mathcal{A}$ be a complete and cocomplete abelian category, and let $\{ M_i \}_{i \in I}$ be an infinite family of objects of $\mathcal{A}$. Then we have two objects: The product $\Pi := \prod_{i \in I} M_i$ and the coproduct $\Sigma := \bigoplus_{i \in I} M_i$. There is a canonical morphism $\Sigma \to \Pi$. When $\mathcal{A}$ is the category of modules over a ring and $M_i \ne 0$ for infinitely many $i$, the canonical morphism is never an isomorphism. However, according to Iovanov \cite[Example 2.9]{iovanov2006isomorphic}, the canonical morphism $\Sigma \to \Pi$ is an isomorphism if $\mathcal{A}$ is an ind-completion of a locally finite abelian category and $\Sigma$ is quasi-finite. We use this to prove:

\begin{theorem}
  \label{thm:simple-current-ext-is-Fb}
The algebra $A$ is Frobenius with the Frobenius form $\lambda : A \to \unit$ given by the projection to $E_1 = \unit$.
\end{theorem}
\begin{proof}
For a family $\{ M_i \}_{i \in I}$ of objects of $\iC$, we have $[ \bigoplus_{i \in I} M_i, N ] \cong \prod_{i \in I} [M_i, N]$. Thus,
  \begin{equation*}
    [A, \unit]
    \cong \prod_{g \in \Gamma} [g, \unit]
    \cong \prod_{g \in \Gamma} g^{-1}
    \cong \prod_{g \in \Gamma} g.
  \end{equation*}
  Unwinding the definition of the morphism $\phi : A \to [A, \unit]$ in Definition \ref{def:Frobenius-alg-in-Ind}, one checks that under the above identifications it coincides with the
  canonical morphism $\bigoplus_{g\in \Gamma} g\to \prod_{g\in \Gamma} g$:
  indeed, writing $A=\bigoplus_{g\in \Gamma} g$, the composite $\lambda \mu$
  vanishes on $g\otimes h$ unless $h\simeq g^{-1}$, in which case it is identified (up to the chosen
  identifications $g\otimes g^{-1}\cong\unit$) with the evaluation pairing.
  By adjunction, this yields a morphism $g\to [A,\unit]\cong\prod_{h\in \Gamma} h$ whose only non-zero component is the
  canonical map $g\to g$ into the $h=g$ factor. Summing over $g$ gives the canonical morphism
  $\bigoplus_{g\in \Gamma} g\to \prod_{g\in \Gamma} g$.
  This is an isomorphism by Iovanov's result cited above.
\end{proof}


\subsubsection{Rigidity of simple modules} \label{subsubsec:rigidity}
We retain the notation in the previous subsection.
We fix a simple object $X$ of $\C$ with stabilizer $S$ and a projective representation $\rho$ of $S^{\op}$ with multiplier $\xi_X^{\op}$.

\begin{theorem}
  Assume that $|S| \ne 0$ in the base field $\kk$.
  Then $\mathbf{M}(X, \rho)$ is a direct summand of the free module $X^{\oplus m} \otimes A$ for some positive integer $m$.
\end{theorem}

\begin{proof}
  Since $A_X$ is separable under the assumption of this theorem \cite[Lemma~3.1]{fuchs2004tft}, the quotient morphism $M \otimes N \to M \otimes_{A_X} N$ for a right $A_X$-module $M$ and a left $A_X$-module $N$ in $\iC$ has a section that is natural in $M$ and $N$.
  Thus, when $N$ is an $A_X$-$B$-bimodule for some algebra $B$ in $\iC$, the right $B$-module $M \otimes_{A_X} N$ is a direct summand of the right $B$-module $M \otimes N$. By applying this argument to $M = \mathbf{L}(X, \rho)$, $N = A$ and $B = A$, the claim follows.
\end{proof}

\begin{corollary}
  Suppose that $X \in \C$ is rigid and $A$ is a central commutative algebra in $\iC$. If $|S| \ne 0$ in $\kk$, $\mathbf{M}(X, \rho)$ is a rigid object in the category of right $A$-modules in $\iC$.
\end{corollary}
\begin{proof}
  The previous theorem says that $\mathbf{M}(X, \rho)$ is a direct summand of a free module $X^{\oplus m} \otimes A$ for some $m$. Since $X \otimes A$ is rigid (with left dual $X^* \otimes A$ and right dual ${}^*X \otimes A$), the claim follows from the fact that a direct summand of a rigid object is rigid.
\end{proof}

In view of the above corollary, we consider the following condition:
\begin{align}\label{eq:condition-diamond}
  \text{$|\Stab_{\Gamma}(X)| \ne 0$ in $\kk$ for every simple object $X \in \C$}.
\end{align}

The condition \eqref{eq:condition-diamond} holds if, for example, $\kk$ is of characteristic zero. Even in the case where $\kk$ has positive characteristic $p > 0$, the condition \eqref{eq:condition-diamond} holds if $g^p \ne 1$ for every non-identity element $g \in \Gamma$. In particular, \eqref{eq:condition-diamond} holds if $\Gamma$ is torsion free.

\begin{corollary}\label{cor:rigidity-simple-modules}
  If $\C$ is a tensor category and the simple current algebra $A$ is  central commutative such that the condition \eqref{eq:condition-diamond} is satisfied, then $\fg\text{-}\iC_A$ is a tensor category.
\end{corollary}
\begin{proof}
  By the previous corollary, every simple object of $\fg\text{-}\iC_A$ is rigid. Moreover, $A$ is a haploid commutative algebra. Thus, by Theorem~\ref{thm:indexact-multitensor}, $\fg\text{-}\iC_A$ is a tensor category.
\end{proof}


\subsubsection{Existence of enough projectives and finiteness}\label{subsubsec:finiteness}
Assume that $\C$ has enough projective objects. Then, by Theorem~\ref{thm:artinian-locally-finite}, Artinian property of $A$ implies that $\fg\text{-}\iC_A$ has enough projective objects. We give a more explicit construction of projective objects.

Let $X \in \C$ be a simple object with
projective cover $\pi : P \to X$, and let $\rho$ be an irreducible projective representation of
$\Stab_{\Gamma}(X)^{\op}$ with multiplier $\xi_X^{\op}$. We set
$\mathbf{P}'(X, \rho) := (P \otimes W) \otimes A_X$ (the free right $A_X$-module), where $W$ is the
representation space of $\rho$. Then $\mathbf{P}'(X, \rho)$ is a projective right $A_X$-module and there is an
epimorphism
\begin{equation*}
  \mathbf{P}'(X, \rho)
  \xrightarrow{\quad (\pi \otimes \id) \otimes \id \quad}
  \mathbf{L}(X, \rho) \otimes A_X
  \xrightarrow{\quad \text{action} \quad} \mathbf{L}(X, \rho)
\end{equation*}
of right $A_X$-modules. Now we set
\begin{equation*}
  \mathbf{P}(X, \rho) := \mathbf{P}'(X, \rho) \otimes_{A_X} A
  \cong (P \otimes W) \otimes A.
\end{equation*}
Then $\mathbf{P}(X, \rho)$ is a projective object in the category of right $A$-modules in $\iC$ and the simple module $\mathbf{M}(X, \rho) = \mathbf{L}(X, \rho) \otimes_{A_X} A$ is a quotient of $\mathbf{P}(X, \rho)$.

Thus $\fg\text{-}\iC_A$ is a locally finite abelian category where every simple object is a quotient of a projective object. This implies that $\fg\text{-}\iC_A$ has enough projective objects.

\begin{theorem}\label{thm:finite-tensor-category-criterion}
  Let $\C$ be a Frobenius tensor category. Suppose that $A$ is a central commutative algebra in $\iC$ such that the condition \eqref{eq:condition-diamond} is satisfied. Then, $\fg\text{-}\iC_A$ is a Frobenius tensor category. Moreover, $\fg\text{-}\iC_A$ is a finite tensor category if and only if $\mathrm{Irr}(\C)/\Gamma$ is finite.
\end{theorem}
\begin{proof}
By Corollary~\ref{cor:rigidity-simple-modules}, $\fg\text{-}\iC_A$ is a tensor category. As $\C$ is Frobenius and $A$ is Artinian, by Theorem~\ref{thm:artinian-locally-finite}, $\fg\text{-}\iC_A$ has enough projective objects. Thus, it is a Frobenius tensor category.
The claim about finiteness follows by Corollary~\ref{cor:number-of-simples}. 
\end{proof}

\begin{question}\label{question:fpdim}
Assuming that $\fg\text{-}\iC_A$ is a finite tensor category, what are the Frobenius-Perron dimensions of simple objects and their projective covers? What is the FPdim of the category?
\end{question}


\subsection{The subcategory of local modules}
From now on, $\C$ is a braided tensor category and $A$ is a commutative simple current algebra over a subgroup $\Gamma < \Inv(\C)$. By Lemma~\ref{lem:group-algebra-comm} and equation \eqref{eq:comm-implies-trivial-monodromy}, this implies that the invertible summands of $A$ centralize each other, that is, the following equation holds:
\begin{equation}\label{eq:Eg-Eh-commutative}
  c_{E_h,E_g} c_{E_g, E_h} = \id_{E_g \otimes E_h} \quad (g,h\in \Gamma).
\end{equation}


\subsubsection{Classification of simple local modules}
Fix a simple object $X\in \C$ and an irreducible projective representation $\rho$ of the stabilizer $S=\Stab_\Gamma(X)^{\op}$. Let
\begin{equation}\label{eq:C-Gamma-def}
\C_\Gamma:=\{\,Y\in\C \mid c_{E_g,Y}c_{Y,E_g}=\id_{Y\otimes E_g}\ \text{for all } g\in \Gamma\,\}  
\end{equation}
be the M\"uger centralizer of the pointed tensor subcategory
$\langle E_g \mid g\in \Gamma\rangle\subset \C$.

\begin{proposition}\label{prop:locality-criterion}
Then $\mathbf{M}(X,\rho)$ is a local $A$-module if and only if $X\in \C_\Gamma$.
\end{proposition}
\begin{proof}
Put $M:=\mathbf{M}(X,\rho)$ and let $a_M:M\otimes A\to M$ be its action.

Assume $c_{E_g,X}c_{X,E_g}=\id_{X\otimes E_g}$ for all $g\in \Gamma$.
Then the free module $X\otimes A$ is local.
As $M$ is a quotient of $X\otimes A$ and $\iC_A^{\loc}$ is closed under cokernels in $\iC_A$, $M$ is local.

Conversely, assume first that $M$ is local. By \eqref{eq:M-X-rho-decomposition}, choose coset representatives
$\{ g_i \}_{i \in I}$ of $S \backslash \Gamma$ with $g_1=1$, so one summand of $M$ is $W\otimes X$. For each $g\in \Gamma$, the $(1,g)$-block
of the action is an isomorphism
\[
a_{1,g}:((W\otimes X)\otimes E_1)\otimes E_g \xrightarrow{\ \sim\ } (W\otimes X)\otimes E_{g_j},
\]
where $Sg_j=Sg$ (for $S=\Stab_\Gamma(X)$). Restricting locality $a_M\circ c_{A,M}=a_M\circ c^{-1}_{M,A}$ to this block and postcomposing with $a_{1,g}^{-1}$ gives
\[
c_{E_g,W\otimes X}=c^{-1}_{W\otimes X,E_g}.
\]
Since $W$ is a finite-dimensional vector space object (hence a direct sum of copies of $\unit$),
this is equivalent to
\[
c_{E_g,X}=c^{-1}_{X,E_g},
\]
i.e. $c_{E_g,X}c_{X,E_g}=\id_{X\otimes E_g}$ for all $g\in \Gamma$.
\end{proof}

\begin{corollary}\label{cor:local-number-of-simples}
For brevity, we set $I_X = \Irr\bigl(\Rep(\Stab_\Gamma(X)^{\op},\,\xi_X^{\op})\bigr)$ for a simple object $X \in \C$, where $\xi_X$ is the $2$-cocycle determined by \eqref{eq:xi-X-def}. Then there is a bijection
\[ \bigsqcup_{X \in \cO_{\Gamma}} I_X \to \Irr(\iC_A^\loc),
\quad [\rho] \mapsto [\mathbf{M}(X, \rho)] \quad ([\rho] \in I_X), \]
where $\cO_{\Gamma}$ is orbit representatives of $\Irr(\C_{\Gamma}) / \Gamma$.
\end{corollary}

The bijection implies $\#\,\Irr(\iC_A^{\loc}) = \sum_{[X]\in \cO_{\Gamma}}\;\# I_X$. Since the set $I_X$ is non-empty and finite for all simple $X \in \C$, we have that $\Irr(\iC_A^{\loc})$ is finite if and only if $\Irr(\C_{\Gamma})/\Gamma$ is finite.

\begin{proof}
By Proposition~\ref{prop:locality-criterion}, the simple objects of $\iC_A^\loc$ are exactly those $\mathbf{M}(X,\rho)$ with $X\in \Irr(\C_\Gamma)$ and $\rho$ irreducible. The remaining argument is the same as Corollary~\ref{cor:number-of-simples}.
\end{proof}


\subsubsection{Braiding between simple local modules}

\begin{lemma}\label{lem:double-braiding}
Let $\mathbf{M}(X,\rho)$ and $\mathbf{M}(Y,\sigma)$ be simple local $A$-modules.
Then the double braiding in $\fg\text{-}\iC^{\loc}_A$ between $\mathbf{M}(X,\rho)$ and $\mathbf{M}(Y,\sigma)$ is
trivial if and only if the double braiding in $\C$ between $X$ and $Y$ is trivial.
\end{lemma}
\begin{proof} 
($\Rightarrow$) We first introduce some notations: We set $S = \Stab_{\Gamma}(X)$ and $T = \Stab_{\Gamma}(Y)$, and let $\{ x_i \}_{i \in I}$, $\{ y_j \}_{j \in J}$ and $\{ z_k \}_{k \in K}$ be representatives of $\Gamma / S$, $\Gamma / T$ and $T \backslash \Gamma / S$, respectively. We assume that the index sets $I$, $J$ and $K$ have a special element $0$ such that $x_0 = y_0 = z_0 = 1$.

By \eqref{eq:M-X-rho-decomposition} and the braiding, we obtain an isomorphism
\begin{equation}\label{eq:double-braiding-proof-1}
    \mathbf{M}(X, \rho) \otimes \mathbf{M}(Y, \sigma)
    \cong \bigoplus_{i \in I} \bigoplus_{j \in J}
    (\mathbf{L}(X, \rho) \otimes \mathbf{L}(Y, \sigma)) \otimes (E_{x_i} \otimes E_{y_j}).
\end{equation}

Since $A \cong \bigoplus_{k \in K} A_Y \otimes (E_{z_k} \otimes A_X)$ as $A_Y$-$A_X$-bimodules, we have
\begin{equation*}
    \mathbf{M}(Y, \sigma)
    = \mathbf{L}(Y, \sigma) \otimes_{A_Y} A
    \cong \bigoplus_{k \in K} \mathbf{L}(Y, \sigma) \otimes (E_{z_k} \otimes A_X)
    \cong \bigoplus_{k \in K} A_X \otimes (\mathbf{L}(Y, \sigma) \otimes E_{z_k})
\end{equation*}
as $A_X$-modules. Hence,
\begin{equation}\label{eq:double-braiding-proof-2}
\begin{aligned}
    \mathbf{M}(X, \rho) \otimes_A \mathbf{M}(Y, \sigma)
    & = (\mathbf{L}(X, \rho) \otimes_{A_X} A) \otimes_A \mathbf{M}(Y, \sigma)
    \cong \mathbf{L}(X, \rho) \otimes_{A_X} \mathbf{M}(Y, \sigma) \\
    & \cong \bigoplus_{k \in K} (\mathbf{L}(X, \rho) \otimes \mathbf{L}(Y, \sigma)) \otimes E_{z_k}
\end{aligned}
\end{equation}
as objects of $\iC$. Let $\pi : \mathbf{M}(X, \rho) \otimes \mathbf{M}(Y, \sigma) \to \mathbf{M}(X, \rho) \otimes_A \mathbf{M}(Y, \sigma)$ be the canonical epimorphism, and define the epimorphism $\tilde{\pi}$ so that the following diagram is commutative:
\begin{equation*}
\begin{tikzcd}[column sep = 64pt]
    \mathbf{M}(X, \rho) \otimes \mathbf{M}(Y, \sigma)
    \arrow[d, "{\pi}"'] \arrow[r, "{\eqref{eq:double-braiding-proof-1}}", "{\cong}"']
    & \bigoplus_{i \in I} \bigoplus_{j \in J} (\mathbf{L}(X, \rho) \otimes \mathbf{L}(Y, \sigma)) \otimes (E_{x_i} \otimes E_{y_j})
    \arrow[d, "{\tilde{\pi}}"]  \\
    \mathbf{M}(X, \rho) \otimes_A \mathbf{M}(Y, \sigma)
    \arrow[r, "{\eqref{eq:double-braiding-proof-2}}", "{\cong}"']
    & \bigoplus_{k \in K} (\mathbf{L}(X, \rho) \otimes \mathbf{L}(Y, \sigma)) \otimes E_{z_k}
\end{tikzcd}    
\end{equation*}
The epimorphism $\widetilde{\pi}$ maps the $(i, j)$-th component $(\mathbf{L}(X, \rho) \otimes \mathbf{L}(Y, \sigma)) \otimes (E_{x_i} \otimes E_{y_j})$ to the $k$-th component, where $k$ is the unique index such that $x_i y_j \in T z_k S$. We let
\begin{equation*}
    \tilde{\pi}_{i,j}:
    (\mathbf{L}(X, \rho) \otimes \mathbf{L}(Y, \sigma)) \otimes (E_{x_i} \otimes E_{y_j})
    \to (\mathbf{L}(X, \rho) \otimes \mathbf{L}(Y, \sigma)) \otimes E_{z_k}
    \quad (x_i y_j \in T z_k S)
\end{equation*}
be the morphism induced by $\tilde{\pi}$. Although the morphism $\tilde{\pi}_{i j}$ may be tedious to be written down explicitly in general (as it involves the actions of $t \in T$ and $s \in S$ such that $x_i y_j = t z_k s$), one can easily verify that $\tilde{\pi}_{0 0}$ is the identity morphism.

We consider the double braiding $M$ for $\mathbf{M}(X, \rho)$ and $\mathbf{M}(Y, \sigma)$. Since $X, Y \in \mathcal{C}_{\Gamma}$, and since the invertible summands of $A$ centralize each other, we have a commutative diagram
\begin{equation*}
\begin{tikzcd}[column sep = 32pt]
    \mathbf{M}(X, \rho) \otimes \mathbf{M}(Y, \sigma)
    \arrow[d, "{M_{\mathbf{M}(X, \rho), \mathbf{M}(Y, \sigma)}}"']
    \arrow[r, "{\eqref{eq:double-braiding-proof-1}}", "{\cong}"']
    & \bigoplus_{i \in I} \bigoplus_{j \in J} (\mathbf{L}(X, \rho) \otimes \mathbf{L}(Y, \sigma)) \otimes (E_{x_i} \otimes E_{y_j})
    \arrow[d, "{\bigoplus_{i \in I} \bigoplus_{j \in J} M_{\mathbf{L}(X, \rho), \mathbf{L}(Y, \sigma)} \otimes \id}"]  \\
    \mathbf{M}(X, \rho) \otimes \mathbf{M}(Y, \sigma)
    \arrow[r, "{\eqref{eq:double-braiding-proof-1}}", "{\cong}"']
    & \bigoplus_{i \in I} \bigoplus_{j \in J} (\mathbf{L}(X, \rho) \otimes \mathbf{L}(Y, \sigma)) \otimes (E_{x_i} \otimes E_{y_j})
\end{tikzcd}
\end{equation*}

Let $M^A$ denote the double braiding of the category $\iC_A^{\loc}$ of local modules.
By definition, $M^A$ is characterized by the equation $\pi \circ M_{\mathbf{M}(X, \rho), \mathbf{M}(Y, \sigma)} = M^A_{\mathbf{M}(X, \rho), \mathbf{M}(Y, \sigma)} \circ \pi$. By the above commutative diagram, we see that there are isomorphisms $m_k$ ($k \in K$) such that the following diagram is commutative:
\begin{equation*}
\begin{tikzcd}
    \mathbf{M}(X, \rho) \otimes_A \mathbf{M}(Y, \sigma)
    \arrow[d, "{M^A_{\mathbf{M}(X, \rho), \mathbf{M}(Y, \sigma)}}"']
    \arrow[r, "{\eqref{eq:double-braiding-proof-2}}", "{\cong}"']
    & \bigoplus_{k \in K} (\mathbf{L}(X, \rho) \otimes \mathbf{L}(Y, \sigma)) \otimes E_{z_k}
    \arrow[d, "{\bigoplus_{k \in K} m_k}"] \\
    \mathbf{M}(X, \rho) \otimes_A \mathbf{M}(Y, \sigma)
    \arrow[r, "{\eqref{eq:double-braiding-proof-2}}", "{\cong}"']
    & \bigoplus_{k \in K} (\mathbf{L}(X, \rho) \otimes \mathbf{L}(Y, \sigma)) \otimes E_{z_k}
\end{tikzcd}
\end{equation*}
By the above commutative diagrams and the definition of $M^A$, we have
\begin{equation*}
    m_0 = M_{\mathbf{L}(X, \rho), \mathbf{L}(Y, \sigma)} \otimes \id_{\unit}.
\end{equation*}

Now we assume that $M^A_{\mathbf{M}(X, \rho), \mathbf{M}(Y, \sigma)}$ is the identity morphism. Then $m_k$ is the identity morphism for all $k \in K$. Thus, in particular, $m_0 = \id$. Since $\mathbf{L}(X, \rho)$ and $\mathbf{L}(Y, \sigma)$ are direct sums of finitely many copies of $X$ and $Y$, respectively, $m_0 = \id$ implies that $M_{X,Y} = \id$. We have proved the `only if' part of this theorem.

($\Leftarrow$) For the converse, assume that $M_{X,Y}=\id_{X\otimes Y}$. Since $\mathbf{L}(X,\rho)$ and $\mathbf{L}(Y,\sigma)$ are direct sums of finitely many copies of $X$ and $Y$, respectively, we
have $M_{\mathbf{L}(X,\rho),\mathbf{L}(Y,\sigma)}=\id$. Also, by \eqref{eq:Eg-Eh-commutative}, $M_{E_{x_i},E_{y_j}}=\id$ for all $i\in I$ and $j\in
J$. Hence the right vertical morphism in the above commutative diagram is the identity, and
therefore $M_{\mathbf{M}(X,\rho),\mathbf{M}(Y,\sigma)}=\id$. Now the relation
\[
\pi \circ M_{\mathbf{M}(X,\rho),\mathbf{M}(Y,\sigma)}
=
M^A_{\mathbf{M}(X,\rho),\mathbf{M}(Y,\sigma)} \circ \pi
\]
implies, since $\pi$ is an epimorphism, that
\[
M^A_{\mathbf{M}(X,\rho),\mathbf{M}(Y,\sigma)}=\id.
\]
This proves the `if' part.
\end{proof}


\subsubsection{Ribbon structure and non-degeneracy}
Let $\G:=\langle E_g\mid g\in \Gamma\rangle\subset \C$. By \eqref{eq:Eg-Eh-commutative}, $\G$ is symmetric, hence $\G\subset \C_\Gamma=\Z_{(2)}(\G\subset \C)$. Now we can state the main result of this section.

\begin{theorem}\label{thm:fg-local-simple-main}
  Let $\C$ be a (Frobenius) braided tensor category such that condition \eqref{eq:condition-diamond} is satisfied. 
  \begin{enumerate}
    \item Then, $\fg\text{-}\iC_A^{\loc}$ is a (Frobenius) braided tensor category. 
    \item If $\C$ is ribbon and $\theta_{E_g}=\id_{E_g}$ for all $g\in \Gamma$, then $\fg\text{-}\iC_A^{\loc}$ is a ribbon tensor category. 
    \item When $\C$ has enough projectives, $\fg\text{-}\iC_A^{\loc}$ is a braided finite tensor category if and only if $\Irr(\C_\Gamma)/\Gamma$ is finite.
  \end{enumerate}
\end{theorem}
\begin{proof}
(a) Recall that $A$ is Artinian and haploid. Moreover, by Corollary~\ref{cor:rigidity-simple-modules}, all simple $A$-modules are rigid. Thus, by Theorem~\ref{thm:indexact-multitensor-local}, $\fg\text{-}\iC_A^\loc$ is a braided (Frobenius) tensor category. 

(b) Note that $A$ is Frobenius by Theorem~\ref{thm:simple-current-ext-is-Fb}, and that $\theta_A=\id_A$ by the assumption on twists. Since $\fg\text{-}\iC_A^\loc$ is rigid by part (a), Theorem~\ref{thm:local-modules-ribbon} implies that it is ribbon. 

(c) By part (a), $\fg\text{-}\iC_A^\loc$ is a locally finite abelian category with enough projectives. Hence, it is braided finite if and only if it has finitely many simple objects. The latter is equivalent to $\Irr(\C_\Gamma)/\Gamma$ being finite by Corollary~\ref{cor:local-number-of-simples}.
\end{proof}

\begin{theorem}\label{thm:fg-local-simple-main-non-deg}
  Let $\C$ be a Frobenius braided tensor category and \textup{char}$(\kk)=0$. If the simple objects in $\C_\Gamma$ that trivially double braid with all other simples in $\C_\Gamma$ are precisely $E_g$ for $g\in \Gamma$, then $\fg\text{-}\iC^{\loc}_A$ is non-degenerate.
\end{theorem}
\begin{proof}
As char$(\kk)=0$, condition \eqref{eq:condition-diamond} is satisfied. Thus, Theorem~\ref{thm:fg-local-simple-main}(a) implies that $\fg\text{-}\iC^{\loc}_A$ is a Frobenius tensor category. Hence, by Lemma~\ref{lem:non-deg-Muger-center}, $\fg\text{-}\iC^{\loc}_A$ is non-degenerate if and only if the only simple object in its M\"uger center is $A$. But, simple objects in $\fg\text{-}\iC_A^\loc$ are $\mathbf{M}(X,\rho)$ for $X\in\C_\Gamma$ and for such an object to lie in the M\"uger center, it has to double braid trivially with all other simple objects. By Lemma~\ref{lem:double-braiding}, this implies that $X$ trivially double braids with all simple objects of $\C_\Gamma$. By assumption of the theorem, $X\cong E_g$ for some $g\in \Gamma$. Now, Lemma~\ref{lem:M-X-rho-Eg} implies that $\mathbf{M}(X,\rho)\cong A$. Thus, the M\"uger center of $\fg\text{-}\iC^{\loc}_A$ is trivial, i.e. $\fg\text{-}\iC^{\loc}_A$ is non-degenerate.
\end{proof}


\section{Examples from quantum (super)group categories}\label{sec:examples}
In this section we study two families of examples coming from unrolled quantum groups $U^H(\mathfrak{g})$ and the unrolled quantum supergroups $\mathfrak{gl}(1|1)$. 
We analyze the corresponding categories of finite-dimensional weight modules and the simple-current local-module categories they produce. We follow the same program:
\begin{enumerate}
\item identify the invertible simple objects and their tensor product law;
\item compute the relevant braiding and twist data on these invertibles;
\item choose commutative simple current algebras and extract consequences for the corresponding category of local modules.
\end{enumerate}
The categorical consequences in step (c) are obtained by applying Theorem~\ref{thm:fg-local-simple-main} to the explicit simple-current data. 

Let $\bZ/2\bZ=\{\bar 0,\bar 1\}$ be the additive group of order $2$. We will sometimes write $M_{S,T}$ to denote the double braiding between two objects $S$ and $T$.


\subsection{Unrolled quantum group categories}
The categorical input follows Creutzig--Rupert \cite{creutzig2022uprolling} (and references therein); we recall only the data needed for applying the simple-current machinery of Section~\ref{sec:simple-current-algebras}. The setting for general sublattices $L'\subseteq \Lambda_{\mathrm{inv}}$ (Proposition~\ref{prop:unrolled-ALprime} and its corollaries) is new.

Fix a simple complex Lie algebra $\mathfrak g$ with Cartan matrix $A=(a_{ij})$,
root lattice $Q\subset \mathfrak h^*$, and weight lattice $P\subset \mathfrak h^*$.
Assume $q$ is a primitive $\ell$-th root of unity such that ord$(q^2)>d_i$ for all $i$. We fix $\log q$ and set $q^z = \exp(z \log q)$ for $z \in \mathbb{C}$. We also set \begin{equation}\label{eq:unrolled-r}
r=\ell \text{ if }\ell\text{ is odd},\qquad r=\ell/2 \text{ if }\ell\text{ is even}. 
\end{equation}
Let $\mathcal C:=\Rep^{\mathrm{wt}}(U_q^H(\mathfrak g))$ be the category of finite-dimensional weight modules for Rupert's unrolled restricted quantum group, obtained from the unrolled quantum group by quotienting by the Hopf ideal generated by the prescribed powers of the root vectors. Concretely, $U_q^H(\mathfrak g)$ is generated by $E_i,F_i,H_i$ together with group-like elements $K_\gamma$ indexed by a fixed lattice between $Q$ and $P$, and a weight module is a finite-dimensional module that decomposes into common $H_i$-eigenspaces such that, for $\gamma=\sum_i k_i\alpha_i$, the operator $K_\gamma$ acts on the $\lambda$-weight space by the scalar $q^{\sum_i k_i d_i\lambda(H_i)}$.
Let $\lambda(H_i)$ denote evaluation of $\lambda\in\mathfrak h^*$ on $H_i$.
By Rupert \cite{rupert2022categories} (see also \cite{creutzig2022uprolling}), $\mathcal C$ is a ribbon Frobenius tensor category with enough projectives (equivalently, projective covers), and its M\"uger center is trivial.


Let $M_{\lambda}$ denote the simple quotient of the Verma module with highest weight $
\lambda\in\mathfrak h^*$. Then
\[
\Irr(\mathcal C)=\{\,M_\lambda \mid \lambda\in\mathfrak h^*\,\}.
\]
The module $M_\lambda$ is one-dimensional, hence invertible, if and only if $E_i$ and $F_i$ act by
zero for all $i$, which is equivalent to $\lambda(H_i)\in \frac{\ell}{2d_i}\bZ$ for all $i$. In that case we write $C_\lambda:=M_\lambda$. Hence
\[
\Lambda_{\mathrm{inv}}
:=
\left\{\lambda\in\mathfrak h^* \,\middle|\, \lambda(H_i)\in \frac{\ell}{2d_i}\bZ\, \forall\,
i\right\}
\]
is isomorphic to the additive group of invertible simple objects in $\mathcal C$, via $\lambda\mapsto C_\lambda$, and for $\lambda,\mu\in \Lambda_{\mathrm{inv}}$ one has $C_\lambda\otimes C_\mu \cong C_{\lambda+\mu}$. On invertibles one has
\[
c_{C_\lambda,C_\mu}=q^{\langle\lambda,\mu\rangle}\tau,
\qquad
M_{C_\lambda,C_\mu}=q^{2\langle\lambda,\mu\rangle}\id.
\]
where $\tau$ is the braiding of $\Vect$ and
\[
\langle\lambda,\mu\rangle:=\sum_{i,j} d_i(A^{-1})_{ij}\lambda(H_i)\mu(H_j)
\]
for the invariant bilinear form on $\mathfrak h^*$. Here $d_i = \langle\alpha_i,\alpha_i\rangle/2$ are the symmetrizers. The twist is \cite[Eq.~(4.12)]{creutzig2022uprolling}:
\begin{equation}\label{eq:unrolled-twist}
\theta_{C_\lambda}=q^{\langle\lambda,\lambda+2(1-r)\rho\rangle}\id_{C_\lambda},
\end{equation}
where $\rho$ is the Weyl vector and $r$ is as in \eqref{eq:unrolled-r}. Thus, for an additive subgroup
$L'\subseteq \Lambda_{\mathrm{inv}}$, the corresponding simple current algebra is
\[
A_{L'}=\bigoplus_{\lambda\in L'} C_\lambda\in \widehat{\mathcal C}.
\]


In the notation of Section~\ref{sec:simple-current-algebras}, write
\[
\beta(\lambda,\mu)=q^{\langle\lambda,\mu\rangle},
\qquad
b(\lambda,\mu)=q^{2\langle\lambda,\mu\rangle},
\qquad
\mathsf{q}(\lambda)=q^{\langle\lambda,\lambda\rangle}=\beta(\lambda,\lambda).
\]
For an additive subgroup $L'\subseteq \mathfrak h^*$, we write
\[
(L')^*:=\{\gamma\in \mathfrak h^* \mid \langle \gamma,\lambda\rangle\in \bZ \text{ for all } \lambda\in L'\}
\]
for the dual lattice with respect to $\langle-,-\rangle$. 

\begin{proposition}\label{prop:unrolled-ALprime}
\begin{enumerate}
\item[\textup{(a)}] The object $A_{L'}$ is a haploid algebra in $\widehat{\mathcal C}$. It is commutative if and only if $\mathsf{q}|_{L'}\equiv 1$. Equivalently, $\langle \lambda,\lambda\rangle\in \ell\bZ$ for $\lambda\in L'$.

\item[\textup{(b)}] Let
\[
\mathcal C_{L'}
:=
\left\{X\in\mathcal C\ \middle|\ c_{C_\lambda,X}\circ c_{X,C_\lambda}=\id_{X\otimes C_\lambda}\ \ \forall\,\lambda\in L'\right\}.
\]
Then $\mathcal C_{L'}=\mathcal C^{\mathrm{int}}_{L'}$, where $\mathcal C^{\mathrm{int}}_{L'}$ is the full subcategory of objects all of whose weights lie in $\frac{\ell}{2}(L')^*$.

\item[\textup{(c)}] Assume that $A_{L'}$ is commutative. Then the orbit set $\Irr(\mathcal C_{L'})/L'$, where $L'$ acts by tensoring with the invertibles $C_\lambda$, is identified with
\[
\Lambda(L')
:=
\frac{\frac{\ell}{2}(L')^*}{L'}.
\]
\end{enumerate}
\end{proposition}

Note that when $A_{L'}$ is commutative, equivalently $\langle \lambda,\lambda\rangle\in \ell\bZ$ for all $\lambda\in L'$, we get $L'\subseteq \frac{\ell}{2}(L')^*$, so the quotient $\Lambda(L')$ is well defined.

\begin{proof}
\textup{(a)} On invertibles, the braiding scalar is $\beta(\lambda,\mu)=q^{\langle\lambda,\mu\rangle}$, and this is a bicharacter on $\Lambda_{\mathrm{inv}}$. Hence the associated abelian $3$-cocycle on the pointed subcategory of invertibles is represented by $(1,\beta)$, so the ordinary associator cocycle $\omega$ restricts trivially to $L'$. By Lemma~\ref{lem:group-algebra-comm}(a), this gives an algebra structure on $A_{L'}$ for any additive subgroup $L'\subseteq \Lambda_{\mathrm{inv}}$. Then Lemma~\ref{lem:group-algebra-comm}(b) implies that $A_{L'}$ is commutative if and only if $\mathsf{q}|_{L'}\equiv 1$. Since $\mathsf{q}(\lambda)=q^{\langle\lambda,\lambda\rangle}$, this is equivalent to $\langle\lambda,\lambda\rangle\in \ell\bZ$ for all $\lambda\in L'$. That $A$ is haploid follows from $\Hom_{\mathcal C}(\unit,C_\lambda)=0$ for $\lambda\neq 0$, hence
$\Hom_{\widehat{\mathcal C}}(\unit,A_{L'})\cong\Hom_{\mathcal C}(\unit,C_0)\cong\kk$.

\textup{(b)} Let $X\in\C$ be a weight module. For a weight vector $w_\gamma\in X$ one has
\[
\bigl(c_{C_\lambda,X}\circ c_{X,C_\lambda}\bigr)(v_\lambda\otimes w_\gamma)
=q^{2\langle\lambda,\gamma\rangle}(v_\lambda\otimes w_\gamma).
\]
Hence $X\in\mathcal C_{L'}$ iff $2\langle\lambda,\gamma\rangle\in \ell\bZ$ for every $\lambda\in L'$ and every weight $\gamma$ of $X$, i.e.\ iff all weights lie in $\frac{\ell}{2}(L')^*$.

\textup{(c)} Since $A_{L'}$ is commutative by \textup{(a)}, Theorems~\ref{thm:classification-1} and \ref{thm:classification-2} together with Proposition~\ref{prop:locality-criterion} show that simple local $A_{L'}$-modules are classified by pairs $(M_\gamma,\rho)$, where $M_\gamma\in \mathcal C_{L'}$ and $\rho$ is an irreducible projective representation of
\[
\Stab_{L'}(M_\gamma)
:=
\{\,\eta\in L' \mid M_\gamma\otimes C_\eta\cong M_\gamma\,\}^{\op}.
\]
Since $M_\gamma\otimes C_\eta\cong M_{\gamma+\eta}$ and $M_\lambda\cong M_\mu$ only for $\lambda=\mu$, this stabilizer is trivial. Hence $\rho$ is necessarily trivial, so simple local $A_{L'}$-modules are parametrized exactly by $L'$-orbits in $\Irr(\mathcal C_{L'})$.

Moreover, $Q\subseteq \tfrac{\ell}{2}(L')^*$ since $\langle \eta,\alpha_i\rangle=d_i\eta(H_i)\in \tfrac{\ell}{2}\bZ$ for $\eta\in L'\subseteq\Lambda_{\mathrm{inv}}$ and each simple root $\alpha_i$; hence if $\gamma\in \tfrac{\ell}{2}(L')^*$, then every weight of $M_\gamma$ also lies in $\tfrac{\ell}{2}(L')^*$.

As $\Irr(\mathcal C)=\{\,M_\gamma\mid \gamma\in\mathfrak h^*\,\}$, part~\textup{(b)} gives that the simple objects of $\mathcal C_{L'}$ are exactly
\[
\{\,M_\gamma\mid \gamma\in \tfrac{\ell}{2}(L')^*\,\}.
\]
Moreover, $M_\gamma\otimes C_\eta \cong M_{\gamma+\eta}$ for $\eta\in L'$,
so two such simples lie in the same $L'$-orbit if and only if their parameters differ by an element of $L'$. This identifies $\Irr(\mathcal C_{L'})/L'$ with $\frac{\frac{\ell}{2}(L')^*}{L'}=\Lambda(L')$.
\end{proof}

\begin{lemma}\label{lem:internal-detector}
Let $M_\gamma\in \C_{L'}$ be simple. If $M_\gamma$ is non-invertible, then there exists
a simple object $M_\alpha\in \C_{L'}$ such that
\[
c_{M_\alpha,M_\gamma}\circ c_{M_\gamma,M_\alpha}\neq \id.
\]
In particular, every simple object of $\Z_{(2)}(\C_{L'})$ is invertible.
\end{lemma}
\begin{proof}
As observed in the proof of Proposition~\ref{prop:unrolled-ALprime}\textup{(c)}, one has
$Q\subseteq \frac{\ell}{2}(L')^*$.

Now let $v_\gamma$ and $v_{\alpha_i}$ be highest weight vectors of $M_\gamma$ and $M_{\alpha_i}$.
Write the braiding as $\tau\circ \check R$ with $\check R=\mathcal H\,\widetilde R$, where
$\mathcal H$ acts on weight vectors by $q^{\langle-,-\rangle}$ and
\[
\widetilde R
=
1+\sum_{\nu\in Q_{>0}}\sum_s c_{\nu,s}\,E_{\nu,s}\otimes F_{\nu,s}
\]
is a PBW expansion with each $E_{\nu,s}$ of positive root weight $\nu>0$. Since $v_\gamma$ and
$v_{\alpha_i}$ are highest weight vectors, every nontrivial PBW term kills the first tensor factor, so
\[
\widetilde R(v_\gamma\otimes v_{\alpha_i})=v_\gamma\otimes v_{\alpha_i},
\qquad
\widetilde R(v_{\alpha_i}\otimes v_\gamma)=v_{\alpha_i}\otimes v_\gamma.
\]
Therefore
\[
\bigl(c_{M_{\alpha_i},M_\gamma}\circ c_{M_\gamma,M_{\alpha_i}}\bigr)(v_\gamma\otimes v_{\alpha_i})
=
q^{2\langle\gamma,\alpha_i\rangle}(v_\gamma\otimes v_{\alpha_i})
=
q_i^{2\gamma(H_i)}(v_\gamma\otimes v_{\alpha_i}).
\]
If $M_\gamma$ is non-invertible, then $\gamma(H_i)\notin \frac{\ell}{2d_i}\bZ$ for some $i$, hence
$q_i^{2\gamma(H_i)}\neq 1$. For this $i$, we have $\alpha_i\in Q\subseteq \frac{\ell}{2}(L')^*$, so
$M_{\alpha_i}\in \C_{L'}$ and the above monodromy is not the identity. The final statement
follows immediately.
\end{proof}

\begin{corollary}\label{cor:unrolled-ALprime-combined}
Assume that $A_{L'}$ is commutative.
\begin{enumerate}
\item\label{cor:unrolled-ALprime-ribbon} $\fg\text{-}\widehat{\mathcal C}^{\loc}_{A_{L'}}$ is ribbon if and only if
$2(1-r)\langle\lambda,\rho\rangle\in \ell\bZ$ for all $\lambda\in L'$.
\item\label{cor:unrolled-ALprime-finite} $\fg\text{-}\widehat{\mathcal C}^{\loc}_{A_{L'}}$ is braided finite if and only if $\Lambda(L')$ is finite, equivalently $\rank(L')=\rank(P)$.
\item\label{cor:unrolled-ALprime-nondeg} $\fg\text{-}\widehat{\mathcal C}^{\loc}_{A_{L'}}$ is non-degenerate if the induced bicharacter on $\Lambda(L')$ given by
$\overline b([\lambda],[\mu])=q^{2\langle\lambda,\mu\rangle}$
has trivial radical.
\end{enumerate}
\end{corollary}

\begin{proof}
(a) By Theorem~\ref{thm:fg-local-simple-main}(b), being ribbon is equivalent to $\theta_{C_\lambda}=\id$ for all $\lambda\in L'$.
From \eqref{eq:unrolled-twist},
$\theta_{C_\lambda}
=
q^{\langle\lambda,\lambda\rangle}\,q^{2(1-r)\langle\lambda,\rho\rangle}\id_{C_\lambda}$.
By commutativity of $A_{L'}$, $\langle\lambda,\lambda\rangle\in \ell\bZ$. So, the first factor is $1$, and the criterion is exactly
$2(1-r)\langle\lambda,\rho\rangle\in \ell\bZ$ for all $\lambda\in L'$.

(b) The general criterion is Theorem~\ref{thm:fg-local-simple-main}(c):
braided finiteness is equivalent to finiteness of $\Irr(\mathcal C_{L'})/L'$.
By Proposition~\ref{prop:unrolled-ALprime}(c), this orbit set is $\Lambda(L')$.
The equivalence $\Lambda(L')$ finite $\Longleftrightarrow \rank(L')=\rank(P)$ is the lattice computation in \cite[Theorem~4.2]{creutzig2022uprolling}.

(c) By Theorem~\ref{thm:fg-local-simple-main-non-deg}, non-degeneracy holds if every
simple object of $\Z_{(2)}(\C_{L'})$ lies in $\langle C_\lambda\mid \lambda\in L'\rangle$.
By Lemma~\ref{lem:internal-detector}, every simple object of $\Z_{(2)}(\C_{L'})$ is invertible. Thus
$\Z_{(2)}(\C_{L'})$ consists only of invertibles $C_\gamma$ with
$\gamma\in \Lambda_{\mathrm{inv}}\cap \frac{\ell}{2}(L')^*$. For such $C_\gamma$, the double braiding with
$M_\delta\in\C_{L'}$ is $q^{2\langle\gamma,\delta\rangle}\id$, so
$C_\gamma\in\Z_{(2)}(\C_{L'})$ if and only if $q^{2\langle\gamma,\delta\rangle}=1$ for all
$\delta\in\frac{\ell}{2}(L')^*$, i.e.\ $[\gamma]$ lies in the radical of $\overline b$ on $\Lambda(L')$.
Hence, if the radical of $\overline b$ is trivial, then every simple object of $\Z_{(2)}(\C_{L'})$
lies in $\langle C_\lambda\mid\lambda\in L'\rangle$, and the claim follows from
Theorem~\ref{thm:fg-local-simple-main-non-deg}.
\end{proof}


\subsection{General linear supergroup \texorpdfstring{$\fgl(1|1)$}{gl(1|1)}}\label{subsec:example-glsupergroup}
Let $U^E_q(\fgl(1|1))$ denote the unrolled quantum supergroup of $\fgl(1|1)$ at a root of unity $q=e^{2\pi i/r}$ ($r\geq 3$) \cite[\S2.2]{geer2025three}; it is a Hopf superalgebra with even generators $E,G,K^{\pm 1}$ and odd generators $X,Y$.
Let $\C:=\D^{q,\mathrm{int}}$ denote the category of finite-dimensional integral weight $U^E_q(\fgl(1|1))$-supermodules.
By a weight module we mean a finite-dimensional supermodule on which the commuting even generators $E$ and $G$ act semisimply; thus every weight vector $v$ has a weight $\lambda=(\lambda_E,\lambda_G)\in\bC^2$ characterized by
\[
Ev=\lambda_E v,\qquad Gv=\lambda_G v.
\]
The adjective ``integral'' means that all $G$-weights satisfy $\lambda_G\in\bZ$.
Then $\C$ is a ribbon tensor category with enough projectives \cite[Theorem~2.14]{geer2025three}. Throughout, $q^z:=\exp(z\,\log q)$ with $\log q=2\pi i/r$.


\subsubsection{Simple objects and their fusion}
The simple objects of $\C$ are \cite[\S2.3.2]{geer2025three}:
\begin{itemize}
\item the one-dimensional modules $\varepsilon\!\left(\frac{n r}{2},\,b\right)_{\bar p}$, where $n\in\bZ$, $b\in\bZ$ and $\bar p\in\bZ/2\bZ$: it has a basis vector $v$ of parity $\bar p$ such that
\[
  Ev=\tfrac{n r}{2}v,\qquad   
  Gv=bv,\qquad
  Xv=0,\qquad
  Yv=0.
\]
\item the $2$-dimensional quantum Kac modules $V(\alpha,a)_{\bar p}$, where $\alpha\in\bC\setminus\frac{r}{2}\bZ$, $a\in\bZ$ and $\bar p\in\bZ/2\bZ$: it has two basis vectors $v$ (of degree $\bar p$) and $v'$ (of degree $\bar p+\bar 1$) such that
\begin{equation}\label{eq:gl11-Kac-module}
\begin{aligned}
  Ev=\alpha v,\qquad Gv=av,\qquad & Xv=0,\qquad Yv=v',\\
  Ev'=\alpha v',\qquad Gv'=(a-1)v',\qquad & Xv'=[\alpha]_q v,\qquad Yv'=0.
\end{aligned}
\end{equation}
\end{itemize}
Note that the unit object is $\varepsilon(0,0)_{\bar 0}$.

The tensor products between simple objects are computed in \cite[\S2.3.3]{geer2025three}. Here we only recall the formulas relevant for the subsequent analysis of simple current extensions. For $n,n',b,b',a\in\bZ$, $\bar p,\bar p',\bar q\in\bZ/2\bZ$ and $\alpha\in\bC\setminus\frac r2\bZ$ one has
\begin{align*}
\varepsilon\!\left(\frac{n r}{2},b\right)_{\bar p}\otimes
\varepsilon\!\left(\frac{n' r}{2},b'\right)_{\bar p'}
&\cong
\varepsilon\!\left(\frac{(n+n')r}{2},\,b+b'\right)_{\bar p+\bar p'},\\
V(\alpha,a)_{\bar p}\otimes \varepsilon\!\left(\frac{n r}{2},b\right)_{\bar q}
&\cong
V\!\left(\alpha+\frac{n r}{2},\,a+b\right)_{\bar p+\bar q},
\end{align*}
where all sums of $\bar p$'s are taken in $\bZ/2\bZ$.


\subsubsection{Braiding between simples}
The braiding in $\C$ is given by
\[
  c_{V,W}=\tau_{V,W}\circ \widetilde R_{V,W}\circ \Upsilon_{V,W},
\]
where, for weight vectors $v\in V$ of weight $\lambda$ and $w\in W$ of weight $\mu$, we have
\[
  \Upsilon_{V,W}(v\otimes w)
  = q^{-\lambda_E\mu_G-\lambda_G\mu_E}(v\otimes w),
  \qquad
  \tau_{V,W}(v\otimes w)=(-1)^{\bar v\,\bar w}(w\otimes v),
\]
and $\widetilde R$ is left multiplication by
$1+(q-q^{-1})(X\otimes Y)(K\otimes K^{-1})$.

For the one-dimensional modules we have $X=Y=0$, hence $\widetilde R=\id$.
Thus for $\varepsilon=\varepsilon(nr/2,b)_{\bar p}$ and $\varepsilon'=\varepsilon(n'r/2,b')_{\bar p'}$
with basis vectors $v,v'$, respectively, we obtain
\begin{equation}\label{eq:Dq-eps-braiding}
  c_{\varepsilon,\varepsilon'}(v\otimes v')
  = (-1)^{\bar p\,\bar p'} q^{\frac{(-n b'-bn')r}{2}} (v'\otimes v).
\end{equation}
In particular, the double braiding is scalar:
\begin{equation}\label{eq:Dq-eps-double-braiding}
  c_{\varepsilon',\varepsilon}\circ c_{\varepsilon,\varepsilon'}
  = q^{-(n b' + n' b)r}\,\id_{\varepsilon\otimes \varepsilon'}
  = \id_{\varepsilon\otimes \varepsilon'}.
\end{equation}

Let $V=V(\alpha,a)_{\bar p}$ be a $2$-dimensional quantum Kac module and let
$\varepsilon=\varepsilon(nr/2,b)_{\bar q}$ be $1$-dimensional with basis vector $w$.
Choose a homogeneous weight basis $\{v,v'\}$ of $V$ satisfying \eqref{eq:gl11-Kac-module}. 
Since $Xw=Yw=0$, the correction term in $\widetilde R$ vanishes on both $V\otimes \varepsilon$ and
$\varepsilon\otimes V$, hence $\widetilde R=\id$ in both cases. Therefore
\begin{align*}
c_{V,\varepsilon}(v\otimes w)
=(-1)^{\bar p\,\bar q}\,q^{-\alpha b-a\frac{nr}{2}}\,(w\otimes v),
\;\;&
c_{\varepsilon,V}(w\otimes v)
=(-1)^{\bar p\,\bar q}\,q^{-\frac{nr}{2}a-b\alpha}\,(v\otimes w),
\\
c_{V,\varepsilon}(v'\otimes w)
=(-1)^{(\bar p+\bar1)\,\bar q}\,q^{-\alpha b-(a-1)\frac{nr}{2}}\,(w\otimes v'),
\;\;&
c_{\varepsilon,V}(w\otimes v')
=(-1)^{(\bar p+\bar1)\,\bar q}\,q^{-\frac{nr}{2}(a-1)-b\alpha}\,(v'\otimes w).
\end{align*}
In particular, the parity signs cancel in the double braiding. Thus, the double braiding is diagonal on
the weight basis, and since $q^r=1$ we get
\begin{equation}\label{eq:Dq-Kac-eps-double-braiding}
c_{\varepsilon,V}\circ c_{V,\varepsilon}
=
q^{-2\alpha b}\,\id_{V\otimes \varepsilon}.  
\end{equation}


\subsubsection{M\"uger center of $\C$}
\begin{theorem}\label{thm:gl11-muger-center}
The M\"uger center $\Z_{(2)}(\C)$ is a full tensor subcategory whose simple objects are
\[
\varepsilon\!\left(\frac{nr}{2},0\right)_{\bar p}\qquad (n\in\bZ,\ \bar p\in\bZ/2\bZ).
\]
\end{theorem}
\begin{proof}
Let $S$ be a transparent simple object.  If $S=V(\alpha,a)_{\bar p}$, take $T=\varepsilon(0,1)_{\bar 0}$. By \eqref{eq:Dq-Kac-eps-double-braiding},
\[
M_{S,T}=q^{-2\alpha}\,\id_{S\otimes T}.
\]
Since $\alpha\notin \frac r2\bZ$, we have $q^{-2\alpha}\neq 1$, contradiction. 
Hence $S\cong \varepsilon(\frac{nr}{2},b)_{\bar p}$.
Using \eqref{eq:Dq-Kac-eps-double-braiding} with $T=V(\alpha,0)_{\bar 0}$ gives
\[
M_{S,T}=q^{-2\alpha b}\,\id_{S\otimes T}.
\]
Transparency for all such $T$ forces $b=0$. So every transparent simple is of the form
$\varepsilon(\frac{nr}{2},0)_{\bar p}$.

Conversely, let $S=\varepsilon\!\left(\frac{nr}{2},0\right)_{\bar p}$,  and let $T\in \C$ be arbitrary. Choose a basis vector $v$ of $S$, and let $w\in T$ be a homogeneous weight vector of weight $(\alpha,a)$ and parity $\bar q$. Since $X$ and $Y$ act trivially on $S$, the correction term in $\widetilde R$ vanishes on both $S\otimes T$ and $T\otimes S$. Hence
\[
c_{S,T}(v\otimes w)
=
(-1)^{\bar p\,\bar q}q^{-\frac{nr}{2}a}\,(w\otimes v),
\qquad
c_{T,S}(w\otimes v)
=
(-1)^{\bar p\,\bar q}q^{-a\frac{nr}{2}}\,(v\otimes w).
\]
Therefore
\[
M_{S,T}(v\otimes w)
=
q^{-nra}(v\otimes w)
=
(q^r)^{-na}(v\otimes w)
=
v\otimes w.
\]
Since weight vectors span $T$, it follows that $M_{S,T}=\id_{S\otimes T}$. Thus $S\in \Z_{(2)}(\C)$.
\end{proof}


\subsubsection{The group of invertibles}
The invertible simple objects of $\C$ are the one-dimensional modules $\varepsilon\!\left(\tfrac{n r}{2}, b\right)_{\bar p}$ with $n,b\in\bZ$ and $\bar p\in\{\bar 0,\bar 1\}$. 
The group of simple invertible objects of $\C$ is (non-canonically)
isomorphic to $\bZ\times \bZ \times \bZ/2\bZ$ under addition of parameters $(n,b,\bar p)$.

For $\varepsilon=\varepsilon\!\left(\tfrac{n r}{2},b\right)_{\bar p}$ with $n,b\in\bZ$,
\begin{equation}
  c_{\varepsilon,\varepsilon}
  = (-1)^{\bar p} q^{-nrb}\,\id_{\varepsilon\otimes \varepsilon}
  = (-1)^{\bar p}\,\id_{\varepsilon\otimes \varepsilon}.
\end{equation}

Let $\theta$ denote the ribbon twist of $\C$.
For $\varepsilon(\alpha,b)_{\bar p}$, one can compute $\theta$ directly from
the definition using the explicit right evaluation/coevaluation maps:
\[
  \ev_V(v\otimes f)=(-1)^{\bar f\,\bar v}f(Kv),
  \qquad
  \coev_V(1)=\sum_i (-1)^{\bar v_i} v_i^*\otimes K^{-1}v_i.
\]
Using the same calculation as in \cite[Theorem~2.14]{geer2025three}, we obtain
\begin{equation}\label{eq:Dq-eps-twist}
  \theta_{\varepsilon(\alpha,b)_{\bar p}} = q^{-2\alpha b+\alpha}\,\id_{\varepsilon(\alpha,b)_{\bar p}}.
\end{equation}
For $\alpha=\frac{nr}{2}$ (with $n,b\in\bZ$), this becomes
\[
  \theta_{\varepsilon(\frac{nr}{2},b)_{\bar p}}
  = q^{-nrb+\frac{nr}{2}}\,\id
  = q^{\frac{nr}{2}}\,\id
  = (-1)^n\,\id,
\]
where we used $q^r=1$, so $q^{-nrb}=1$, and $q^{nr/2}=e^{\pi i n}=(-1)^n$.

We now determine exactly which simple current algebras are commutative.
Define the quadratic form $\mathsf{q}:\bZ\times \bZ \times \bZ/2\bZ\to \bC^\times$ by
\begin{equation}\label{eq:gl11-quad-form}
  \mathsf{q}(n,b,\bar p) = (-1)^{\bar p}.
\end{equation}
Hence
\begin{equation}\label{eq:gl11-comm-subgroups}
\mathsf{q}|_{\Gamma}\equiv 1
\iff
\Gamma=L\times\{\bar 0\}\ \text{for some }L\le \bZ^2.
\end{equation}
Thus the commutative simple current algebras in $\widehat{\C}$ are precisely those indexed by subgroups $L\le\bZ^2$:
\[
A_\Gamma=\bigoplus_{(n,b)\in L}\varepsilon\!\left(\tfrac{n r}{2},b\right)_{\bar 0}
\quad
(\Gamma=L\times\{\bar 0\}).
\]
We denote by $\pi_n$ and $\pi_b$ the projections onto the first and second coordinate of $\bZ\times \bZ \times \bZ/2\bZ$.

\begin{proposition}\label{prop:gl11-local-ribbon-criterion}
Let $\Gamma=L\times\{\bar 0\}\le \Inv(\C)$ with $L\le\bZ^2$, and set $A_\Gamma=\bigoplus_{\gamma\in\Gamma}\gamma\in\widehat{\C}$.
Write each $\gamma\in\Gamma$ as
$\gamma=\varepsilon\!\left(\frac{n(\gamma)r}{2},b(\gamma)\right)_{\bar 0}$. Then,
$\fg\text{-}\widehat{\C}_{A_\Gamma}^{\loc}$ is braided tensor.
It is ribbon iff $\theta_\gamma=\id_\gamma$ for all $\gamma\in\Gamma$, equivalently iff
$n(\gamma)$ is even for every $\gamma\in\Gamma$.
\end{proposition}
\begin{proof}
By Theorem~\ref{thm:fg-local-simple-main}(a),
$\fg\text{-}\widehat{\C}_{A_\Gamma}^{\loc}$ is braided tensor.
By Theorem~\ref{thm:fg-local-simple-main}(b), it is ribbon iff
$\theta_\gamma=\id_\gamma$ for all $\gamma\in\Gamma$.
For $\gamma=\varepsilon(\frac{n(\gamma)r}{2},b(\gamma))_{\bar 0}$, the twist formula above gives
$\theta_\gamma=(-1)^{n(\gamma)}\id_\gamma$, so this is equivalent to
$n(\gamma)$ being even for every $\gamma\in\Gamma$.
\end{proof}


\subsubsection{Finiteness of $\fg\text{-}\widehat{\C}_A$ and $\fg\text{-}\widehat{\C}_A^{\loc}$}
The full module category $\fg\text{-}\widehat{\C}_A$ is never finite, but the local subcategory may be.
Let $\Gamma=L\times\{\bar 0\}\le \Inv(\C)$ with $L\le\bZ^2$, and let
$A_\Gamma=\bigoplus_{\gamma\in\Gamma}\gamma$.
The $\Gamma$-orbits on Kac simples have a continuous $\alpha$-parameter modulo the
discrete subgroup $\frac r2\pi_n(L)$, hence $\Irr(\C)/\Gamma$ is infinite
and $\fg\text{-}\widehat{\C}_{A_\Gamma}$ has infinitely many simple objects by Theorem~\ref{thm:finite-tensor-category-criterion}.

For $\gamma=\varepsilon(\frac{nr}{2},b)_{\bar 0}\in\Gamma$ and a Kac simple $V(\alpha,a)_{\bar p}$,
\eqref{eq:Dq-Kac-eps-double-braiding} gives $M_{\gamma,V}=q^{-2\alpha b}\id$. Write
\[
B:=\pi_b(\Gamma)\le \bZ.
\]
Hence $V(\alpha,a)_{\bar p}$ is local iff $q^{-2\alpha b}=1$ for all $b\in B$.
If $B=d\bZ$ ($d>0$), this is equivalent to
\[
\alpha\in \frac{r}{2d}\bZ,
\]
while for $B=0$ there is no restriction on $\alpha$.
So locality replaces the continuous $\alpha$-family by an arithmetic condition determined by $B$. In our $\fgl(1|1)$ setting, $\C_\Gamma$ contains all one-dimensional simples $\varepsilon(\frac{nr}{2},b)_{\bar p}$. It contains Kac simples $V(\alpha,a)_{\bar p}$ exactly for $\alpha\in\frac{r}{2d}\bZ\setminus\frac r2\bZ$ when $B=d\bZ$ and $d>0$, while for $B=0$ there is no extra restriction beyond $\alpha\notin\frac r2\bZ$.
Thus locality removes the continuous $\alpha$-parameter exactly when $\pi_b(L)\neq 0$. Finiteness of $\fg\text{-}\widehat{\C}_{A_\Gamma}^{\loc}$ is stronger, and will hold exactly when $L$ has finite index in $\bZ^2$.

\begin{proposition}\label{prop:gl11-local-finite-criterion}
Let $\Gamma\le \Inv(\C)$ satisfy $\mathsf{q}|_\Gamma\equiv 1$, write
$\Gamma=L\times\{\bar 0\}$ with $L\le\bZ^2$, and let
$A_\Gamma=\bigoplus_{\gamma\in\Gamma}\gamma$.
Then $\fg\text{-}\widehat{\C}_{A_\Gamma}^{\loc}$ is finite if and only if $[\bZ^2:L]<\infty$.
\end{proposition}
\begin{proof}
By \eqref{eq:Dq-eps-double-braiding}, every one-dimensional simple
$\varepsilon(\frac{nr}{2},b)_{\bar p}$ centralizes $\Gamma$, hence belongs to $\C_\Gamma$.
The $\Gamma$-action on these simples is translation by $L$ on $(n,b)\in\bZ^2$, so their
orbit set is
\[
(\bZ^2/L)\times (\bZ/2\bZ).
\]
Therefore, if $[\bZ^2:L]=\infty$, then $\Irr(\C_\Gamma)/\Gamma$ is already infinite.

Conversely, assume $[\bZ^2:L]<\infty$. Then
\[
N:=\pi_n(L)=n_0\bZ
\qquad\text{and}\qquad
B:=\pi_b(L)=d\bZ
\]
with $n_0,d>0$. For a Kac simple $V(\alpha,a)_{\bar p}$,
\eqref{eq:Dq-Kac-eps-double-braiding} gives locality if and only if
$q^{-2\alpha b}=1$ for all $b\in B$, equivalently
\[
\alpha\in \frac{r}{2d}\bZ .
\]
Thus the local Kac simples are parametrized by
\[
\left(\frac{r}{2d}\bZ\setminus \frac r2\bZ\right)\times \bZ \times \bZ/2\bZ .
\]

Now the $\Gamma$-action on Kac parameters is induced by
\[
(n,b)\in L
\quad\mapsto\quad
\left(\alpha,a\right)\longmapsto \left(\alpha+\frac{nr}{2},\,a+b\right).
\]
Let
\[
L' := \left\{\left(\frac{nr}{2},\,b\right)\in \frac r2\bZ\times \bZ \,\middle|\, (n,b)\in L\right\}
\subset \frac{r}{2d}\bZ\times \bZ .
\]
Since $L$ has finite index in $\bZ^2$, it has rank $2$, hence $L'$ is a rank-$2$ sublattice of
\[
\frac r2\bZ\times \bZ .
\]
Moreover,
\[
\left[\frac{r}{2d}\bZ\times \bZ : \frac r2\bZ\times \bZ\right]=d<\infty,
\]
so $L'$ is also a finite-index sublattice of $\frac{r}{2d}\bZ\times \bZ$. Therefore
\[
\left(\frac{r}{2d}\bZ\times \bZ\right)\big/ L'
\]
is finite. Since the excluded subset $\frac r2\bZ\times \bZ$ is $L'$-stable, the same is true after restricting to
\[
\left(\frac{r}{2d}\bZ\setminus \frac r2\bZ\right)\times \bZ .
\]
Hence the local Kac simples contribute only finitely many $\Gamma$-orbits; the parity parameter contributes only the finite factor $\bZ/2\bZ$.

The one-dimensional part also contributes finitely many $\Gamma$-orbits because
\[
(\bZ^2/L)\times (\bZ/2\bZ)
\]
is finite. Hence $\Irr(\C_\Gamma)/\Gamma$ is finite.

Now apply Theorem~\ref{thm:fg-local-simple-main}(c).
\end{proof}

In summary, finite-index subgroups $L\le\bZ^2$ yield explicit finite braided (and, when all $n(\gamma)$ are even, ribbon) local-module categories in the $\fgl(1|1)$ setting.


\appendix 


\section{Non-braided version of Etingof-Penneys' lemma}\label{app:EP-lemma}
In Appendix \ref{app:EP-lemma}, by an {\em abelian monoidal category}, we mean an abelian category endowed with a structure of a monoidal category such that the monoidal product is additive and right exact in each variable. 


\subsection{Etingof-Penneys' Lemma}
An object $X$ of an abelian monoidal category $\C$ is said to be {\em left flat} if the endofunctor $X \otimes (-)$ on $\C$ is exact. Similarly, $X$ is said to be {\em right flat} if $(-) \otimes X$ is exact. We say that $\C$ has enough left (right) flat objects if every object of $\C$ is a quotient of a left (right) flat object of $\C$. A {\em flat} object is a left and right flat object.

\begin{lemma}
  \label{lem:Etingof-Penneys}
  Let $\C$ be an abelian monoidal category, and let
  $0 \to Y \xrightarrow{i} Z \xrightarrow{p} X \to 0$ be a short exact sequence in $\C$.
  Assume that $\C$ has enough left flat objects and enough right flat objects.
  \begin{enumerate}
  \item If $X$ and $Y$ are left rigid and right flat, then $Z$ is left rigid.
  \item If $X$ and $Z$ are left rigid, then $Y$ is left rigid.
  \item If $Y$ and $Z$ are left rigid and $i^*$ is an epimorphism, then $X$ is left rigid.
  \end{enumerate}
\end{lemma}

This lemma was established by Etingof and Penneys under the assumption that $\C$ is braided \cite[Lemma 4.2]{etingof2024rigidity}. They also remarked that the use of the braiding is not essential \cite[Remark 4.4]{etingof2024rigidity}. One of the aims of this Appendix is to demonstrate that \cite[Lemma 4.2]{etingof2024rigidity} holds in the above form in the absence of a braiding by tracing the original proof carefully.

As is well-known, a left (right) rigid object is left (right) flat (see Lemma~\ref{lem:rigid-implies-flat} below). One technical difficulty in the non-braided case is that a left (right) rigid object need not be right (left) rigid. For example, given a left rigid object $Y$, a map between Yoneda $\Ext^1$ groups induced by the functor $Y^* \otimes (-)$ is considered in the proof of \cite[Lemma 4.2]{etingof2024rigidity}.  However, since the functor $Y^* \otimes (-)$ is not necessarily exact in our non-braided setting, we will face some technical problems. For completeness, in \S\ref{subsec:Yoneda-Ext1}, we discuss maps between Yoneda $\Ext^1$ groups induced by an additive functor which is not necessarily exact.

Another point to note is that techniques of the Tor functor were used in the proof of \cite[Lemma 4.2]{etingof2024rigidity}. Instead of justifying the use of the Tor functor in our setting, we have chosen to assume that $\C$ has enough left flat objects and enough right flat objects, and to verify some necessary technical lemmas on flat objects in a direct way in \S\ref{subsec:EP-lemma-Tor}.
The assumption on $\C$ does not matter for our main application to the category of finitely generated modules over a central commutative algebra in a braided tensor category in \S\ref{sec:framework}, since every object of such a category is a quotient of a flat (in fact, rigid) object.

An important consequence of the above Lemma is:

\begin{theorem}\label{thm:EP-rigidity-main}       
Let $\C$ be an abelian monoidal category in which every object has finite length. Assume that $\C$ has enough left flat objects and enough right flat objects.
If every simple object of $\C$ is rigid, then $\C$ is rigid.
\end{theorem}
\begin{proof}
It is obvious that the zero object of $\C$ is rigid. Let $M$ be a non-zero object of $\C$.
We prove the rigidity of $M$ by induction on the length of $M$.
If $M$ has length $1$, then $M$ is simple and hence rigid by assumption.

Now let $M$ have length $n>1$ and assume that every object of $\C$ of length $<n$ is rigid.
Choose a short exact sequence $0\to N\to M\to S\to 0$ with $S$ simple.
By induction, $N$ is rigid, and by assumption $S$ is rigid.
In particular, $N$ and $S$ are flat (see Lemma~\ref{lem:rigid-implies-flat} below).
Applying Lemma~\ref{lem:Etingof-Penneys} to this exact sequence shows that $M$ is left rigid.
Applying Lemma~\ref{lem:Etingof-Penneys} in the reversed monoidal category $\C^{\mathrm{rev}}$
shows that $M$ is right rigid.
Hence $M$ is rigid.
\end{proof}


\subsection{Lemmas on rigid objects}
We first collect useful lemmas on rigid objects. The following lemma is found in the proof of \cite[Lemma 4.2]{etingof2024rigidity}:

\begin{lemma}
  \label{lem:duality-modification}
  Let $\C$ be a monoidal category, let $L$ and $R$ be objects of $\C$, and let $\varepsilon : L \otimes R \to \unit$ and $\eta : \unit \to R \otimes L$ be morphisms in $\C$. Suppose that
  \begin{equation*}
    \xi := (\varepsilon \otimes \id_L) \circ (\id_L \otimes \eta)
    \quad \text{and} \quad
    \zeta := (\id_R \otimes \varepsilon) \circ (\eta \otimes \id_R)
  \end{equation*}
  are invertible in $\C$. Then we have
  \begin{equation*}
    \varepsilon \circ (\xi^{-1} \otimes \id_R)
    = \varepsilon \circ (\id_L \otimes \zeta^{-1}),
    \quad (\zeta^{-1} \otimes \id_L) \circ \eta
    = (\id_R \otimes \xi^{-1}) \circ \eta.
  \end{equation*}
  Let $\varepsilon'$ and $\eta'$ be both sides of the former and the latter equation, respectively. Then the triples $(L, \varepsilon', \eta)$ and $(L, \varepsilon, \eta')$ are left dual objects of $R$.
\end{lemma}
\begin{proof}
  By the functorial property of $\otimes$, it is easy to verify
  \begin{equation*}
    \varepsilon \circ (\xi \otimes \id_R)
    = (\varepsilon \otimes \varepsilon) \circ (\id_L \otimes \eta \otimes \id_R)
    = \varepsilon \circ (\id_L \otimes \zeta).
  \end{equation*}
  By composing $\xi^{-1} \otimes \zeta^{-1}$ from the right, we obtain the first equation of the statement. The second one is proved in a similar way. Now we have
  \begin{gather*}
    (\varepsilon' \otimes \id_L) \circ (\id_L \otimes \eta)
    = (\varepsilon \otimes \id_L) \circ (\xi^{-1} \otimes \eta)
    = \xi \circ \xi^{-1} = \id_L, \\
    (\id_R \otimes \varepsilon') \circ (\eta \otimes \id_R)
    = (\id_R \otimes \varepsilon) \circ (\eta \otimes \zeta^{-1})
    = \zeta \circ \zeta^{-1} = \id_R,
  \end{gather*}
  which means that $(L, \varepsilon', \eta)$ is a left dual object of $R$.
  The case of $(L, \varepsilon, \eta')$ is proved in a similar way.
\end{proof}

Let $A$ and $B$ be objects of $\C$, and let $f: A \to B$ be a morphism in $\C$. To discuss the rigidity of $B$ from that of $A$ or its converse, we note the following lemma: Let $A'$ and $B'$ also be objects of $\C$, and let
\begin{equation*}
  \varepsilon_A : A' \otimes A \to \unit,
  \quad \eta_A : \unit \to A \otimes A',
  \quad \varepsilon_B : B' \otimes B \to \unit,
  \quad \eta_B : \unit \to B \otimes B'
\end{equation*}
and $f' : B' \to A'$ be morphisms in $\C$ such that
\begin{equation*}
  \varepsilon_B \circ (\id_{B'} \otimes f)
  = \varepsilon_A \circ (f' \otimes \id_A),
  \quad
  (f \otimes \id_{A'}) \circ \eta_A
  = (\id_B \otimes f') \circ \eta_B.
\end{equation*}

\begin{lemma}
  \label{lem:duality-transfer-by-morphism}
  In the above situation, we have
  \begin{gather*}
    f \circ (\id_A \otimes \varepsilon_A) \circ (\eta_A \otimes \id_A)
    = (\id_B \otimes \varepsilon_B) \circ (\eta_B \otimes \id_B) \circ f, \\
    (\varepsilon_A \otimes \id_{A'}) \circ (\id_{A'} \otimes \eta_A) \circ f'
    = f' \circ (\varepsilon_B \otimes \id_{B'}) \circ (\id_{B'} \otimes \eta_B).
  \end{gather*}
\end{lemma}

For example, if $(A', \varepsilon_A, \eta_A)$ is a left dual object of $A$, $f$ is epic and $f'$ is monic, then we can conclude that $(B', \varepsilon_B, \eta_B)$ is a left dual object of $B$ by this lemma.

\begin{proof}
  The first equation is proved as follows:
  \begin{align*}
    & f \circ (\id_A \otimes \varepsilon_A) \circ (\eta_A \otimes \id_A) \\
    & = (\id_B \otimes \varepsilon_A \otimes \id_A)
      \circ (f \otimes \id_{A'} \otimes \id_A)
      \circ (\eta_A \otimes \id_A) \\
    & = (\id_B \otimes \varepsilon_A \otimes \id_A)
      \circ (\id_A \otimes f' \otimes \id_A)
      \circ (\eta_B \otimes \id_A) \\
    & = (\id_B \otimes \varepsilon_A \otimes \id_A)
      \circ (\id_A \otimes \id_{B'} \otimes f)
      \circ (\eta_B \otimes \id_A) \\
    & = (\id_B \otimes \varepsilon_B) \circ (\eta_B \otimes \id_B) \circ f.
  \end{align*}
  The second one is proved by a straightforward computation in a similar way.
\end{proof}

Let $X$ be a left rigid object of $\C$. As is well-known, the functor $X^* \otimes (-)$ is left adjoint to the functor $X \otimes (-)$. The unit and the counit of this adjunction are given by the coevaluation and the evaluation, respectively. This observation yields the following well-known fact:

\begin{lemma}
  \label{lem:rigid-implies-flat}
  In an abelian monoidal category, a left rigid object is left flat.
  Similarly, a right rigid object is right flat.
\end{lemma}
\begin{proof}
  The functor $X \otimes (-)$ is right exact by our definition of an abelian monoidal category. As it has a
  left adjoint $X^* \otimes (-)$, the functor $X \otimes (-)$ is also left exact. Thus $X$ is left flat. The right rigid case is
  similar.
\end{proof}

Let $f: X \to Y$ be a morphism in a monoidal category $\C$ between left rigid objects $X$ and $Y$. Then the left dual morphism of $f$ is defined and denoted by
\begin{equation*}
  f^* := (\ev_Y \otimes \id_{X^*})(\id_{Y^*} \otimes f \otimes \id_{X^*}) (\id_{Y^*} \otimes \coev_X) : Y^* \to X^*.
\end{equation*}
The zig-zag relations imply
\begin{equation*}
  \ev_X (f^* \otimes \id_{X^*}) = \ev_Y (\id_{Y^*} \otimes f),
  \quad (\id_{Y} \otimes f^*) \coev_Y = (f \otimes \id_{X^*}) \coev_X.
\end{equation*}
The following fact was remarked in the proof of \cite[Lemma 4.2]{etingof2024rigidity}:

\begin{lemma}
  \label{lem:dual-of-epi}
  If $f : X \to Y$ is an epimorphism in an abelian monoidal category $\C$ between left rigid objects $X$ and $Y$, then $f^*$ is monic.
\end{lemma}
\begin{proof}
  We fix an object $T \in \C$ and consider the commutative diagram
  \begin{equation*}
    \begin{tikzcd}
      \Hom_{\C}(T, Y^*)
      \arrow[r, "{\cong}"]
      \arrow[d, "{\Hom_{\C}(T, f^*)}"']
      & \Hom_{\C}(T \otimes Y, \unit)
      \arrow[d, "{\Hom_{\C}(\id_T \otimes f, \unit)}"] \\
      \Hom_{\C}(T, X^*)
      \arrow[r, "{\cong}"]
      & \Hom_{\C}(T \otimes X, \unit).
    \end{tikzcd}
  \end{equation*}
  Since $f$ is an epimorphism, and since the monoidal product in $\C$ is assumed to be right exact, the
  morphism $\id_T \otimes f$ is an epimorphism. This implies that the right vertical arrow of the above diagram
  is injective, and thus so is the left vertical arrow. The proof is done.
\end{proof}


\subsection{Lemmas on flat objects}
\label{subsec:EP-lemma-Tor}

Let $\C$ be an abelian monoidal category.
We provide some lemmas on flat objects.

\begin{lemma}
  \label{lem:flatness-criterion-3}
  Let $0 \to Y \to Z \to X \to 0$ be an exact sequence in $\C$, and let $A$ be an object of $\C$.
  \begin{enumerate}
  \item If $X$ is right flat and $A$ is a quotient of a left flat object, then $0 \to A \otimes Y \to A \otimes Z \to A \otimes X \to 0$ is exact.
  \item If $X$ is left flat and $A$ is a quotient of a right flat object, then $0 \to Y \otimes A \to Z \otimes A \to X \otimes A \to 0$ is exact.
  \end{enumerate}
\end{lemma}
\begin{proof}
(a) We choose an epimorphism $p : F \to A$ with $F$ left flat and let $i: K \to F$ be the kernel of $p$ so that we have an exact sequence $0 \to K \to F \to A \to 0$. We consider the following commutative diagram:
\begin{equation*}
  \begin{tikzcd}[row sep = 16pt]
    & \Ker(i \otimes \id_Y) \arrow[d, hook] \arrow[r]
    & \Ker(i \otimes \id_Z) \arrow[d, hook] \arrow[r]
    & \Ker(i \otimes \id_X) \arrow[d, hook] \\
    & K \otimes Y \arrow[r] \arrow[d, "{i \otimes \id_Y}"]
    & K \otimes Z \arrow[r] \arrow[d, "{i \otimes \id_Z}"]
    & K \otimes X \arrow[r] \arrow[d, "{i \otimes \id_X}"]
    & 0 \\
    0 \arrow[r]
    & F \otimes Y \arrow[r] \arrow[d, two heads, "{\, p \otimes \id_Y}"]
    & F \otimes Z \arrow[r] \arrow[d, two heads, "{\, p \otimes \id_Z}"]
    & F \otimes X \arrow[r] \arrow[d, two heads, "{\, p \otimes \id_X}"]
    & 0 \\
    & A \otimes Y \arrow[r] 
    & A \otimes Z \arrow[r] 
    & A \otimes X \arrow[r] 
    & 0 
  \end{tikzcd}
\end{equation*}
Here, for brevity, some of vertical arrows $0 \to$ and $\to 0$ are omitted, and the symbols $\hookrightarrow$ and $\twoheadrightarrow$ are used to mean a monomorphism and an epimorphism, respectively. The third row is exact since $F$ is left flat. The other rows and all columns are also exact by the right exactness of the tensor product.
Since $X$ is right flat, $i \otimes \id_X$ is monic. Thus the snake lemma yields the exact sequence $0 \to A \otimes Y \to A \otimes Z \to A \otimes X \to 0$.

(b) Apply (a) to $\mathcal{C}^{\rev}$.
\end{proof}

\begin{lemma}
  \label{lem:flatness-criterion-2}
  Let $0 \to Y \to Z \to X \to 0$ be an exact sequence in $\C$.
  \begin{itemize}
  \item [(a)] If $\C$ has enough left flat objects and both $X$ and $Y$ are right flat, then $Z$ is right flat.
  \item [(b)] If $\C$ has enough right flat objects and both $X$ and $Y$ are left flat, then $Z$ is left flat.
  \end{itemize}
\end{lemma}
\begin{proof}
(a) We assume that $\C$ has enough left flat objects and both $X$ and $Y$ are right flat.
    Let $0 \to B \to C \to A \to 0$ be an exact sequence in $\C$. We consider the following commutative diagram obtained by tensoring $0 \to B \to C \to A \to 0$ and $0 \to Y \to Z \to X \to 0$:
\begin{equation*}
  \begin{tikzcd}
    & B \otimes Y \arrow[r, hook] \arrow[d, hook]
    & C \otimes Y \arrow[r, two heads] \arrow[d, hook]
    & A \otimes Y \arrow[d, hook] \\ 
    & B \otimes Z \arrow[r, hook] \arrow[d, two heads]
    & C \otimes Z \arrow[r, two heads] \arrow[d, two heads]
    & A \otimes Z \arrow[d, two heads] \\ 
    & B \otimes X \arrow[r, hook] 
    & C \otimes X \arrow[r, two heads] 
    & A \otimes X 
  \end{tikzcd}
\end{equation*}
The first and the third row are exact since $X$ and $Y$ are right flat. As $\C$ has enough left flat objects and $X$ is right flat, by Lemma~\ref{lem:flatness-criterion-3}, the columns are exact. Thus, by the nine lemma, the middle row is also exact. Therefore $Z$ is right flat.

(b) Apply (a) to $\mathcal{C}^{\rev}$.
\end{proof}


\subsection{Reminder on the Yoneda Ext}
\label{subsec:Yoneda-Ext1}
The Ext functor is defined in an arbitrary abelian category by assigning the set of equivalence classes of exact sequences of a particular form (so-called Yoneda Ext). Here we give a reminder on basics on the Yoneda Ext functor, restricting ourselves to $\Ext^1$.

\subsubsection{Definition of the Yoneda Ext functor}

Let $\mathcal{A}$ be an abelian category.
For objects $X$ and $Y$ of $\mathcal{A}$, the set $\Ext^1_{\mathcal{A}}(X, Y)$ is defined to be the set of equivalence classes of exact sequences of the form $0 \to Y \to E \to X \to 0$ for some $E \in \mathcal{A}$. When $\mathcal{A}$ is clear from the context, we often write it as $\Ext^1(X, Y)$. The construction of the set $\Ext^1(X, Y)$ in fact extends to a functor from $\mathcal{A}^{\op} \times \mathcal{A}$ to the category $\mathbf{Ab}$ of abelian groups. The functorial property and the abelian group structure are explained as follows:
\begin{enumerate}
\renewcommand{\labelenumi}{\textup{(\arabic{enumi})}}
\item Given a morphism $f : X' \to X$ in $\mathcal{A}$ and an object $Y \in \mathcal{A}$, the map\footnote{To distinguish it from duals in a monoidal category, we use $\star$ to mean the maps between Ext groups induced by a morphism. More specifically, $(-)^{\star}$ is used in the contravariant case, while $(-)_{\star}$ is used in the covariant case.}
  \begin{equation*}
    f^\star : \Ext^1(X, Y) \to \Ext^1(X', Y)
  \end{equation*}
  is defined as follows: Let $e$ be an element of $\Ext^1(X, Y)$ represented by an exact sequence $0 \to Y \to E \to X \to 0$. We consider the commutative diagram
  \begin{equation}
    \label{eq:Ext1-pullback}
    \begin{tikzcd}
      0 \arrow[r]
      & Y \arrow[r] \arrow[d, equal]
      & E \times_X X' \arrow[r] \arrow[d] \arrow[rd, phantom, "\text{(PB)}"]
      & X' \arrow[r] \arrow[d, "{f}"] & 0 \\
      0 \arrow[r] &  Y \arrow[r] & E \arrow[r] & X \arrow[r] & 0
    \end{tikzcd}
  \end{equation}
  with exact rows, where (PB) means a pullback square. The element $f^\star(e)$ is defined to be the element represented by the first row of the above diagram.
\item Given a morphism $g : Y \to Y'$ in $\mathcal{A}$ and an object $X \in \mathcal{A}$, the map
  \begin{equation*}
    g_{\star} : \Ext^1(X, Y) \to \Ext^1(X, Y')
  \end{equation*}
  is defined dually to (1): if $e \in \Ext^1(X, Y)$ is represented by $0 \to Y \to E \to X \to 0$, consider the commutative diagram
  \begin{equation}
    \label{eq:Ext1-pushout}
    \begin{tikzcd}
      0 \arrow[r]
      & Y \arrow[r] \arrow[d, "{g}"']
      \arrow[rd, phantom, "\text{(PO)}"]
      & E \arrow[r] \arrow[d]
      & X \arrow[r] \arrow[d, equal] & 0 \\
      0 \arrow[r] & Y' \arrow[r] & E \amalg_Y Y' \arrow[r] & X \arrow[r] & 0
    \end{tikzcd}
  \end{equation}
  with exact rows, where (PO) means a pushout square. The element $g_{\star}(e)$ is defined to be the element represented by the second row of the above diagram.
\item For morphisms $f$ and $g$ as in (1) and (2), one can show that $f^\star$ and $g_{\star}$ commute. Now we define the map $\Ext^1(f, g)$ by
  \begin{equation*}
    \Ext^1(f, g) := f^\star \circ g_{\star} = g_{\star} \circ f^\star : \Ext^1(X, Y) \to \Ext^1(X', Y').
  \end{equation*}
\item For $e_1, e_2 \in \Ext^1(X, Y)$, their sum is defined by
  \begin{equation*}
    e_1 + e_2 := \Ext^1(\mathrm{diag}_X, \mathrm{sum}_Y)(e_1 \oplus e_2)
    \quad \text{(the Baer sum)},
  \end{equation*}
  where $\mathrm{diag}_X : X \to X \oplus X$ and $\mathrm{sum}_Y : Y \oplus Y \to Y$ are the diagonal morphism and the sum, respectively, and $e_1 \oplus e_2 \in \Ext^1(X \oplus X, Y \oplus Y)$ is the element represented by the exact sequence obtained by taking the direct sum of exact sequences representing $e_1$ and $e_2$.
\item If $e \in \Ext^1(X, Y)$ is represented by $\displaystyle 0 \to Y \mathrel{\mathop{\to}^i} E \to X \to 0$, then the inverse of $e$ (with respect to the Baer sum) is represented by the sequence
  \begin{equation*}
    \begin{tikzcd}
      0 \arrow[r] & Y \arrow[r, "-i"] & E \arrow[r] & X \arrow[r] & 0.
    \end{tikzcd}
  \end{equation*}
\end{enumerate}

The following standard lemma is useful for computing $f^{\star}$ and $g_{\star}$.

\begin{lemma}
\label{lem:Ext1-pullback-pushout}
Let $0 \to Y \to E \to X \to 0$ be an exact sequence in $\mathcal{A}$, and let $e$ be an element of $\Ext^1(X, Y)$ represented by this exact sequence.
\begin{itemize}
\item [(a)] If there is a commutative diagram
\begin{equation} \label{eq:Ext1-pullback-2}
    \begin{tikzcd}
      0 \arrow[r]
      & Y \arrow[r] \arrow[d, equal]
      & F \arrow[r] \arrow[d]
      & X' \arrow[r] \arrow[d, "{f}"] & 0 \\
      0 \arrow[r] &  Y \arrow[r] & E \arrow[r] & X \arrow[r] & 0
    \end{tikzcd}
\end{equation}
with exact rows, then the top row represents $f^{\star}(e)$.
\item[(b)] If there is a commutative diagram
\begin{equation} \label{eq:Ext1-pushout-2}
    \begin{tikzcd}
      0 \arrow[r]
      & Y \arrow[r] \arrow[d, "{g}"']
      & E \arrow[r] \arrow[d]
      & X \arrow[r] \arrow[d, equal] & 0 \\
      0 \arrow[r] & Y' \arrow[r] & G \arrow[r] & X \arrow[r] & 0
    \end{tikzcd}
\end{equation}
with exact rows, then the bottom row represents $g_{\star}(e)$.
\end{itemize}
\end{lemma}
\begin{proof}
We only give a proof of (a) since (b) can be verified in a similar way.
We assume that there is a commutative diagram as in (a), and label the arrows in the diagrams \eqref{eq:Ext1-pullback} and \eqref{eq:Ext1-pullback-2} as follows:
\begin{equation*}
  \text{\eqref{eq:Ext1-pullback}: }
    \begin{tikzcd}
      Y \arrow[r, hook, "{a'}"] \arrow[d, equal]
      & E \times_X X' \arrow[r, two heads, "{\mathrm{pr}_2}"] \arrow[d, "{\mathrm{pr}_1}"]
      & X' \arrow[d, "{f}"] \\
      Y \arrow[r, hook, "{a}"] & E \arrow[r, two heads, "{b}"] & X
    \end{tikzcd} \qquad \qquad
  \text{\eqref{eq:Ext1-pullback-2}: }
    \begin{tikzcd}
      Y \arrow[r, hook, "{a''}"] \arrow[d, equal]
      & F \arrow[r, "{p_2}"] \arrow[d, "{p_1}"]
      & X' \arrow[d, "{f}"] \\
      Y \arrow[r, hook, "{a}"] & E \arrow[r, two heads, "{b}"] & X
    \end{tikzcd}
\end{equation*}
By the universal property of the pullback, there exists a unique morphism $s : F \to E \times_X X'$ such that $p_i = \mathrm{pr}_i \circ s$ ($i = 1, 2$). Since $\mathrm{pr}_1 s a'' = p_1 a'' = a = \mathrm{pr}_1 a'$ and $\mathrm{pr}_2 s a'' = p_2 a'' = 0 = \mathrm{pr}_2 a'$, we have $s a'' = a'$. Hence we obtain the commutative diagram
  \begin{equation*}
    \begin{tikzcd}
      0 \arrow[r] &  Y \arrow[r, "{a''}"] \arrow[d, equal]
      & F \arrow[r, "{p_2}"] \arrow[d, "{s}"] & X' \arrow[r] \arrow[d, equal] & 0 \\
      0 \arrow[r] & Y \arrow[r, "{a'}"] & E \times_X X' \arrow[r, "{\mathrm{pr}_2}"] & X' \arrow[r] & 0
    \end{tikzcd}
  \end{equation*}
  with exact rows. The five lemma implies that $s$ is an isomorphism. Since $f^{\star}(e)$ is represented by the top row of \eqref{eq:Ext1-pullback} by definition, $f^{\star}(e)$ is also represented by the top row of  \eqref{eq:Ext1-pullback-2}.
\end{proof}

It is known that the functor $\Ext^1 : \mathcal{A}^{\op} \times \mathcal{A} \to \mathbf{Ab}$ is additive in each variable. Thus, in particular, there are canonical isomorphisms
\begin{align*}
  \Ext^1(X \oplus Y, Z) & \cong \Ext^1(X, Z) \oplus \Ext^1(Y, Z), \\
  \Ext^1(X, Y \oplus Z) & \cong \Ext^1(X, Y) \oplus \Ext^1(X, Z)
\end{align*}
for $X, Y, Z \in \mathcal{A}$. We note the following well-known relation between the above isomorphisms and the Baer sum.

\begin{lemma}
  \label{lem:Ext1-Baer-sum}
  The following diagram is commutative:
  \begin{equation*}
    \begin{tikzcd}[column sep = 24pt, row sep = 8pt]
      & \Ext^1(X, Y) \oplus \Ext^1(X, Y)
      \arrow[ld, "{\cong}"'] \arrow[rd, "{\cong}"]
      \arrow[dd, "\text{the Baer sum}"] \\
      \Ext^1(X \oplus X, Y) \arrow[rd, "{(\mathrm{diag}_X)^\star}"']
      & & \Ext^1(X, Y \oplus Y) \arrow[ld, "{(\mathrm{sum}_Y)_{\star}}"] \\
      & \Ext^1(X, Y)
    \end{tikzcd}
  \end{equation*}
\end{lemma}

\subsubsection{Maps induced by functors}

Let $F : \mathcal{A} \to \mathcal{B}$ be an additive functor between abelian categories. For brevity, we define $\mathcal{E}_{F}$ to be the class of pairs $(X, Y)$ consisting of objects of $\mathcal{A}$ such that $F$ preserves extensions of $X$ by $Y$, that is, for every exact sequence in $\mathcal{A}$ of the form $0 \to Y \to E \to X \to 0$, the sequence $0 \to F(Y) \to F(E) \to F(X) \to 0$ is exact. For each pair $(X, Y) \in \mathcal{E}_F$, we have an obvious map
\begin{equation*}
    \widetilde{F}_{X,Y} : \Ext^1_{\mathcal{A}}(X, Y) \to \Ext^1_{\mathcal{B}}(F(X), F(Y))
\end{equation*}
induced by the functor $F$ (when no confusion arises, we often write $\widetilde{F}_{X,Y}(e)$ simply as $F(e)$). Here, we exhibit properties of the map $\widetilde{F}_{X,Y}$.
All of the results should be well-known when the functor in question is exact. Since the exactness is too strong for our intended applications, we revisit the standard argument and explicitly demonstrate that the results continue to hold under weaker assumptions.

First, we discuss the naturality of $\widetilde{F}_{X,Y}$.

\begin{lemma}
  \label{lem:Ext1-and-functor-1}
    Let $F$ be as above.
    \begin{enumerate}
    \item For fixed $X \in \mathcal{A}$, the map $\widetilde{F}_{X,Y}$ is natural in $Y \in \mathcal{A}$ such that $(X, Y) \in \mathcal{E}_F$.
    \item For fixed $Y \in \mathcal{A}$, the map $\widetilde{F}_{X,Y}$ is natural in $X \in \mathcal{A}$ such that $(X, Y) \in \mathcal{E}_F$.
    \end{enumerate}
\end{lemma}
\begin{proof}
We only present the proof of (a), since the proof of (b) is analogous.
    We fix $X \in \mathcal{A}$ and let $g : Y \to Y'$ be a morphism in $\mathcal{A}$ such that both $(X, Y)$ and $(X, Y')$ belong to the class $\mathcal{E}_F$. Let $e \in \Ext^1(X, Y)$ be an element represented by an exact sequence $0 \to Y \to E \to X \to 0$. By applying $F$ to the diagram \eqref{eq:Ext1-pushout}, we obtain a commutative diagram
    \begin{equation*}
        \begin{tikzcd}
            0 \arrow[r]
          & F(Y) \arrow[r] \arrow[d, "{F(g)}"']
          & F(E) \arrow[r] \arrow[d]
          & F(X) \arrow[r] \arrow[d, equal] & 0 \\
          0 \arrow[r] & F(Y') \arrow[r]
          & F(E \amalg_Y Y') \arrow[r] & F(X) \arrow[r] & 0
        \end{tikzcd}
    \end{equation*}
    whose rows are exact since $(X, Y), (X, Y') \in \mathcal{E}_F$. The second row represents $F(g_{\star}(e))$. Since the first row represents $F(e)$, the second row also represents $F(g)_{\star} (F(e))$ by Lemma \ref{lem:Ext1-pullback-pushout}. Therefore we have $F(g_{\star}(e)) = F(g)_{\star}(F(e))$. This means that $\widetilde{F}_{X,Y}$ is natural as stated.
\end{proof}

The additivity of $\widetilde{F}_{X,Y}$ appears to require a somewhat subtle condition.

\begin{lemma}
  \label{lem:Ext1-and-functor-2}
Let $F$ be as above, and let $X, Y \in \mathcal{A}$. If both $(X, Y)$ and $(X, Y \oplus Y)$ belong to the class $\mathcal{E}_F$, then the map $\widetilde{F}_{X,Y}$ is additive.
\end{lemma}
\begin{proof}
Let $e_i$ ($i = 1, 2$) be an element of $\Ext^1(X, Y)$ represented by an exact sequence $0 \to Y \to E_i \to X \to 0$. We consider the following commutative diagram with exact rows:
\begin{equation*}
    \begin{tikzcd}
        0 \arrow[r] & Y \oplus Y \arrow[r] & E_1 \oplus E_2 \arrow[r] & X \oplus X \arrow[r] & 0 \\
        0 \arrow[r] & Y \oplus Y \arrow[r] \arrow[u, equal] \arrow[d, "{\mathrm{sum}_Y}"']
        & \bullet \arrow[r] \arrow[u] \arrow[d] \arrow[ru, "{\text{(PB)}}", phantom]
        & X \arrow[r] \arrow[u, "{\mathrm{diag}_X}"'] & 0 \\
        0 \arrow[r] & Y \arrow[r]
        & \bullet \arrow[r] \arrow[lu, phantom, "\text{(PO)}"]
        & X \arrow[u, equal] \arrow[r] & 0
    \end{tikzcd}
\end{equation*}
We call this diagram $D$ and consider the diagram $F(D)$ obtained from $D$ by applying $F$.
Since $e_1 + e_2$ is represented by the third row of $D$, the third row of $F(D)$ represents $F(e_1 + e_2)$.
Now we compute the third row of $F(D)$ in a different way. Since $F$ is additive, we may identify $F(V \oplus W) = F(V) \oplus F(W)$ for $V, W \in \mathcal{A}$. Under this identification, we have $F(\mathrm{diag}_X) = \mathrm{diag}_{F(X)}$ and $F(\mathrm{sum}_Y) = \mathrm{sum}_{F(Y)}$. It is straightforward to verify that the first row of $F(D)$ is exact and represents $F(e_1) \oplus F(e_2)$. By the assumption that $(X, Y), (X, Y \oplus Y) \in \mathcal{E}_F$, the second and the third rows of $F(D)$ are also exact. Thus, by using Lemma \ref{lem:Ext1-pullback-pushout}, we see that the third row of $F(D)$ represents $F(e_1) + F(e_2)$. The proof is done.
\end{proof}

This lemma provides an application that is directly related to our objective:

\begin{lemma}
\label{lem:Ext1-and-functor-2-corollary}
Let $\mathcal{C}$ be an abelian monoidal category, and let $X, Y, A \in \mathcal{C}$. If $X$ is right flat and $A$ is a quotient of a left flat object, then the functor $A \otimes (-)$ induces a homomorphism \[ \Ext^1(X, Y) \to \Ext^1(A \otimes X, A \otimes Y) \] of abelian groups.
\end{lemma}
\begin{proof}
   We consider the functor $F = A \otimes (-)$. Lemma \ref{lem:flatness-criterion-3} implies that $F$ preserves extensions of $X$ by arbitrary object. Thus the functor $F$ induces a map between $\Ext^1$-groups as stated, and this map is in fact additive by Lemma \ref{lem:Ext1-and-functor-2}.
\end{proof}

\subsubsection{The adjunction isomorphism}
As before, $F : \mathcal{A} \to \mathcal{B}$ is an additive functor between abelian categories.
We assume that $F$ has a right adjoint $G: \mathcal{B} \to \mathcal{A}$. Let $\eta : \id_{\mathcal{A}} \to G F$ and $\varepsilon : F G \to \id_{\mathcal{B}}$ be the unit and the counit of the adjunction, respectively. Under a suitable condition for $X \in \mathcal{A}$ and $Y \in \mathcal{B}$, the following lemma gives an isomorphism $\Ext^1_{\mathcal{A}}(X, G(Y)) \cong \Ext^1_{\mathcal{B}}(F(X), Y)$ given by a formula similar to the adjunction isomorphism.

\begin{lemma}
  \label{lem:Ext1-adjunction}
  If $(X, G(Y)) \in \mathcal{E}_F$ and $(F(X), Y) \in \mathcal{E}_G$, then the map
  \begin{equation*}
    \begin{tikzcd}[column sep = 32pt]
      \Phi : \Ext^1_{\mathcal{A}}(X, G(Y)) \arrow[r, "{\widetilde{F}}"]
      & \Ext^1_{\mathcal{B}}(F(X), F G(Y))
      \arrow[r, "{(\varepsilon_Y)_{\star}}"]
      & \Ext^1_{\mathcal{B}}(F(X), Y)
    \end{tikzcd}
  \end{equation*}
  is invertible with the inverse
  \begin{equation*}
    \begin{tikzcd}[column sep = 32pt]
      \Psi : \Ext^1_{\mathcal{B}}(F(X), Y) \arrow[r, "{\widetilde{G}}"]
      & \Ext^1_{\mathcal{A}}(G F(X), G(Y))
      \arrow[r, "{(\eta_X)^{\star}}"]
      & \Ext^1_{\mathcal{A}}(X, G(Y)).
    \end{tikzcd}
  \end{equation*}
\end{lemma}

\begin{proof}
We show that $\Psi\circ \Phi=\id$; the proof that $\Phi\circ\Psi=\id$ is analogous.
Let $e\in \Ext^1_{\mathcal{A}}(X,G(Y))$ be an element represented by a short exact sequence
$0\to G(Y)\xrightarrow{a} E\xrightarrow{b} X\to 0$. By definition, $\Phi(e)$ is represented by the second row of the following commutative diagram with exact rows:
\begin{equation*}
\begin{tikzcd}
    0 \arrow[r] & FG(Y)\arrow[r,"{F(a)}"]  \arrow[rd, "\text{(PO)}", phantom] \arrow[d,"{\varepsilon_Y}"']
    & F(E)\arrow[d,"{j}"] \arrow[r, "{F(b)}"]
    & F(X) \arrow[r] \arrow[d, equal] & 0 \\
    0 \arrow[r] & Y\arrow[r,"{a'}"'] & E' \arrow[r, "{b'}"'] & F(X) \arrow[r] & 0
\end{tikzcd}
\end{equation*}
$\Psi(\Phi(e))$ is represented by the first row of the following commutative diagram with exact rows:
\begin{equation*}
\begin{tikzcd}
0 \arrow[r] & G(Y) \arrow[d, equal] \arrow[r, "{a''}"]
& E''  \arrow[rd, "\text{(PB)}", phantom]
\arrow[r,"{v}"]\arrow[d,"{u}"'] & X \arrow[d,"{\eta_X}"] \arrow[r] & 0\\
0 \arrow[r] & G(Y) \arrow[r, "{G(a')}"] & G(E') \arrow[r,"{G(b')}"'] & GF(X) \arrow[r] & 0
\end{tikzcd}
\end{equation*}
Consider the composite $t:=G(j)\circ \eta_E:E\to G(E')$.
By naturality of $\eta$, we have
\[
G(b')\circ t
=G(b')\circ G(j)\circ \eta_E
=GF(b)\circ \eta_E
=\eta_X\circ b.
\]
Hence, by the universal property of the pullback, there is a unique morphism $s:E\to E''$ such that $u\circ s=t$ and $v\circ s=b$.
Using naturality of $\eta$, we get
\begin{gather*}
u\circ s\circ a
=t\circ a
=G(j)\circ \eta_E\circ a
=G(j)\circ GF(a)\circ \eta_{G(Y)} \\
=G(a')\circ G(\varepsilon_Y)\circ \eta_{G(Y)}
=G(a') = u \circ a'',
\end{gather*}
where the last equality is the triangle identity for the unit and the counit.
Since also $v\circ s\circ a=b\circ a=0 = v \circ a''$, we have $s \circ a = a''$. Thus the map $s$ is a morphism of short exact sequences
\[
\begin{tikzcd}
0\arrow[r] & G(Y)\arrow[r,"a"]\arrow[d,equal] & E\arrow[r,"b"]\arrow[d,"s"] & X\arrow[r]\arrow[d,equal] & 0\\
0\arrow[r] & G(Y)\arrow[r, "{a''}"] & E''\arrow[r,"v"] & X\arrow[r] & 0.
\end{tikzcd}
\]
By the five lemma, $s$ is an isomorphism, hence $e$ and $\Psi(\Phi(e))$ are equivalent extensions.
Thus they define the same class in $\Ext^1_{\mathcal{A}}(X,G(Y))$.
\end{proof}


\subsection{Exercise in homological algebra}

Let $\mathcal{A}$ be an abelian category. Throughout this subsection, we consider the commutative diagram
\begin{equation}
  \label{eq:exercise-3x3-diagram}
  \begin{tikzcd}[column sep = 32pt]
    A_{11}   \arrow[d, hook', "{j_1}"'] \arrow[r, hook, "{i_1}"]
    & A_{12} \arrow[d, hook', "{j_2}"'] \arrow[r, two heads, "{p_1}"]
    & A_{13} \arrow[d, hook', "{j_3}"'] \\
    A_{21}   \arrow[d, two heads, "{q_1}"'] \arrow[r, hook, "{i_2}"]
    & A_{22} \arrow[d, two heads, "{q_2}"'] \arrow[r, two heads, "{p_2}"]
    & A_{23} \arrow[d, two heads, "{q_3}"'] \\
    A_{31}   \arrow[r, hook, "{i_3}"]
    & A_{32} \arrow[r, two heads, "{p_3}"]
    & A_{33}
  \end{tikzcd}
\end{equation}
in $\mathcal{A}$ with exact rows and exact columns. Here, we establish technical lemmas concerning this diagram.

\begin{lemma}
  \label{lem:exercise-lemma-1}
  There is an exact sequence
  \begin{equation}
    \label{eq:exercise-ex-seq-1}
    \begin{tikzcd}
      0 \arrow[r]
      & A_{11} \arrow[r, "{\iota}"]
      & \Ker(p_3 q_2) \arrow[r, "{\pi}"]
      & A_{13} \oplus A_{31} \arrow[r] & 0,
    \end{tikzcd}
  \end{equation}
  where $\iota$ is the restriction of $i_2 j_1 : A_{11} \to A_{22}$ and $\pi$ is the unique morphism making the following diagram commutative:
  \begin{equation*}
    \begin{tikzcd}
      A_{22} \arrow[d, two heads, "{p_2}"']
      & & \Ker(p_3q_2) \arrow[d, "{\pi}"]
      \arrow[ll, hook'] \arrow[rr, hook]
      & & A_{22} \arrow[d, two heads, "{q_2}"] \\
      A_{23} & A_{13} \arrow[l, hook', "{j_3}"]
      & A_{13} \oplus A_{31}
      \arrow[l, two heads] \arrow[r, two heads]
      & A_{31} \arrow[r, hook, "{i_3}"'] & A_{32}
    \end{tikzcd}
  \end{equation*}
\end{lemma}
\begin{proof}
We use Mitchell's embedding theorem to treat objects of $\mathcal{A}$ as if they were abelian groups. We first define $\pi$.
Since the diagram \eqref{eq:exercise-3x3-diagram} commutes, we have
\[
q_3 p_2 = p_3 q_2 : A_{22}\to A_{33}.
\]
Thus, for $z\in \Ker(p_3 q_2)$ we have $q_3 p_2(z)=0$, so $p_2(z)\in \Ker(q_3)=\Img(j_3)$.
Since $j_3$ is monic, there is a unique morphism $\pi_{13}:\Ker(p_3 q_2)\to A_{13}$ such that
$j_3\circ \pi_{13}=p_2|_{\Ker(p_3 q_2)}$.

Similarly, if $z\in \Ker(p_3 q_2)$ then $p_3 q_2(z)=0$, so $q_2(z)\in \Ker(p_3)=\Img(i_3)$.
Since $i_3$ is monic, there is a unique morphism $\pi_{31}:\Ker(p_3 q_2)\to A_{31}$ such that
$i_3\circ \pi_{31}=q_2|_{\Ker(p_3 q_2)}$.
We set $\pi=(\pi_{13},\pi_{31}):\Ker(p_3q_2)\to A_{13}\oplus A_{31}$.

We prove that the sequence \eqref{eq:exercise-ex-seq-1} is exact.

Injectivity of $\iota$ is clear since $i_2j_1$ is monic.
Moreover, $\pi\circ \iota=0$ because $p_2 i_2=0$ and $q_2 i_2=i_3 q_1$.
Thus $\Img(\iota)\subseteq \Ker(\pi)$.

To show that $\pi$ is epic, let $(x,y)\in A_{13}\oplus A_{31}$.
Choose $z_0\in A_{22}$ with $q_2(z_0)=i_3(y)$ (possible since $q_2$ is epic).
Then $p_3 q_2(z_0)=0$, hence $z_0\in\Ker(p_3 q_2)$.
Since $q_3 p_2(z_0)=p_3 q_2(z_0)=0$, we have $p_2(z_0)\in \Ker(q_3)=\Img(j_3)$, so there is
$x_0\in A_{13}$ with $p_2(z_0)=j_3(x_0)$.
Let $d:=x_0-x\in A_{13}$.
Choose $w\in A_{12}$ with $p_1(w)=d$ (possible since $p_1$ is epic).
Then
\[
q_2(z_0-j_2(w))=q_2(z_0)
\quad\text{and}\quad
p_2(z_0-j_2(w))=p_2(z_0)-p_2j_2(w)=j_3(x_0)-j_3p_1(w)=j_3(x).
\]
Thus $z:=z_0-j_2(w)$ lies in $\Ker(p_3q_2)$ and satisfies $\pi(z)=(x,y)$.
So $\pi$ is epic.

Finally, let $z\in \Ker(\pi)$.
Then $p_2(z)=j_3\pi_{13}(z)=0$ and $q_2(z)=i_3\pi_{31}(z)=0$.
Hence $z\in \Ker(p_2)\cap \Ker(q_2)=\Img(i_2)\cap \Img(j_2)$.
Choose $u\in A_{21}$ and $v\in A_{12}$ with $i_2(u)=j_2(v)=z$.
Applying $q_2$ and using commutativity gives
\[
0=q_2(z)=q_2 i_2(u)=i_3 q_1(u),
\]
so $q_1(u)=0$ since $i_3$ is monic.
Thus $u\in \Ker(q_1)=\Img(j_1)$, so $u=j_1(a)$ for some $a\in A_{11}$.
Then
\[
z=i_2(u)=i_2j_1(a)=j_2 i_1(a),
\]
and since $j_2$ is monic, we get $v=i_1(a)$.
Therefore $z=\iota(a)$, i.e.\ $\Ker(\pi)=\Img(\iota)$.
\end{proof}

By the additivity of $\Ext^1$, we have an isomorphism
\begin{equation}
  \label{eq:exercise-ext-iso-1}
  \Ext^1(A_{13} \oplus A_{31}, A_{11})
  \cong \Ext^1(A_{13}, A_{11}) \oplus \Ext^1(A_{31}, A_{11})
\end{equation}
of abelian groups. The exact sequence \eqref{eq:exercise-ex-seq-1} defines an element of the left hand side of \eqref{eq:exercise-ext-iso-1}. The element of the right hand side of~\eqref{eq:exercise-ext-iso-1} corresponding to \eqref{eq:exercise-ex-seq-1} is described as follows:

\begin{lemma}
  \label{lem:exercise-lemma-2}
  Let $\alpha_1$ and $\beta_1$ be the elements of the Ext group corresponding to the first row and the first column of \eqref{eq:exercise-3x3-diagram}, respectively. Then the element represented by the exact sequence \eqref{eq:exercise-ex-seq-1} corresponds to the element $(\alpha_1, \beta_1)$ of the right hand side of \eqref{eq:exercise-ext-iso-1}.
\end{lemma}
\begin{proof}
Under the canonical isomorphism
$\Ext^1(A_{13}\oplus A_{31},A_{11})\cong \Ext^1(A_{13},A_{11})\oplus \Ext^1(A_{31},A_{11})$,
the extension class is sent to the pair of its pullbacks along the inclusions
$A_{13}\hookrightarrow A_{13}\oplus A_{31}$ and $A_{31}\hookrightarrow A_{13}\oplus A_{31}$.

Let $e$ be the extension class represented by \eqref{eq:exercise-ex-seq-1}.
The pullback of \eqref{eq:exercise-ex-seq-1} along $A_{13}\hookrightarrow A_{13}\oplus A_{31}$ has middle term
$\Ker(\pi_{31})$, where $\pi_{31}$ is the composite $\Ker(p_3 q_2)\xrightarrow{\pi} A_{13}\oplus A_{31}\to A_{31}$.
By construction of $\pi$ in Lemma~\ref{lem:exercise-lemma-1}, we have
$i_3\circ \pi_{31}=q_2|_{\Ker(p_3 q_2)}$, hence $\Ker(\pi_{31})=\Ker(q_2)\cong A_{12}$.
Moreover, the induced map $\Ker(\pi_{31})\to A_{13}$ is identified with $p_1$ by the commutativity
$p_2 j_2=j_3 p_1$ and the defining property $j_3\circ \pi_{13}=p_2|_{\Ker(p_3 q_2)}$.
Thus this pullback extension is isomorphic to the first row of \eqref{eq:exercise-3x3-diagram}, hence equals
$\alpha_1$.

Similarly, the pullback along $A_{31}\hookrightarrow A_{13}\oplus A_{31}$ has middle term $\Ker(\pi_{13})$.
Since $j_3\circ \pi_{13}=p_2|_{\Ker(p_3 q_2)}$, we have $\Ker(\pi_{13})=\Ker(p_2)\cong A_{21}$.
The induced map $\Ker(\pi_{13})\to A_{31}$ is identified with $q_1$ using commutativity
$q_2 i_2=i_3 q_1$ and the defining property $i_3\circ \pi_{31}=q_2|_{\Ker(p_3 q_2)}$.
Hence this pullback extension is isomorphic to the first column of \eqref{eq:exercise-3x3-diagram}, so it equals
$\beta_1$.

Therefore $e$ corresponds to $(\alpha_1,\beta_1)$.
\end{proof}

In the opposite category $\mathcal{A}^{\op}$, the diagram \eqref{eq:exercise-3x3-diagram} looks like:
\begin{equation*}
  \begin{tikzcd}[column sep = 32pt]
    A_{33}   \arrow[d, hook', "{p_3}"'] \arrow[r, hook, "{q_3}"]
    & A_{32} \arrow[d, hook', "{p_2}"'] \arrow[r, two heads, "{i_3}"]
    & A_{31} \arrow[d, hook', "{p_1}"'] \\
    A_{23}   \arrow[d, two heads, "{j_3}"'] \arrow[r, hook, "{q_2}"]
    & A_{22} \arrow[d, two heads, "{j_2}"'] \arrow[r, two heads, "{i_2}"]
    & A_{21} \arrow[d, two heads, "{j_1}"'] \\
    A_{13}   \arrow[r, hook, "{q_1}"]
    & A_{12} \arrow[r, two heads, "{i_1}"]
    & A_{11}.
  \end{tikzcd}
\end{equation*}
We apply the above lemmas to this diagram and then interpret the result in $\mathcal{A}$ to obtain:

\begin{lemma}
  \label{lem:exercise-lemma-3}
  There is an exact sequence
  \begin{equation}
    \label{eq:exercise-ex-seq-2}
    \begin{tikzcd}
      0 \arrow[r]
      & A_{13} \oplus A_{31} \arrow[r, "{\iota}"]
      & \Cok(i_2 j_1) \arrow[r, "{\pi}"]
      & A_{33} \arrow[r] & 0,
    \end{tikzcd}
  \end{equation}
  where $\pi$ is induced by $p_3 q_2 : A_{22} \to A_{33}$ and $\iota$ is the unique morphism making the following diagram commutative:
  \begin{equation*}
    \begin{tikzcd}
      A_{12} \arrow[r, two heads, "{p_1}"] \arrow[d, hook', "{j_2}"']
      & A_{13} \arrow[r, hook]
      & A_{13} \oplus A_{31} \ar[d, "{\iota}"]
      & A_{31} \arrow[l, hook']
      & A_{21} \arrow[l, "{q_1}"', two heads] \arrow[d, hook', "{i_2}"] \\
      A_{22} \arrow[rr, two heads]
      & & \Cok(i_2 j_1)
      & & A_{22} \arrow[ll, two heads]
    \end{tikzcd}
  \end{equation*}
  Under the canonical isomorphism
  \begin{equation}
    \label{eq:exercise-ext-iso-2}
    \Ext^1(A_{33}, A_{13} \oplus A_{31})
    \cong \Ext^1(A_{33}, A_{13}) \oplus \Ext^1(A_{33}, A_{31})
  \end{equation}
  of abelian groups, the element of the left hand side represented by \eqref{eq:exercise-ex-seq-2} corresponds to the element $(\alpha_3, \beta_3)$ of the right hand side, where $\alpha_3$ and $\beta_3$ are elements represented by the third column and the third row of \eqref{eq:exercise-3x3-diagram}, respectively.
\end{lemma}


For the rest of this section, we fix a short exact sequence in an abelian monoidal category $\C$ with enough
left flat and enough right flat objects:
\begin{equation}
  \label{eq:EP-lemma-proof-exact-seq-1}
  \begin{tikzcd}
    0 \arrow[r] & Y \arrow[r, "{i}"] & Z \arrow[r, "{p}"] & X \arrow[r] & 0.
  \end{tikzcd}
\end{equation}

\subsection{Proof of Lemma \ref{lem:Etingof-Penneys} (1)}

We give a proof of Lemma \ref{lem:Etingof-Penneys} (1).
Assume that $X$ and $Y$ are right flat and left rigid.
Lemma \ref{lem:rigid-implies-flat} implies that $X$ and $Y$ are also left flat.

\subsubsection{Construction of the dual object}
We first consider the following diagram:
\begin{equation*}
\begin{tikzcd}[column sep = 16pt]
  \Ext^1(X, Y) \arrow[r, "{\widetilde{- \otimes X^*}}"]
  \arrow[d, "{\widetilde{Y^* \otimes -}}"]
  & \Ext^1(X \otimes X^*, Y \otimes X^*)
  \arrow[r, "{(\coev_X)^{\star}}"]
  \arrow[d, "{\widetilde{Y^* \otimes -}}"]
  & \Ext^1(\unit, Y \otimes X^*)
  \arrow[d, "{\widetilde{Y^* \otimes -}}"] \\
  \Ext^1(Y^* \otimes X, Y^* \otimes Y)
  \arrow[r, "{\widetilde{- \otimes X^*}}" {yshift = 5pt}]
  \arrow[d, "{(\ev_Y)_{\star}}"]
  & \Ext^1(Y^* \otimes X \otimes X^*, Y^* \otimes Y \otimes X^*)
  \arrow[r, "{(\id \otimes \coev_X)^{\star}}" {yshift = 5pt}]
  \arrow[d, "{(\ev_Y \otimes \id)_{\star}}"]
  & \Ext^1(Y^*, Y^* \otimes Y \otimes X^*)
  \arrow[d, "{(\ev_Y \otimes \id)_{\star}}"] \\
  \Ext^1(Y^* \otimes X, \unit)
  \arrow[r, "{\widetilde{- \otimes X^*}}"]
  & \Ext^1(Y^* \otimes X \otimes X^*, X^*)
  \arrow[r, "{(\id \otimes \coev_X)^{\star}}"]
  & \Ext^1(Y^*, X^*),
\end{tikzcd}
\end{equation*}
where $\widetilde{- \otimes X^*}$ and $\widetilde{Y^* \otimes -}$ are the maps induced by the functors $- \otimes X^*$ and $Y^* \otimes -$, respectively.
Since $X^*$ is right flat by Lemma \ref{lem:rigid-implies-flat}, the maps $\widetilde{- \otimes X^*}$ in the diagram are well-defined homomorphisms of abelian groups.
The objects $X$, $X \otimes X^*$ and $\unit$ are right flat.
Thus, by Lemma \ref{lem:Ext1-and-functor-2-corollary}, the maps $\widetilde{Y^* \otimes -}$ in the diagram are also well-defined homomorphisms of abelian groups.

The top-left square commutes since the functors $- \otimes X^*$ and $Y^* \otimes -$ commute. By the basic property of the Yoneda Ext, the bottom-right square also commutes. The remaining squares also commute by Lemma \ref{lem:Ext1-and-functor-1}. The rows and the columns of the above diagram are bijections induced by adjunctions discussed in Lemma~\ref{lem:Ext1-adjunction}. Hence the above diagram shrinks to the following commutative diagram of isomorphisms of abelian groups:
\begin{equation}
  \label{eq:EP-lemma-proof-Ext1-adj}
  \begin{tikzcd}
    \Ext^1(X, Y) \arrow[r, "{\cong}"] \arrow[d, "{\cong}"'] 
    & \Ext^1(\unit, Y \otimes X^*) \arrow[d, "{\cong}"] \\
    \Ext^1(Y^* \otimes X, \unit) \arrow[r, "{\cong}"] 
    & \Ext^1(Y^*, X^*)
  \end{tikzcd}
\end{equation}

We denote by $\alpha \in \Ext^1(X, Y)$ the element represented by \eqref{eq:EP-lemma-proof-exact-seq-1}, and let $\beta \in \Ext^1(Y^*, X^*)$ be the element corresponding to $-\alpha$. We choose an exact sequence
\begin{equation}
  \label{eq:EP-lemma-proof-exact-seq-2}
  \begin{tikzcd}
    0 \arrow[r] & X^* \arrow[r, "{j}"] & Z^{\vee} \arrow[r, "{q}"] & Y^* \arrow[r] & 0
  \end{tikzcd}
\end{equation}
representing $\beta$. We will show that the middle term, $Z^{\vee}$, is a left dual object of $Z$.

\subsubsection{Construction of the coevaluation}

By taking the tensor product of exact sequences \eqref{eq:EP-lemma-proof-exact-seq-1} and \eqref{eq:EP-lemma-proof-exact-seq-2}, we obtain the following commutative diagram:
\begin{equation}
  \label{eq:EP-lemma-proof-diagram-1}
  \begin{tikzcd}[row sep = 16pt, column sep = 48pt]
    Y \otimes X^*
    \arrow[r, hook, "{i \otimes \id}"]
    \arrow[d, hook, "{\id \otimes j}"']
    & Z \otimes X^*
    \arrow[r, two heads, "{p \otimes \id}"]
    \arrow[d, hook, "{\id \otimes j}"']
    & X \otimes X^*
    \arrow[d, hook, "{\id \otimes j}"'] \\
    Y \otimes Z^{\vee}
    \arrow[r, hook, "{i \otimes \id}"]
    \arrow[d, two heads, "{\id \otimes q}"']
    & Z \otimes Z^{\vee}
    \arrow[r, two heads, "{p \otimes \id}"]
    \arrow[d, two heads, "{\id \otimes q}"']
    & X \otimes Z^{\vee}
    \arrow[d, two heads, "{\id \otimes q}"'] \\
    Y \otimes Y^*
    \arrow[r, hook, "{i \otimes \id}"]
    & Z \otimes Y^*
    \arrow[r, two heads, "{p \otimes \id}"]
    & X \otimes Y^* 
  \end{tikzcd}
\end{equation}
Since $X$ and $Y$ are left flat, $Z$ is also left flat by Lemma \ref{lem:flatness-criterion-2}. Thus all columns of this diagram are exact. Since $X^*$ and $Y^*$ are right flat, $Z^{\vee}$ is also right flat by Lemma \ref{lem:flatness-criterion-2}, and thus all rows of this diagram are exact.
Hence, by Lemma~\ref{lem:exercise-lemma-1}, we obtain an exact sequence
\begin{equation}
  \label{eq:EP-lemma-proof-exact-seq-3}
  \begin{tikzcd}[column sep = 20pt]
    0 \arrow[r]
    & X \otimes Y^* \arrow[r, "{i \otimes j}"]
    & \Ker(p \otimes q) \arrow[r, "{\pi}"]
    & (X \otimes X^*) \oplus (Y \otimes Y^*)
    \arrow[r] & 0,
  \end{tikzcd}
\end{equation}
where $\pi$ is a morphism in $\C$ such that the following diagram is commutative:
\begin{equation}
  \label{eq:EP-lemma-proof-diagram-2}
  \begin{tikzcd}[column sep = 16pt]
    Z \otimes Z^{\vee} \arrow[d, two heads, "{p \otimes \id}"']
    & & \Ker(p \otimes q) \arrow[d, "{\pi}"]
    \arrow[ll, hook'] \arrow[rr, hook]
    & & Z \otimes Z^{\vee} \arrow[d, two heads, "{\id \otimes q}"] \\
    X \otimes Z^{\vee}
    & X \otimes X^* \arrow[l, hook', "{\id \otimes j}"' {yshift=3pt}]
    & (X \otimes X^*) \oplus (Y \otimes Y^*)
    \arrow[l, two heads, "{\mathrm{pr}_X}"' {yshift=3pt}]
    \arrow[r, two heads, "{\mathrm{pr}_Y}" {yshift=3pt}]
    & Y \otimes Y^* \arrow[r, hook, "{i \otimes \id}" {yshift=3pt}]
    & Z \otimes Y^*,
  \end{tikzcd}
\end{equation}
where $\mathrm{pr}_X$ and $\mathrm{pr}_Y$ are respective projections.

We compute the pullback of the exact sequence \eqref{eq:EP-lemma-proof-exact-seq-3} along the morphism $\coev_{X, Y} := (\coev_X \oplus \coev_Y) \circ \mathrm{diag}_{\unit}$.
By Lemma \ref{lem:Ext1-Baer-sum} and the naturality of the Baer sum, we have the following commutative diagram:
\begin{equation*}
  \begin{tikzcd}
      \Ext^1(X \otimes X^*, X \otimes Y^*) \oplus \Ext^1(Y \otimes Y^*, X \otimes Y^*)
    \arrow[r, "{\cong}"]
    \arrow[d, "{(\coev_X)^{\star} \oplus (\coev_Y)^{\star}}"']
    & \Ext^1((X \otimes X^*) \oplus (Y \otimes Y^*), X \otimes Y^*)
    \arrow[d, "{(\coev_{X,Y})^{\star}}"] \\
      \Ext^1(\unit, X \otimes Y^*) \oplus \Ext^1(\unit, X \otimes Y^*)
    \arrow[r, "\text{the Baer sum}"]
    & \Ext^1(\unit, X \otimes Y^*)
  \end{tikzcd}
\end{equation*}
By Lemma~\ref{lem:exercise-lemma-2}, the element \eqref{eq:EP-lemma-proof-exact-seq-3} of the upper right corner corresponds to the element $(\alpha \otimes X^*, Y \otimes \beta)$ of the upper left corner. Thus, by the commutativity of the diagram, we have
\begin{align*}
  (\coev_{X,Y})^{\star} \eqref{eq:EP-lemma-proof-exact-seq-3}
  & = (\coev_X)^{\star}(\alpha \otimes X^*) + (\coev_Y)^{\star}(Y \otimes \beta)
\end{align*}
in $\Ext^1(\unit, X \otimes Y^*)$
The first and the second term of the right hand side correspond to $\alpha \in \Ext^1(X, Y)$ and $\beta \in \Ext^1(Y^*, X^*)$ in the diagram \eqref{eq:EP-lemma-proof-Ext1-adj}, respectively. Since $\beta$ corresponds to $-\alpha \in \Ext^1(X, Y)$ in the diagram \eqref{eq:EP-lemma-proof-Ext1-adj} by definition, the right hand side of the above equation is in fact zero.
This means that the pullback of \eqref{eq:EP-lemma-proof-exact-seq-3} along $\coev_{X,Y}$ splits. Summarizing our discussion so far, we have the commutative diagram
\begin{equation}
  \label{eq:EP-lemma-proof-diagram-3}
  \begin{tikzcd}[column sep = 20pt]
    0 \arrow[r]
    & X \otimes Y^* \arrow[r, "{i \otimes j}"]
    & \Ker(p \otimes q) \arrow[r, "{\pi}"]
    & (X \otimes X^*) \oplus (Y \otimes Y^*)
    \arrow[r] & 0 \\
    0 \arrow[r]
    & X \otimes Y^* \arrow[u, equal] \arrow[r]
    & E \arrow[r, "{r}"'] \arrow[u, "{u}"]
    \arrow[ru, phantom, "\text{(PB)}"]
    & \unit \arrow[u, "{\coev_{X,Y}}"']
    \arrow[l, dashed, bend left=30, "{s}"]
    \arrow[r] & 0
  \end{tikzcd}
\end{equation}
with exact rows, where (PB) is a pullback and $s$ is a section of $r$. We define
\begin{equation}
  \label{eq:EP-lemma-proof-def-coev}
  \eta_Z : \unit
  \xrightarrow{s} E \xrightarrow{u} \Ker(p \otimes q) \hookrightarrow Z \otimes Z^{\vee},
\end{equation}
which, in fact, becomes the coevaluation after suitable modification.
Before proceeding further, we remark that the morphism $\eta_Z$ makes the following diagram commute:
\begin{equation}
  \label{eq:EP-lemma-proof-diagram-4}
  \begin{tikzcd}[column sep = 48pt]
    Y \otimes Y^* \arrow[d, "{i \otimes \id}"']
    & \unit \arrow[l, "{\coev_Y}"']
    \arrow[r, "{\coev_X}"] \arrow[d, "{\eta_Z}"]
    & X \otimes X^* \arrow[d, "{\id \otimes j}"] \\
    Z \otimes Y^*
    & Z \otimes Z^{\vee}
    \arrow[l, "{\id \otimes q}"]
    \arrow[r, "{p \otimes \id}"']
    & X \otimes Z^{\vee}
  \end{tikzcd}
\end{equation}
Indeed, by the commutative diagrams \eqref{eq:EP-lemma-proof-diagram-2} and \eqref{eq:EP-lemma-proof-diagram-3}, we have
\begin{gather*}
  (\id_Z \otimes q) \circ \eta_Z
  = (i \otimes \id_{Y^*}) \circ \mathrm{pr}_Y \circ \pi \circ u \circ s \\
  = (i \otimes \id_{Y^*}) \circ \mathrm{pr}_Y \circ \coev_{X,Y}
  = (i \otimes \id_{Y^*}) \circ \coev_Y
\end{gather*}
and, in a similar way, $(p \otimes \id_{Z^{\vee}}) \circ \eta_Z = (\id_X \otimes j) \circ \coev_X$.

\subsubsection{Construction of the evaluation}

By taking the tensor product of exact sequences \eqref{eq:EP-lemma-proof-exact-seq-1} and \eqref{eq:EP-lemma-proof-exact-seq-2} in the different order as before, we obtain the following commutative diagram:
\begin{equation}
  \label{eq:EP-lemma-proof-diagram-1-alt}
  \begin{tikzcd}[row sep = 16pt, column sep = 48pt]
    X^* \otimes Y
    \arrow[r, hook, "{\id \otimes i}"]
    \arrow[d, hook, "{j \otimes \id}"']
    & X^* \otimes Z
    \arrow[r, two heads, "{\id \otimes p}"]
    \arrow[d, hook, "{j \otimes \id}"']
    & X^* \otimes X
    \arrow[d, hook, "{j \otimes \id}"'] \\
    Z^{\vee} \otimes Y
    \arrow[r, hook, "{\id \otimes i}"]
    \arrow[d, two heads, "{q \otimes \id}"']
    & Z^{\vee} \otimes Z
    \arrow[r, two heads, "{\id \otimes p}"]
    \arrow[d, two heads, "{q \otimes \id}"']
    & Z^{\vee} \otimes X
    \arrow[d, two heads, "{q \otimes \id}"'] \\
    Y^* \otimes Y
    \arrow[r, hook, "{\id \otimes i}"]
    & Y^* \otimes Z
    \arrow[r, two heads, "{\id \otimes p}"]
    & Y^* \otimes X
  \end{tikzcd}
\end{equation}
By Lemma~\ref{lem:flatness-criterion-2} and the flatness of $X$ and $Y$, the columns of the
diagram~\eqref{eq:EP-lemma-proof-diagram-1-alt} are exact. By Lemma~\ref{lem:flatness-criterion-3} applied to
$A = X^*, Z^{\vee}, Y^*$, we see that the rows of \eqref{eq:EP-lemma-proof-diagram-1-alt} are exact.

We compute the pullback of the exact sequence
\begin{equation*}
  0 \to (X^* \otimes X) \oplus (Y^* \otimes Y) \to \Cok(j \otimes i) \to Y^* \otimes X \to 0
\end{equation*}
obtained by applying Lemma~\ref{lem:exercise-lemma-3} to the diagram \eqref{eq:EP-lemma-proof-diagram-1-alt} along the morphism $\ev_{X,Y} := \mathrm{sum}_{\unit} \circ (\ev_X \oplus \ev_Y)$. By an argument dual to one when we have obtained the commutative diagram \eqref{eq:EP-lemma-proof-diagram-3}, we obtain the following commutative diagram with exact rows:
\begin{equation*}
  \begin{tikzcd}[column sep = 20pt]
    0 \arrow[r]
    & (X^* \otimes X) \oplus (Y^* \otimes Y) \arrow[r, "{\iota}"]
    \arrow[d, "{\ev_{X,Y}}"']
    & \Cok(j \otimes i) \arrow[r]
    \arrow[d, "{u'}"] \arrow[ld, phantom, "\text{(PO)}"]
    & Y^* \otimes X \arrow[r] \arrow[d, equal]
    & 0 \\
    0 \arrow[r]
    & \unit \arrow[r, "{r'}"']
    & E' \arrow[r] \arrow[l, bend left=30, dashed, "{s'}"]
    & Y^* \otimes X \arrow[r] & 0 \\
  \end{tikzcd}
\end{equation*}
Here, $s'$ is a section of $r'$. Now we set
\begin{equation}
  \label{eq:EP-lemma-proof-def-ev}
  \varepsilon_Z : Z^{\vee} \otimes Z
  \twoheadrightarrow \Cok(j \otimes i)
  \xrightarrow{u'} E' \xrightarrow{s'} \unit.
\end{equation}
Again by the dual argument, we see that the following diagram is commutative:
\begin{equation}
  \label{eq:EP-lemma-proof-diagram-5}
  \begin{tikzcd}[column sep = 48pt]
    Z^{\vee} \otimes Y
    \arrow[r, "{\id \otimes i}"]
    \arrow[d, "{q \otimes \id}"']
    & Z^{\vee} \otimes Z
    \arrow[d, "{\varepsilon_Z}"]
    & X^* \otimes Z
    \arrow[l, "{j \otimes \id}"']
    \arrow[d, "{\id \otimes p}"] \\
    Y^* \otimes Y \arrow[r, "{\ev_Y}"]
    & \unit
    & X^* \otimes X \arrow[l, "{\ev_X}"']
  \end{tikzcd}
\end{equation}

\subsubsection{Verifying the zig-zag relation}

We have defined the morphisms $\eta_Z : \unit \to Z \otimes Z^{\vee}$ and $\varepsilon_Z : Z^{\vee} \otimes Z \to \unit$ by \eqref{eq:EP-lemma-proof-def-coev} and \eqref{eq:EP-lemma-proof-def-ev}, respectively. We set
\begin{equation*}
  \xi = (\varepsilon_Z \otimes \id_{Z^{\vee}}) (\id_{Z^{\vee}} \otimes \eta_Z)
  \quad \text{and} \quad
  \zeta = (\id_Z \otimes \varepsilon_Z) (\eta_Z \otimes \id_Z),
\end{equation*}
and aim to show that $\xi$ and $\zeta$ are invertible. To show that $\xi$ is invertible, we consider the diagram of Figure~\ref{fig:EP-lemma-proof-of-zig-zag}. The cells in the diagram labeled `(nat.)' is commutative by the naturality of the monoidal product. Thus the diagram is commutative and it shrinks to the following commutative diagram with exact rows:
\begin{equation*}
  \begin{tikzcd}
    0 \arrow[r]
    & X^* \arrow[r, hook, "{j}"] \arrow[d, equal]
    & Z^{\vee} \arrow[r, two heads, "{q}"] \arrow[d, "{\xi}"]
    & Y \arrow[r] \arrow[d, equal]
    & 0 \\
    0 \arrow[r]
    & X^* \arrow[r, hook, "{j}"]
    & Z^{\vee} \arrow[r, two heads, "{q}"]
    & Y \arrow[r]
    & 0
  \end{tikzcd}
\end{equation*}
By the five lemma, $\xi$ is invertible. A similar argument shows that $\zeta$ is also invertible (the detail of this part is written in \cite{etingof2024rigidity}). Hence, by Lemma~\ref{lem:duality-modification}, $Z$ is left rigid. The proof is done.

\begin{figure}
  \begin{equation*}
    \begin{tikzcd}[row sep = 24pt, column sep = 24pt]
      X^* \arrow[dd, "{j}"']
      \arrow[rr, "{\id \otimes \coev_X}"]
      \arrow[rd, "{\id \otimes \eta_Z}"]
      & & X^* \otimes X \otimes X^*
      \arrow[ld, phantom, "\text{\eqref{eq:EP-lemma-proof-diagram-4}}"]
      \arrow[rr, "{\ev_X \otimes \id}"]
      \arrow[rd, "{\id \otimes \id \otimes j}"]
      & & X^* \arrow[dd, "{j}"]
      \arrow[ld, phantom, "\text{(nat.)}"]
      \\ 
      & X^* \otimes Z \otimes Z^{\vee}
      \arrow[rr, "{\id \otimes p \otimes \id}"]
      \arrow[rd, "{j \otimes \id \otimes \id}"]
      \arrow[ld, phantom, "\text{(nat.)}"]
      & & X^* \otimes X \otimes Z^{\vee}
      \arrow[ld, phantom, "\text{\eqref{eq:EP-lemma-proof-diagram-5}}"]
      \arrow[rd, "{\ev_X \otimes \id}"]
      \\ 
      Z^{\vee}
      \arrow[rr, "{\id \otimes \eta_Z}"]
      \arrow[rd, "{\id \otimes \coev_Y}"]
      \arrow[dd, "{q}"']
      & & Z^{\vee} \otimes Z \otimes Z^{\vee}
      \arrow[rr, "{\varepsilon_Z \otimes \id}"]
      \arrow[rd, "{\id \otimes \id \otimes q}"]
      \arrow[ld, phantom, "\text{\eqref{eq:EP-lemma-proof-diagram-4}}"]
      & & Z^{\vee}
      \arrow[dd, "{q}"]
      \\ 
      & Z^{\vee} \otimes Y \otimes Y^*
      \arrow[ld, phantom, "\text{(nat.)}"]
      \arrow[rr, "{\id \otimes i \otimes \id}"]
      \arrow[rd, "{q \otimes \id \otimes \id}" ]
      & & Z^{\vee} \otimes Z \otimes Y^*
      \arrow[rd, "{\varepsilon_Z \otimes \id}"]
      \arrow[ru, phantom, "\text{(nat.)}"]
      \arrow[ld, phantom, "\text{\eqref{eq:EP-lemma-proof-diagram-5}}"]
      \\ 
      Y^* \arrow[rr, "{\id \otimes \coev_Y}"]
      & & Y^* \otimes Y \otimes Y^*
      \arrow[rr, "{\ev_Y \otimes \id}"]
      & & Y^*
    \end{tikzcd}
  \end{equation*}
  \caption{Proof of the zig-zag equation}
  \label{fig:EP-lemma-proof-of-zig-zag}
\end{figure}


\subsection{Proof of Lemma \ref{lem:Etingof-Penneys} (2)}

We give a proof of Lemma \ref{lem:Etingof-Penneys} (2).
Assume that $X$ and $Z$ are left rigid.

We define $Y^{\vee} = \Cok(p^* : X^* \to Z^*)$ and denote by $\pi : Z^* \to Y^{\vee}$ the quotient morphism. Since $X$ is left flat, we have an exact sequence
\begin{equation*}
  \begin{tikzcd}
    0 \arrow[r]
    & Y \otimes Y^{\vee} \arrow[r, "{i \otimes \id}"]
    & Z \otimes Y^{\vee} \arrow[r, "{p \otimes \id}"]
    & X \otimes Y^{\vee} \arrow[r] & 0
  \end{tikzcd}
\end{equation*}
by Lemma~\ref{lem:flatness-criterion-3}. We set $\xi = (\id_{Z} \otimes \pi)\coev_Z$. Since $\pi p^* = 0$, we have
\begin{equation*}
  (p \otimes \id_{Y^{\vee}}) \xi
  = (p \otimes \pi)\coev_Z = (\id_{X} \otimes \pi p^*) \coev_X = 0.
\end{equation*}
Thus $\xi$ factors through $\Ker(p \otimes \id_{Y^{\vee}})$, which is equal to $\Img(i \otimes \id_{Y^{\vee}})$. Since $i \otimes \id_{Y^{\vee}}$ is monic, there exists a morphism $\eta_Y : \unit \to Y \otimes Y^{\vee}$ in $\C$ such that
\begin{equation}
  \label{eq:EP-lemma-2-proof-eq-1}
  (i \otimes \id_{Y^{\vee}}) \circ \eta_Y = (\id_{Z} \otimes \pi) \circ \coev_Z.
\end{equation}
There is also an exact sequence
\begin{equation*}
  \begin{tikzcd}
    X^* \otimes Y \arrow[r, "{p^* \otimes \id}"]
    & Z^* \otimes Y \arrow[r, "{\pi \otimes \id}"]
    & Y^{\vee} \otimes Y \arrow[r]
    & 0.
  \end{tikzcd}
\end{equation*}
We set $\zeta = \ev_Z (\id_{Z^*} \otimes i)$. Since $\zeta (p^* \otimes \id_{Y}) = \ev_X (\id_{X^*} \otimes p \circ i) = 0$, the kernel of $\zeta$ contains the image of $p^* \otimes \id_{Y}$, which is equal to the kernel of $\pi \otimes \id_{Y}$. Hence there exists a morphism $\varepsilon_Y : Y^{\vee} \otimes Y \to \unit$ in $\C$ such that
\begin{equation}
  \label{eq:EP-lemma-2-proof-eq-2}
  \varepsilon_Y \circ (\pi \otimes \id_{Y}) = \ev_Z \circ (\id_{Z^*} \otimes i).
\end{equation}
By \eqref{eq:EP-lemma-2-proof-eq-1}, \eqref{eq:EP-lemma-2-proof-eq-2} and Lemma~\ref{lem:duality-transfer-by-morphism}, we have
\begin{equation*}
  i \circ (\id_{Y} \otimes \varepsilon_Y) \circ (\eta_Y \otimes \id_{Y}) = i,
  \quad (\varepsilon_Y \otimes \id_{Y^{\vee}}) \circ (\id_{Y^{\vee}} \otimes \eta_Y) \circ \pi = \pi.
\end{equation*}
Since $i$ is monic and $\pi$ is epic, we conclude that $(Y^{\vee}, \varepsilon_Y, \eta_Y)$ is a left dual object of $Y$. The proof is done.


\subsection{Proof of Lemma \ref{lem:Etingof-Penneys} (3)}

We give a proof of Lemma \ref{lem:Etingof-Penneys} (3).
Assume that $Y$ and $Z$ are left rigid, and that $i^* : Z^* \to Y^*$ is epic in $\C$.

We define $X^{\vee} = \Ker(i^*)$ and denote by $\iota : X^{\vee} \to Z^*$ the inclusion morphism. Since $i^*$ is epic, we have an exact sequence $0 \to X^{\vee} \xrightarrow{\iota} Z^* \xrightarrow{i^*} Y^* \to 0$. We note that $Y^*$ is right flat since it has a right dual object $Y$. Thus, by Lemma \ref{lem:flatness-criterion-3}, we have the following exact sequence:
\begin{equation*}
  \begin{tikzcd}
    0 \arrow[r]
    & X \otimes X^{\vee} \arrow[r, "{\id_X \otimes \iota}"]
    & X \otimes Z^* \arrow[r, "{\id_X \otimes i^*}"]
    & X \otimes Y^* \arrow[r]
    & 0.
  \end{tikzcd}
\end{equation*}
Since $(\id_X \otimes i^*) (p \otimes \id_{Z^*}) \coev_Z = 0$, the morphism $(p \otimes \id_{Z^*}) \coev_Z$ factors through the kernel of $\id_X \otimes i^*$, which is equal to the image of $\id_X \otimes \iota$. Since $\id_X \otimes \iota$ is monic, we have a morphism $\eta_X : \unit \to X \otimes X^{\vee}$ such that
\begin{equation}
  \label{eq:EP-lemma-3-proof-eq-1}
  (\id_X \otimes \iota) \circ \eta_X = (p \otimes \id_{Z^*}) \circ \coev_Z.
\end{equation}
There is also an exact sequence $X^{\vee} \otimes Y \to X^{\vee} \otimes Z \to X^{\vee} \otimes X \to 0$. By the same argument as the proof of Part (2), we have a morphism $\varepsilon_X : X^{\vee} \otimes X \to \unit$ in $\C$ satisfying the following equation:
\begin{equation}
  \label{eq:EP-lemma-3-proof-eq-2}
  \varepsilon_X \circ (\id_{X^{\vee}} \otimes p)
  = \ev_Z \circ (\iota \otimes \id_Z).
\end{equation}
By \eqref{eq:EP-lemma-3-proof-eq-1}, \eqref{eq:EP-lemma-3-proof-eq-2} and Lemma~\ref{lem:duality-transfer-by-morphism}, we conclude that $(X^{\vee}, \varepsilon_X, \eta_X)$ is a left dual object of $X$. The proof is done.

\bibliographystyle{alpha}
\bibliography{references}

\end{document}